\documentclass[a4paper]{amsart}

\usepackage{setspace}
\usepackage[totalwidth=14cm,totalheight=22cm]{geometry}
\usepackage[T1]{fontenc}
\usepackage[dvips]{graphicx}
\usepackage{epsfig}
\usepackage{amsmath,amssymb,amscd,amsthm,nicefrac}
\usepackage{subfigure}
\usepackage[latin1]{inputenc}
\usepackage{latexsym}
\usepackage{graphics}
\usepackage{colortbl}
\usepackage{color}
\usepackage{ae}
\usepackage{enumitem}
\usepackage[all]{xy}
\usepackage[bf,small]{caption}
\usepackage{floatflt}
\usepackage{float}

\usepackage{chngcntr}
\counterwithin{figure}{section}

\DeclareMathOperator\arcsinh{arcsinh}

\makeatletter
\newtheorem*{rep@theorem}{\rep@title}
\newcommand{\newreptheorem}[2]{%
\newenvironment{rep#1}[1]{%
 \def\rep@title{#2 \ref{##1}}%
 \begin{rep@theorem}}%
 {\end{rep@theorem}}}
\makeatother

\makeatletter
\newtheorem*{rep@cor}{\rep@title}
\newcommand{\newrepcor}[2]{%
\newenvironment{rep#1}[1]{%
 \def\rep@title{#2 \ref{##1}}%
 \begin{rep@cor}}%
 {\end{rep@cor}}}
\makeatother

\makeatletter
\newtheorem*{rep@prop}{\rep@title}
\newcommand{\newrepprop}[2]{%
\newenvironment{rep#1}[1]{%
 \def\rep@title{#2 \ref{##1}}%
 \begin{rep@prop}}%
 {\end{rep@prop}}}
\makeatother

\newtheorem{cor}{Corollary}[section]
\newtheorem*{cor*}{Corollary}
\newtheorem{corx}{Corollary}

\makeatletter
\@addtoreset{corx}{subsection}
\makeatother

\newtheorem{theorem}[cor]{Theorem}
\newtheorem{thmx}[corx]{Theorem}

\newtheorem{prop}[cor]{Proposition}
\newtheorem{propx}[corx]{Proposition}

\newrepcor{cor}{Corollary}
\newreptheorem{theorem}{Theorem}
\newrepprop{prop}{Proposition}

\newtheorem{lemma}[cor]{Lemma}
\newtheorem{sublemma}[cor]{Sublemma}
\theoremstyle{definition}
\newtheorem{defi}[cor]{Definition}
\theoremstyle{remark}
\newtheorem{remark}[cor]{Remark}
\newtheorem{example}[cor]{Example}

\newtheoremstyle{TheoremNum}
        {\topsep}{\topsep}              %%% space between body and thm
        {\itshape}                      %%% Thm body font
        {}                              %%% Indent amount (empty = no indent)
        {\bfseries}                     %%% Thm head font
        {.}                             %%% Punctuation after thm head
        { }                             %%% Space after thm head
        {\thmname{#1}\thmnote{ \bfseries #3}}%%% Thm head spec
    \theoremstyle{TheoremNum}

\usepackage{xpatch}

\xpretocmd{\subsection}{\setcounter{corx}{0}}{}{}

\newcommand*{\bullets}{\raisebox{-0.25ex}{\scalebox{1.5}{$\cdot$}}}

\newcommand{\C}{{\mathbb C}}

\newcommand{\R}{{\mathbb R}}

\newcommand{\hess}{\mbox{Hess}}

\newcommand{\Hyp}{\mathbb{H}}
\newcommand{\AdS}{\mathbb{A}\mathrm{d}\mathbb{S}}

\newcommand{\PSL}{\mathrm{PSL}}

\newcommand{\ph}{\varphi}

\newcommand{\RP}{\mathbb{R}\mathrm{P}}
\newcommand{\Pml}{\Phi_{M\!L}}

\newcommand{\grad}{\operatorname{grad}}%{\mbox{grad}}
\newcommand{\isom}{\mathrm{Isom}}

\newcommand{\Ip}{\mathrm{I}^+}
\newcommand{\II}{I\hspace{-0.1cm}I}

\newcommand{\rar}{\rightarrow}

\newcommand{\id}{\mathrm{I}}
\newcommand{\SO}{\mathrm{SO}}

\newcommand{\ddt}{\left.\frac{d}{dt}\right|_{t=0}}

\newcommand{\D}{\mathbb{D}}

\newcommand{\wAdS}{\widehat{\AdS^3}}

\def\l{\lambda}

\def\Hess{\mathrm{Hess}}

\begin{document}

\setcounter{secnumdepth}{3}
\setcounter{tocdepth}{2}

\title[Maximal surfaces in $\AdS^3$, width and quasiconformal extensions]{Maximal surfaces in Anti-de Sitter space, width of convex hulls and quasiconformal extensions of quasisymmetric homeomorphisms}

\author{Andrea Seppi} 
\address{A. Seppi: Dipartimento di Matematica ``Felice Casorati", Universit\`{a} degli Studi di Pavia, Via Ferrata 5, 27100, Pavia, Italy.} \email{andrea.seppi01@ateneopv.it}

%\date{\today}
\maketitle

%\tableofcontents

%\bibstyle{plain}

\begin{abstract}
We give upper bounds on the principal curvatures of a maximal surface of nonpositive curvature in three-dimensional Anti-de Sitter space, which only depend on the width of the convex hull of the surface. Moreover, given a quasisymmetric homeomorphism $\phi$, we study the relation between the width of the convex hull of the graph of $\phi$, as a curve in the boundary of infinity of Anti-de Sitter space, and the cross-ratio norm of $\phi$. 

As an application, we prove that if $\phi$ is a quasisymmetric homeomorphism of $\mathbb{R}\mathrm{P}^1$ with cross-ratio norm $||\phi||$, then $\ln K\leq C||\phi||$, where $K$ is the maximal dilatation of the minimal Lagrangian extension of $\phi$ to the hyperbolic plane.
\end{abstract}

%---------------------------

%\listoffigures

\section*{Introduction}

The study of three-dimensional Anti-de Sitter space $\AdS^3$ was initiated by the pioneering work of Mess (\cite{mess}) of 1990, and has been widely developed since then, with emphasis on its relation with Teichm\"uller theory, for instance in \cite{aiyama, notes, bbzads, bon_schl, bsk_multiblack, bonschlfixed, Schlenker-Krasnov, bonseppitamb}.

In particular, in \cite{bon_schl} Bonsante and Schlenker studied zero mean curvature spacelike surfaces - hereafter \emph{maximal} surfaces - with boundary contained in the boundary at infinity $\partial_\infty\AdS^3$. The latter is identified in a natural way to $\RP^1\times\RP^1$, where isometries of $\AdS^3$ extend to projective transformations which act as elements of $\PSL(2,\R)\times\PSL(2,\R)$. Therefore the asymptotic boundary of a maximal surface is represented by the graph of an orientation-preserving homeomorphism $\phi:\RP^1\to\RP^1$. Bonsante and Schlenker proved that every curve in $\partial_\infty\AdS^3$ corresponding to the graph of an orientation-preserving homeomorphism $\phi:\RP^1\rar \RP^1$ bounds a maximal disc $S$ with nonpositive curvature. This result might be thought of as an asymptotic Plateau problem in Anti-de Sitter geometry. See also \cite{ksurfaces} for an alternative proof.
By the Gauss equation in $\AdS^3$, nonpositivity of curvature is equivalent to the condition that the principal curvatures of $S$ are in $[-1,1]$. 

Bonsante and Schlenker also provided a more precise description of maximal discs under the assumption that $\phi$ is a quasisymmetric homeomorphism of $\RP^1$ - namely, if the \emph{cross-ratio norm}
$$||\phi||_{cr}=\sup_{cr(Q)=-1}\left|\ln\left|cr(\phi(Q))\right|\right|$$
is finite. In this case, the maximal disc is unique and the principal curvatures are in $[-1+\epsilon,1-\epsilon]$ for some $\epsilon$. By means of a construction which associates to a maximal surface with nonpositive curvature a minimal Lagrangian diffeomorphism from $\Hyp^2$ to $\Hyp^2$ (a diffeomorphism of $\Hyp^2$ is minimal Lagrangian if it is area-preserving and its graph is a minimal surface in $\Hyp^2\times \Hyp^2$), the existence and uniqueness theorems on maximal surfaces led to the proof of the fact that every quasisymmetric homeomorphism of $\RP^1$ admits a unique quasiconformal minimal Lagrangian extension to $\Hyp^2$.

Another important ingredient introduced in \cite{bon_schl} is the \emph{width} of the convex hull. This is defined as the supremum of the length of timelike paths contained in the convex hull of the curve $gr(\phi)$. By a simple application of the maximum principle, the maximal surface $S$ with $\partial_\infty S=gr(\phi)$ is itself contained in the convex hull. Bonsante and Schlenker proved that for every orientation-preserving homeomorphism $\phi$, the width is at most $\pi/2$, and it is strictly less than $\pi/2$ precisely when $\phi$ is quasisymmetric.

The purpose of this paper is to study the quantitative relations between the cross-ratio norm of $\phi$, the width $w$ of its convex hull, and the supremum $||\lambda||_\infty$ of the principal curvatures of the maximal surface $S$ of nonpositive curvature such that $\partial_\infty S=gr(\phi)$. By the above discussion, $||\phi||_{cr}<+\infty$ if and only if $w<\pi/2$ if and only if $||\lambda||_\infty<1$, but it is not clear whether there is a direct relation between these quantities. Using a formula proved in \cite{Schlenker-Krasnov} which relates the differential of the minimal Lagrangian extension to the shape operator of $S$, our results will provide estimates on the maximal dilatation of the quasiconformal minimal Lagrangian extension, only depending on the cross-ratio norm of $\phi$.

%\begin{prop}
%If a simple closed curve $\Gamma$ in $\partial_\infty \Hyp^3$ spans a minimal disc $S$ with principal curvatures in $[-1+\epsilon,1-\epsilon]$, then $\Gamma$ is a quasicircle.
%\end{prop}

\subsection*{Principal curvatures of maximal surfaces} \addtocounter{subsection}{1}

The study of the relation between the principal curvatures of a maximal surface and the width of the convex hull is split into two parts. Observe that the principal curvatures of $S$ vanish identically when $S$ is a totally geodesic plane, in which case the width is zero since the convex hull consists of $S$ itself. Our first theorem describes the behavior of maximal surfaces which are close to being a totally geodesic plane:

\begin{thmx} \label{estimate principal curvatures and width ads}
There exists a constant $C_1$ such that, for every maximal surface $S$ with $||\lambda||_\infty<1$ and width $w$, 
$$||\lambda||_\infty\leq C_1\tan w\,.$$
\end{thmx}

This theorem provides interesting information only when $w$ is in some neighborhood of zero, since for large $w$ the already know bound $||\lambda||_\infty<1$ is not improved. On the other hand, Bonsante and Schlenker showed that if a maximal surface of nonpositive curvature has a point where the principal curvatures are $-1$ and $1$, then the principal curvatures are $-1$ and $1$ everywhere, and therefore the induced metric is flat. Moreover, the surface is a so-called \emph{horospherical surface}, which is described explicitly and has width $\pi/2$. Our second theorem concerns surfaces which are close to this situation:

\begin{thmx} \label{theorem width lambda big}
There exist universal constants $M>0$ and $\delta\in(0,1)$ such that, if $S$ is a maximal surface in $\AdS^3$ with $\delta\leq ||\lambda||_\infty<1$ and width $w$, then
$$\tan w\geq \left(\frac{1}{1-||\lambda||_\infty}\right)^{1/M}\,.$$
\end{thmx}

It is worth remarking here that an inequality going in the opposite direction can be obtained more easily, and all the necessary tools were already proved in \cite{Schlenker-Krasnov} and \cite{bon_schl}. Nevertheless, for the sake of completeness we will provide a proof of the following:

\begin{propx} \label{estimate principal curvatures and width ads reverse}
Let $S$ be a maximal surface in $\AdS^3$ with $||\lambda||_\infty\leq 1$ and width $w$. Then 
$$\tan w\leq \frac{2||\lambda||_\infty}{1-||\lambda||_\infty^2}\,.$$
\end{propx}
Since  $2||\lambda||_\infty/(1-||\lambda||_\infty^2)$ behaves like $2||\lambda||_\infty$ as $||\lambda||_\infty\to 0$, one sees that Theorem \ref{estimate principal curvatures and width ads} is optimal for small $||\lambda||_\infty$, up to determining the best possible value of the constant $C_1$. On the other hand, from Proposition \ref{estimate principal curvatures and width ads reverse} one obtains that $\tan w\leq 2/(1-||\lambda||_\infty)$, and it remains an open question whether Theorem \ref{theorem width lambda big} can be improved to an inequality of the form $\tan w\geq C_2^{-1}(1-||\lambda||_\infty)^{-1}$.

\subsection*{Minimal Lagrangian extensions} \addtocounter{subsection}{1}

A classical problem in Teichm\"uller theory concerns quasiconformal extensions to the disc of quasisymmetic homeomorphisms of the circle. Classical quasiconformal extensions include, for instance, the Beurling-Ahlfors extension and the Douady-Earle extension. More recently, Markovic \cite{mar_schoenconj} proved the existence of quasiconformal harmonic extensions, where the harmonicity is referred to the complete hyperbolic metric of $\Hyp^2$.

Moreover, the maximal dilatation of the classical extensions has been widely studied. %(\cite{ahl_beur,lehtinen,douadyearle,hu_muzician_douadyearle}).
For instance, Beurling and Ahlfors in \cite{ahl_beur} proved that, if $\Phi_{B\!A}$ is the Beurling-Ahlfors extension of a quasisymmetric homeomorphism $\phi$, then the maximal dilatation $K(\Phi_{B\!A})$ satisfies:
$$\ln K(\Phi_{B\!A})\leq 2||\phi||_{cr}\,.$$
The asymptotic behaviour was later improved in \cite{lehtinen} by
$$\ln K(\Phi_{B\!A})\leq ||\phi||_{cr}+\ln 2\,.$$
For the Douady-Earle extension, \cite{douadyearle} proved that there exist constants $\delta$ and $C$ such that, for every quasisymmetric homeomorphism of the circle $\phi$ with $||\phi||_{cr}<\delta$, the Douady-Earle extension $\Phi_{D\!E}$ satisfies:
%\begin{equation} \label{ineq DE}
$$\ln K(\Phi_{D\!E})\leq C||\phi||_{cr}\,.$$
%\end{equation}
More recently, Hu and Muzician proved in \cite{hu_muzician_douadyearle} that the following always holds:
$$\ln K(\Phi_{D\!E})\leq C_1||\phi||_{cr}+C_2\,.$$
%and therefore this shows that an inequality of the form \eqref{ineq DE} holds globally. 

In this paper we will prove analogous results for the minimal Lagrangian extension, whose existence was proved in \cite{bon_schl} as already remarked. As an application of Theorem \ref{estimate principal curvatures and width ads}, we will prove the following inequality:

\begin{thmx}  \label{minimal lag extension}
There exist universal constants $\delta$ and $C_1$ such that, for any quasisymmetric homeomorphism $\phi$ of $\RP^1$ with cross ratio norm $||\phi||_{cr}<\delta$, the minimal Lagrangian extension $\Phi_{M\!L}:\Hyp^2\rar\Hyp^2$ has maximal dilatation bounded by: $$\ln K(\Phi_{M\!L})\leq C_1||\phi||_{cr}\,.$$
\end{thmx}

On the other hand, by an application of Theorem \ref{theorem width lambda big} we will derive an asymptotic estimate of the maximal dilatation of $\Phi_{M\!L}$:
 
 \begin{thmx} \label{minimal lag extension large}
There exist universal constants $\Delta$ and $C_2$ such that, for any quasisymmetric homeomorphism $\phi$ of $\RP^1$ with cross ratio norm $||\phi||_{cr}>\Delta$, the minimal Lagrangian extension $\Phi_{M\!L}:\Hyp^2\rar\Hyp^2$ has maximal dilatation bounded by: $$\ln K(\Phi_{M\!L})\leq C_2||\phi||_{cr}\,.$$
\end{thmx}
 
Using Proposition \ref{estimate principal curvatures and width ads reverse} we also obtain an inequality in the converse direction, which holds for quasisymmetric homeomorphisms with small cross-ratio norm and shows that Theorem \ref{minimal lag extension} is not improvable from a qualitative point of view.
 
\begin{thmx} \label{minimal lag extension below}
There exist universal constants $\delta$ and $C_0$ such that, for any quasisymmetric homeomorphism $\phi$ of $\RP^1$ with cross ratio norm $||\phi||_{cr}<\delta$, the minimal Lagrangian extension $\Phi:\Hyp^2\rar\Hyp^2$ has maximal dilatation  bounded by: $$C_0||\phi||_{cr}\leq \ln K(\Phi_{M\!L})\,.$$
The constant $C_0$ can be taken arbitrarily close to $1/2$.
\end{thmx} 

Finally, from Theorem \ref{minimal lag extension} and Theorem \ref{minimal lag extension large} we will derive the following corollary.
 
\begin{corx} \label{corollary extension}
There exists a universal constant $C$ such that, for any quasisymmetric homeomorphism $\phi$ of $\RP^1$, the minimal Lagrangian extension $\Phi_{M\!L}:\Hyp^2\rar\Hyp^2$ has maximal dilatation $K(\Phi_{M\!L})$ bounded by: $$\ln K(\Phi_{M\!L})\leq C||\phi||_{cr}\,.$$
\end{corx} 
 
Corollary \ref{corollary extension} is therefore a result for minimal Lagrangian extensions comparable to what has been proved for Beurling-Ahlfors and Douady-Earle extensions.

\subsection*{From the width to the cross-ratio norm} \addtocounter{subsection}{1}

The bridge from Theorem \ref{estimate principal curvatures and width ads} to Theorem \ref{minimal lag extension}, and from Theorem \ref{theorem width lambda big} to Theorem \ref{minimal lag extension large}, is twofold. The first aspect is a direct relation between the principal curvatures of a maximal surface $S$ and the quasiconformal dilatation of the minimal Lagrangian extension at the corresponding point. This is proved in Proposition \ref{estimate principal curvatures and qc coefficient} by using a formula of \cite{Schlenker-Krasnov}, and has as a consequence that: 
$$
K(\Pml)=\left(\frac{1+||\lambda||_\infty}{1- ||\lambda||_\infty}\right)^2\,.
$$

On the other hand, the step from the width to the cross-ratio norm is more subtle. This is the content of the following proposition:

\begin{propx} \label{estimate width cross ratio norm}
Given any quasisymmetric homeomorphism $\phi$ of $\RP^1$, let $w$ be the width of the convex hull of the graph of $\phi$ in $\partial_\infty\AdS^3$. Then
$$\tan w\leq \sinh{\left(\frac{||\phi||_{cr}}{2}\right)}\,.$$
\end{propx}

By means of these two relations and some computations, Theorem \ref{minimal lag extension} and Theorem \ref{minimal lag extension large} are proved on the base of Theorem \ref{estimate principal curvatures and width ads} and Theorem \ref{theorem width lambda big}.

To prove Proposition \ref{estimate width cross ratio norm}, assuming that the width is $w$, we will essentially find two support planes $P_-$ and $P_+$ for the convex hull of $gr(\phi)$, on the two different sides of the convex hull, such that  $P_-$ and $P_+$ are connected by a timelike geodesic segment of length $w$. We will use the fact that the boundaries of the convex hull are pleated surfaces in order to pick four points in $\partial_\infty \AdS^3$ - two in the boundary at infinity $\partial_\infty P_-$ and the other two in $\partial_\infty P_+$ - and use such four points to show that the cross-ratio norm of $\phi$ is large. Turning this qualitative picture into quantitative estimates, leading to the proof of Proposition \ref{estimate width cross ratio norm}, involves careful and somehow technical constructions in Anti-de Sitter space.
%The details of the proof of Proposition \ref{estimate width cross ratio norm} are long, with several technicalities and computations. 
%Moreover, since the width is only defined as a supremum, to apply the above argument we first renormalize in a convenient position, by post-composing with isometries, and apply a compactness theorem for quasisymmetric homeomorphisms.

By using similar techniques, we will also prove an inequality in the converse direction, which is the content of the following proposition.
\begin{propx}  \label{estimate width cross ratio norm below}
Given any quasisymmetric homeomorphism $\phi$ of $\RP^1$, let $w$ the width of the convex hull of the graph of $\phi$ in $\partial_\infty\AdS^3$. Then
$$\tanh\left(\frac{||\phi||_{cr}}{4}\right)\leq\tan w\,.$$
\end{propx}

This inequality, however, is clearly not optimal, as the hyperbolic tangent tends to $1$ as $||\phi||_{cr}$ tends to infinity. Hence the inequality is interesting only for $w<\pi/4$. Nevertheless, this inequality  is used to obtain Theorem \ref{minimal lag extension below} from Proposition \ref{estimate principal curvatures and width ads reverse}. 
To prove Proposition \ref{estimate width cross ratio norm below}, we will assume the cross-ratio norm is $||\phi||_{cr}$ and - composing with M\"obius transformations in an appropriate way - construct a quadruple points in $\partial_\infty\AdS^3$. Then we consider two spacelike lines connecting two pairs of points at infinity chosen in the above quadruple. By construction, those two lines are contained in the convex hull of $gr(\phi)$, hence the maximal length of a timelike geodesic segment between them provides a bound from below on the width.

\subsection*{Outline of the main proofs} \addtocounter{subsection}{1}

Let us now give an outline of some technicalities involved in the proofs of Theorem \ref{estimate principal curvatures and width ads} and  Theorem \ref{theorem width lambda big}.

The starting point behind the proof of Theorem \ref{estimate principal curvatures and width ads} is the fact that a maximal surface $S$ with $\partial_\infty S=gr(\phi)$ in $\partial_\infty\AdS^3$ is contained in the convex hull of $gr(\phi)$. Using this fact, for every point $x\in S$ we find two timelike geodesic segments starting from $x$ and orthogonal to two planes $P_-$,$P_+$ which do not intersect the convex hull of $gr(\phi)$. The sum of the lengths of the two segments is less than the width $w$. Moreover $S$ is contained in the region bounded by $P_-$ and $P_+$.

Now the key step is to show that, heuristically, if $S$ is  contained in the region between two disjoint planes which are close to $x$, then the principal curvatures of $S$ in a neighborhood of $x$ cannot be too large. To make this statement precise, we will apply 
Schauder-type estimates to the linear equation
\begin{equation}
\Delta_S u-2u=0\,, \tag{\ref{lap 2u ads}}
\end{equation}
where $u:S\to\R$ is the function which measures the sine of the (signed, timelike) distance from the plane $P_-$, and $\Delta_S$ is the Laplace-Beltrami operator of $S$ (negative definite as an operator on $L^2$). 
Observe that an easy application of the maximum principle to Equation \eqref{lap 2u ads} proves that a maximal surface is necessarily contained in the convex hull. A more subtle study of \emph{a priori} bounds for this equation provides the key step for Theorem \ref{estimate principal curvatures and width ads}.

A technical point is that the operator $\Delta_S-2\mathrm{id}$ depends on the maximal surface, which will be overcame by using the uniform boundedness of the coefficients, written in normal coordinates, for a class of surfaces we are interested in. The precise statement we will use is the following:

\begin{propx} \label{schauder estimate ads}
There exists a radius $R>0$ and a constant $C>0$ such that for every choice of:
\begin{itemize}
\item A maximal surface $S$ of nonpositive curvature in $\AdS^3$ with $\partial_\infty S$ the graph of an orientation-preserving homeomorphism; 
\item A point $x\in S$;
\item A plane $P_-$ disjoint from $S$ with $d_{\AdS^3}(x,P_-)\leq \pi/4$,
\end{itemize}
the function $u(z)=\sin d_{\AdS^3}(\exp_x(z),P_-)$
 satisfies the Schauder-type inequality
\begin{equation*}
||u||_{C^2(B(0,\frac{R}{2}))}\leq C ||u||_{C^0(B(0,R))}\,.
\end{equation*} 
\end{propx}

The techniques involved here are similar to those used, in the case of minimal surfaces in three-dimensional hyperbolic geometry, in \cite{seppiminimal}.

To conclude, we then use an explicit expression for the shape operator of the maximal surface $S$ in terms of the value of $u$, the first derivatives of $u$, and the second derivatives of $u$. Hence, using Proposition \ref{schauder estimate ads}, the principal curvatures are bounded in terms of the supremum of $u$ on a geodesic ball $B_S(x,R)$. The latter can be estimated in terms of the width $w$. However, in this last step it is necessary to control the size of the image of $B_S(x,R)$ under the projection to the plane $P_-$. To achieve this, a uniform gradient lemma is proved, to show that the maximal surface $S$ is not too ``tilted'' with respect to $P_-$.

Similarly, the key analytical point for the proof of Theorem \ref{theorem width lambda big} comes from an \emph{a priori} estimate. Consider the function 
$\chi:S\to [0,+\infty)$ defined by 
$\chi=-\ln\lambda$, where $\lambda$ is the positive principal curvature of $S$. It turns out that $\chi$ satisfies the equation 
\begin{equation}
\Delta_S\chi=2(1-e^{-2\chi})\,. \tag{\ref{quasi}}
\end{equation}
By an application of the maximum principle, one can prove that if a maximal surface of nonpositive curvature has principal curvatures $-1$ and $1$ at some point, than the principal curvatures are identically $-1$ and $1$. By a careful analysis of Equation \eqref{quasi}, we will prove a more quantitative result, which roughly speaking shows that if the principal curvatures are ``large'' (i.e. close to $1$) at some point, then they remain ``large'' on a ``large'' ball.

\begin{propx} \label{cor estimate francesco}
There exists a universal constant $M$ such that, for every maximal surface $S$ of nonpositive curvature in $\AdS^3$ and every pair of points $p,q\in S$,
$$1-\lambda(q)\leq e^{Md_S(p,q)}(1-\lambda(p))\,.$$
\end{propx}

\noindent The proof of Proposition \ref{cor estimate francesco} %is obtained from a gradient estimate for the function $v(x)=-\ln(1-\lambda(x))$
is based on some estimates already proved jointly by Francesco Bonsante, Jean-Marc Schlenker and Mike Wolf (\cite{personalcomm}) in an unpublished work. The proof presented in this paper closely follows their arguments, which I was kindly transmitted and authorized to adapt to the purpose of this paper.

The strategy to prove Theorem \ref{theorem width lambda big} is then the following. 
Assume there exists a point $x_0$ with $\lambda(x_0)=1-e^{-v(x_0)}$ very close to $1$. We want to show that the width $w$ is very close to $\pi/2$. We will first show that the line of curvature of $S$ corresponding to the positive eigenvalue of $S$ remains, for a certain amount of time, in the concave side of an umbilical surface $U_{\lambda_1}$ tangent to $S$ at $x_0$, whose principal curvatures are both equal to $\lambda_1=1-e^{-v(x_0)/4}$. Such amount of time is finite (in a unit-speed parameterization of the line of curvature), but it can be arranged to tend to infinity as $\lambda(x_0)$ tends to $1$. This step is basically a maximum principle argument, but requires a technical point to show that the intrinsic acceleration of the line of curvature is also small, i.e. comparable to $e^{-v(x_0)}$.

The surface $U_{\lambda_1}$ is obtained as the surface at constant timelike distance $d_1$ from a totally geodesic plane. As $\lambda_1$ tends to $1$, $d_1$ tends to $\pi/4$. By following the line of curvature corresponding to the positive eigenvalue, in the two opposite directions from $x_0$, for a time as indicated in the above paragraph, we obtain two points $p_1$ and $p_2$. Moreover $p_1$ and $p_2$ converge to $\partial_\infty\AdS^3$ as $\lambda(x_0)\to 1$. Of course analogous statements hold for the two points $q_1,q_2$ obtained by following the line of curvature from $x_0$ corresponding to the negative eigenvalue.

After proving quantitative versions of the above statements, we can give a lower bound for the length of the timelike geodesic segment $\overline{r_1r_2}$ which maximizes the distance between the geodesic segments $\overline{p_1p_2}$ and $\overline{q_1q_2}$. Since $\overline{p_1p_2}$ and $\overline{q_1q_2}$ are contained in the convex hull of $S$, also $\overline{r_1r_2}$ is contained in the convex hull, and therefore the lower bound on the length of $\overline{r_1r_2}$ provides a lower bound for the width $w$, which only depends on $v(x_0)=-\ln(1-\lambda(x_0))$. The reason why the obtained estimate is efficient when $\lambda(x_0)$ approaches $1$ is that, in the limit configuration, the lines $\overline{p_1p_2}$ and $\overline{q_1q_2}$ tend to be in dual position: equivalently, in the limit every point of the first line is connected to every point of the second line by a geodesic timelike segment of length $\pi/2$. Thus as $\lambda(x_0)\to 1$, the surface $S$ is approaching a horospherical surface in a well-quantified fashion.

\subsection*{Organization of the paper}
In Section \ref{subsection Hyp AdS}, we introduce the necessary notions on Anti-de Sitter space, maximal surfaces and the width, we collect several results proved in \cite{bon_schl}, and finally we give some generalities on quasisymmetric homeomorphisms. 
In Section \ref{sec cross ratio width} we discuss the relation between the cross-ratio norm and the width in Anti-de Sitter space. The main result is Proposition \ref{estimate width cross ratio norm}. Section \ref{sec estimate small}
 proves Theorem \ref{estimate principal curvatures and width ads}, while Section \ref{sec estimate large} is devoted to the proof of Theorem \ref{theorem width lambda big}. Finally 
in Section \ref{section universal} we introduce quasiconformal mappings and minimal Lagrangian extensions, and we prove Theorem \ref{minimal lag extension}, Theorem \ref{minimal lag extension large}, Proposition \ref{minimal lag extension below} and Corollary \ref{corollary extension}.

\subsection*{Acknowledgements}
I am very grateful to Francesco Bonsante, Jean-Marc Schlenker and Mike Wolf for their interest in this work since my first stay at the University of Luxembourg in Spring 2014, for many discussions and advices, and particularly for suggesting (and permitting) that I use in the proof of Proposition \ref{cor estimate francesco} some crucial estimates arisen from a former collaboration of theirs.

Moreover, I would like to thank Dragomir \v{S}ari\'c and Jun Hu for replying to ad-hoc questions about Teichm\"uller theory in several occasions. Finally, I thank an anonymous referee for several useful comments.

%---------------------------

\section{Anti-de Sitter space and maximal surfaces} \label{subsection Hyp AdS}

Anti-de Sitter space $\AdS^3$ is a pseudo-Riemannian manifold of signature $(2,1)$ of constant curvature -1. Consider $\R^{2,2}$, the vector space $\R^4$ endowed with the bilinear form of signature (2,2):$$\langle x,y\rangle=x^1 y^1+x^2 y^2-x^3 y^3-x^4 y^4$$
and define $$\widehat{\AdS^3}=\left\{x\in\R^{2,2}:\langle x,x\rangle=-1\right\}\,.$$
It turns out that $\wAdS$ is connected, time-orientable and has the topology of a solid torus. We define Anti-de Sitter space to be the projective domain $$\AdS^3=\mathrm{P}(\left\{\langle x,x\rangle<0\right\})\subset\RP^3\,,$$
of which $\wAdS$ is a double cover. The pseudo-Riemannian metric induced on $\wAdS$ descends to a metric on $\AdS^3$ of constant curvature -1, again time-orientable, which will be denoted again by the product $\langle\cdot,\cdot\rangle$.

The tangent space of $\wAdS$ at a point $x$ is $T_x \wAdS\cong x^\perp$. 
Vectors in tangent spaces are classified according to their causal properties. In particular: 
$$v\in T_x \R^{2,1}\text{ is }\begin{cases}
\emph{timelike} &  \text{if }\langle v,v\rangle<0 \\
\emph{lightlike} & \text{if }\langle v,v\rangle=0 \\
\emph{spacelike} & \text{if }\langle v,v\rangle>0 
\end{cases}\,.$$
Hence, given a spacelike curve $\gamma:I\to\AdS^3$ (i.e. a differentiable curve whose tangent vector at every point is spacelike), we define the \emph{length} of $\gamma$
as $$\mathrm{length}(\gamma)=\int_I\sqrt{\langle\dot\gamma(t),\dot\gamma(t)\rangle}dt\,.$$
On the other hand, if $\gamma$ is a timelike curve, we still define its \emph{length}, as
$$\mathrm{length}(\gamma)=\int_I\sqrt{-\langle\dot\gamma(t),\dot\gamma(t)\rangle}dt\,.$$

We will fix once and forever a time orientation, so as to talk about \emph{future-directed} vectors and curves. Our convention is that the vector $(0,0,0,1)$ (based at the point $(0,0,1,0)$) is future-directed.

The group of isometries of $\wAdS$ which preserve the orientation and the time-orientation is  $\SO_+(2,2)$, namely the connected component of the identity in the group of linear isometries of $\R^{2,2}$. It follows that the group of orientation-preserving, time-preserving isometries of $\AdS^3$ is $\SO_+(2,2)/\left\{\pm\id\right\}$, and will be denoted simply by $\isom(\AdS^3)$.

Geodesics of $\wAdS$ are the intersection of $\wAdS$ with linear planes of $\R^{2,2}$. Therefore, geodesics of $\AdS^3$ are projective lines which intersect the projective domain $\AdS^3\subset\RP^3$. It is easy to see that a unit speed parameterization of a spacelike geodesic with initial point $p$ and initial tangent (spacelike) vector $v$ is:
\begin{equation} \label{geodetica1}
\gamma(t)=[\cosh(t)p+\sinh(t)v]\,,
\end{equation}
where the square brackets denote the class in $\AdS^3$ of a point of $\wAdS$. 
On the other hand, when $v$ is a timelike vector, the parameterization of the timelike geodesic is
\begin{equation} \label{geodetica2}
\gamma(t)=[\cos(t)p+\sin(t)v]\,.
\end{equation}
Hence timelike geodesics are closed and have length $\pi$. Equations \eqref{geodetica1} and \eqref{geodetica2} enable to derive immediately the formulae for the length of a spacelike geodesic segment:
\begin{equation} \label{lunghezza1}
\cosh(\mathrm{length}(\overline{pq}))=|\langle p,q\rangle|\,,
\end{equation}
while for a timelike geodesic segment one gets:
\begin{equation} \label{lunghezza2}
\cos(\mathrm{length}(\overline{pq}))=|\langle p,q\rangle|\,.
\end{equation}
Analogously, totally geodesic planes of $\AdS^3$ are projective planes and are the projection of the intersection of $\wAdS$ with three-dimensional linear subspaces of $\R^{2,2}$. There is a duality between totally geodesic planes and points of $\AdS^3$, which is given by associating to a point $x\in\AdS^3$ the \emph{dual plane} $P=x^\perp$. One defines
$$P=x^\perp\text{ is }\begin{cases}
\emph{timelike} &  \text{if }\langle x,x\rangle>0 \quad (\Leftrightarrow \text{the induced metric is Lorentzian}) \\
\emph{lightlike} & \text{if }\langle x,x\rangle=0  \quad ( \Leftrightarrow \text{the induced metric is degenerate})\\
\emph{spacelike} & \text{if }\langle x,x\rangle<0    \quad (\Leftrightarrow \text{the induced metric is Riemannian})
\end{cases}\,.$$
Spacelike totally geodesic planes are isometric to the hyperbolic plane $\Hyp^2$. Using Equation \eqref{lunghezza2}, it is easy to check that the dual plane of a point $x$ coincides with
$$x^\perp=\{\gamma(\pi/2)|\gamma:[0,\pi]\to\AdS^3\text{ is a timelike geodesic },\gamma(0)=\gamma(\pi)=x\}.$$

In the affine chart $\left\{x_4\neq 0\right\}$, $\AdS^3$ fills the domain $\left\{x^2+y^2<1+z^2\right\}$, interior of a one-sheeted hyperboloid; however $\AdS^3$ is not contained in a single affine chart, hence in this description we are missing a totally geodesic plane at infinity $P_\infty$, which is the dual plane of the origin. Since geodesics in $\AdS^3$ are intersections of $\AdS^3$ with linear planes in $\R^{2,2}$, in the affine chart geodesics are represented by straight lines. See Figure \ref{fig:lightcone} for a picture in the affine chart $\left\{x_4\neq 0\right\}$.

\begin{figure}[htbp]
\centering
\begin{minipage}[c]{.45\textwidth}
\centering
\includegraphics[height=5.5cm]{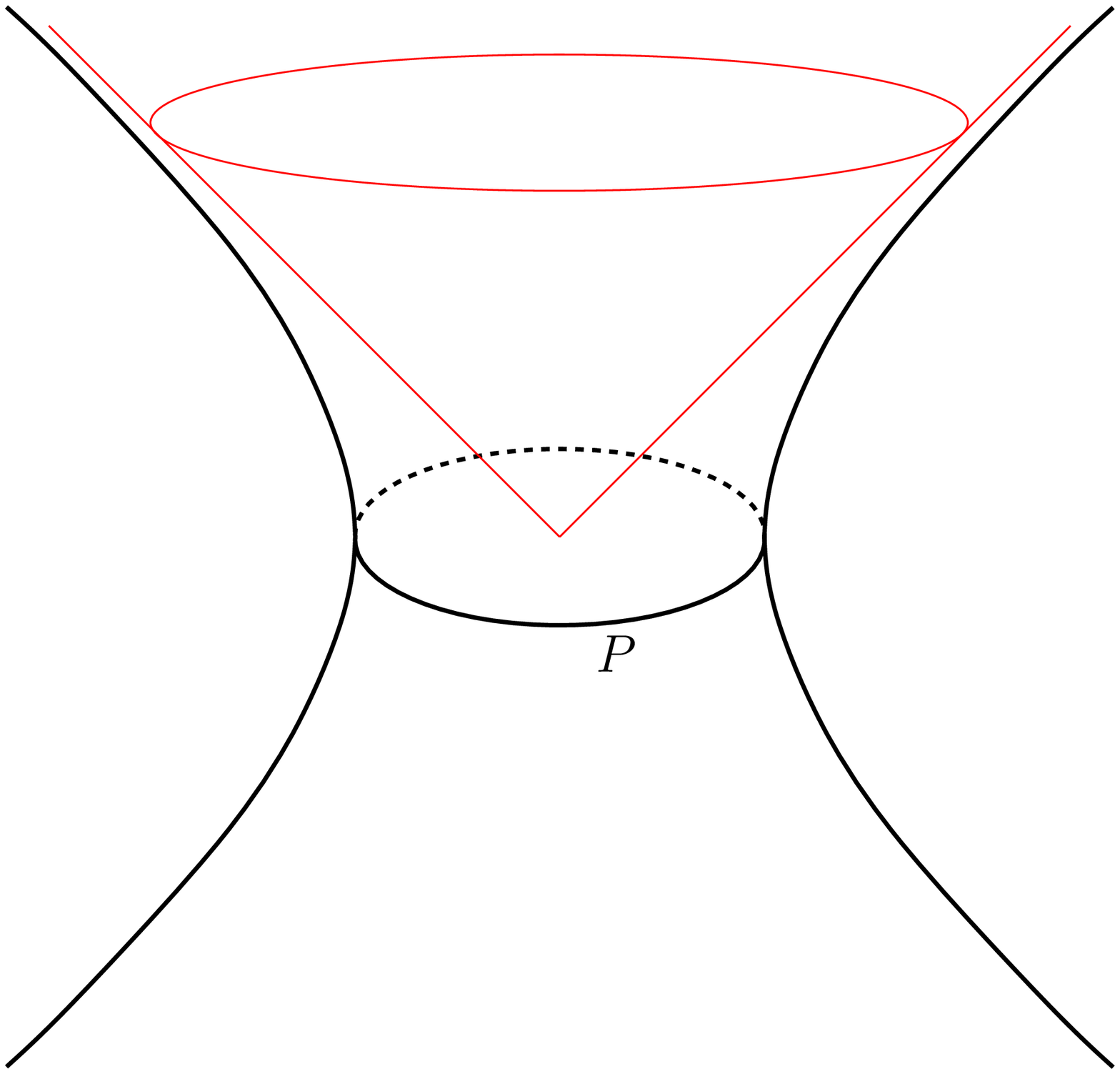} 
%\captionsetup{labelformat=empty}
\caption{The lightcone of future null geodesic rays from a point and a totally geodesic spacelike plane $P$.} \label{fig:lightcone}
\end{minipage}%
\hspace{5mm}
\begin{minipage}[c]{.45\textwidth}
\centering
\includegraphics[height=5.5cm]{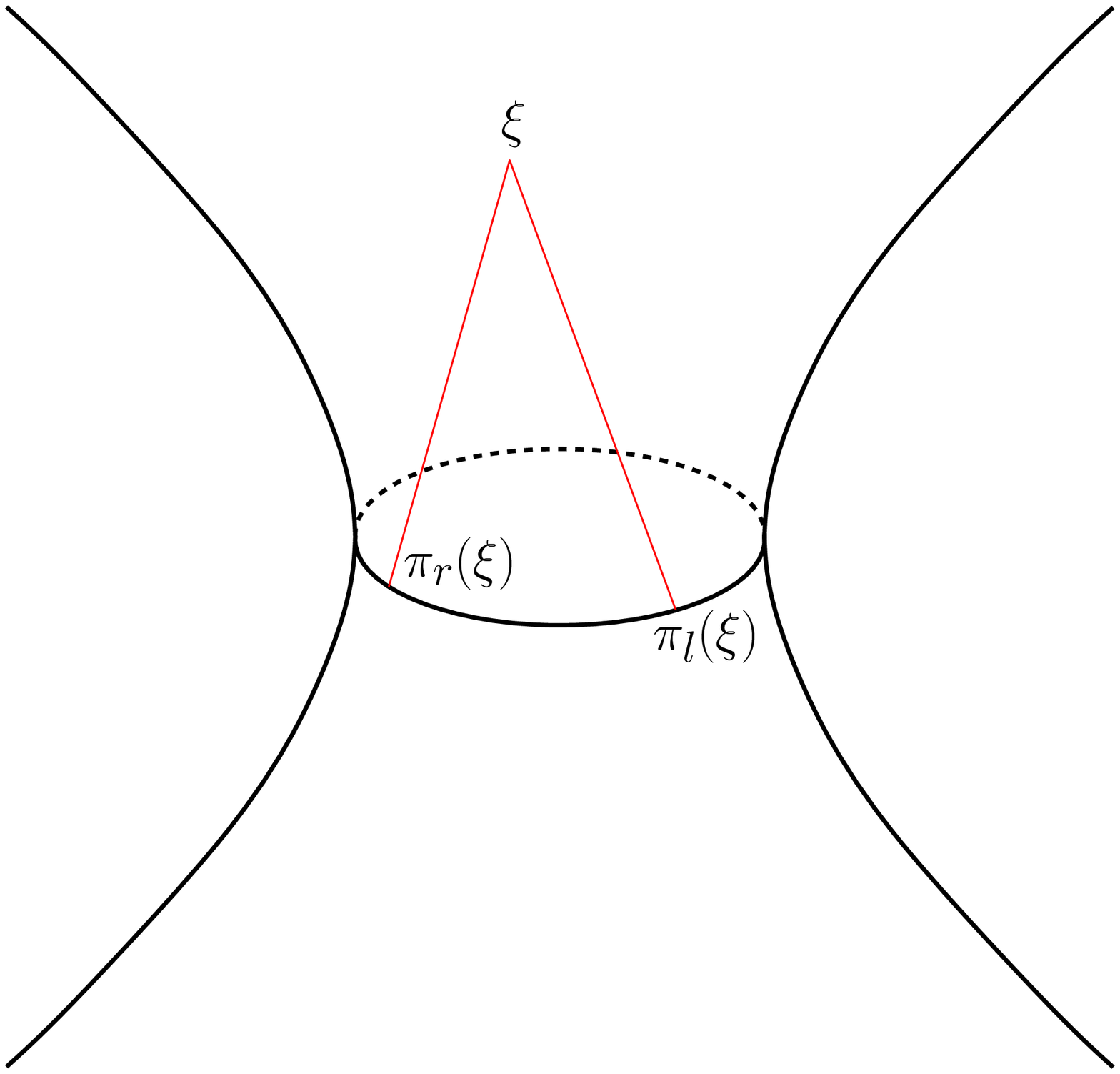}
%\captionsetup{labelformat=empty}
\caption{Left and right projection from a point $\xi\in\partial_\infty\AdS^3$ to the plane $P=\left\{x_3=0\right\}$} \label{fig:projections}
\end{minipage}
\end{figure}

The boundary at infinity of $\AdS^3$ is defined as the topological frontier of $\AdS^3$ in $\RP^3$, namely the doubly ruled quadric  $$\partial_\infty\AdS^3=\mathrm{P}(\left\{\langle x,x\rangle=0\right\})\,.$$  It is naturally endowed with a conformal Lorentzian structure, for which the null lines are precisely the left and right ruling. Given a spacelike plane $P$, which we recall is obtained as intersection of $\AdS^3$ with a linear hyperplane of $\RP^3$ and is a copy of $\Hyp^2$, $P$ has a natural boundary at infinity $\partial_\infty(P)$ which coincides with the usual boundary at infinity of $\Hyp^2$. Moreover, $P$ intersects each line in the left or right ruling in exactly one point. If a spacelike plane $P$ is chosen, $\partial_\infty\AdS^3$ can be identified with $\partial_\infty \Hyp^2\times \partial_\infty \Hyp^2$ by means of the following description: $\xi\in\partial_\infty\AdS^3$ corresponds to $(\pi_l(\xi),\pi_r(\xi))$, where $\pi_l$ and $\pi_r$ are the projection to $\partial_\infty(P)$ following the left and right ruling respectively (compare Figure \ref{fig:projections}). Under this identification, the isometry group of $\AdS^3$ acts on $\partial_\infty \AdS^3$ by projective transformations, and it turns out that
  $$\isom(\AdS^3)\cong\PSL(2,\R)\times\PSL(2,\R)\,.$$
  
Given an orientation-preserving homeomorphism $\phi:\partial_\infty \Hyp^2\rar \partial_\infty \Hyp^2$, by means of this identification, the graph of $\phi$ can be thought of as a curve in $\partial_\infty\AdS^3$, denoted simply by $gr(\phi)$.

%We choose a suitable coordinate system for $\PSL(2,\R)$: write a matrix $M$ as
%$$M=\begin{pmatrix} x^1+x^3 & x^2-x^4 \\ x^2+x^4 & -x^1+x^3 \end{pmatrix}.$$
%The quadratic form $$\det M=-(x^1)^2-(x^2)^2+(x^3)^2+(x^4)^2$$
%identifies $\PSL(2,\R)=P(\left\{M:\det M=1\right\})$ with $\AdS^3$. The boundary at infinity $\partial_\infty\AdS^3$ consists of matrices of rank 0, hence given two elements $V=[a:b],W=[c:d]\in \R P^1$, the corresponding element of $\partial_\infty\AdS^3$ is
%$$V^T W=\begin{bmatrix} a \\ b \end{bmatrix}\begin{bmatrix} c & d \end{bmatrix}=\begin{bmatrix} ac & ad \\ bc & bd \end{bmatrix}.$$

\subsection{Maximal surfaces}

This paper is concerned with spacelike embedded surfaces in Anti-de Sitter space. A smooth embedded surface $\sigma:S\rar\AdS^3$ is called spacelike if the first fundamental form $I(v,w)=\langle d\sigma(v),d\sigma(w)\rangle$ is a Riemannian metric on $S$. Equivalently, the tangent plane is a spacelike plane at every point.
Let $N$ be a unit normal vector field to the embedded surface $S$. We denote by $\nabla$ and $\nabla^S$ the ambient connection and the Levi-Civita connection of the surface $S$, respectively. The \emph{second fundamental form} of $S$ is defined as
$$\nabla_{\tilde v}\tilde w=\nabla^S_{\tilde v}\tilde w+\II(v,w)N$$
if $\tilde v$ and $\tilde w$ are vector fields extending $v$ and $w$. The \emph{shape operator} is the $(1,1)$-tensor defined as $B(v)=\nabla_v N$. It satisfies the property
$$\II(v,w)=\langle B(v),w\rangle\,.$$
\begin{defi}
A smooth embedded spacelike surface $S$ in  $\AdS^3$ is \emph{maximal} if $\mathrm{tr}B=0$.
\end{defi}

The shape operator is symmetric with respect to the first fundamental form of the surface $S$; hence the condition of  maximality amounts to the fact that the principal curvatures (namely, the eigenvalues of $B$) are opposite at every point.

\begin{defi}
We say that an embedded spacelike surface $S$ in  $\AdS^3$ is \emph{entire} if it is a compression disc, i.e. it is a topological disc and its frontier is contained in $\partial_\infty\AdS^3$.
\end{defi}

The condition that $S$ is entire is equivalent to the fact that $S$ can be expressed as a graph over $\Hyp^2$ in a suitable coordinate system, see \cite{bon_schl}.

An existence result for maximal surfaces in $\AdS^3$ was given by Bonsante and Schlenker. 
%Differently from the hyperbolic case, there are no non-uniqueness phenomena in the Anti-de Sitter case.

\begin{theorem}[\cite{bon_schl}] \label{thm bon schl existence}
Given any orientation-preserving homeomorphism $\phi:\RP^1\to \RP^1$, there exists an entire maximal surface with nonpositive curvature $S$ in $\AdS^3$ such that $\partial_\infty S=gr(\phi)$.
\end{theorem}

Observe that, by the Gauss equation in $\AdS^3$, the curvature of the induced metric on the maximal surface $S$ is $-1+\lambda^2$, where $\lambda$ and $-\lambda$ are the principal curvatures. Hence the condition of nonpositive curvature corresponds to the fact that the principal curvatures of $S$ are in $[-1,1]$.

\subsection{Width of convex hulls}

Since $\AdS^3$ is a projective geometry, we have a well-defined notion of convexity. In particular, we can give the definition of \emph{convex hull} of a curve in $\partial_\infty\AdS^3$.

\begin{defi}
Given a curve $\Gamma=gr(\phi)$ in $\partial_\infty \AdS^3$, the convex hull of $\Gamma$, which we denote by $\mathcal{CH}(\Gamma)$, is the intersection of half-spaces bounded by planes $P$ such that $\partial_\infty P$ does not intersect $\Gamma$, and the half-space is taken on the side of $P$ containing $\Gamma$.
\end{defi}
It can be proved that the convex hull of $\Gamma$, which is well-defined in $\RP^3$, is contained in $\AdS^3\cup\partial_\infty\AdS^3$, and is actually contained in an affine chart.

Let us fix  a totally geodesic spacelike plane $Q$ (for instance, the plane $\left\{x_4=0\right\}$, which is the plane at infinity in Figure \ref{fig:lightcone}). We will denote by $d_{\AdS^3}(\cdot,\cdot)$ the \emph{timelike distance} in $\AdS^3\setminus Q$, which is defined as follows.

\begin{defi}
Given points $p$ and $q$ which are connected by a timelike curve, the timelike distance between $p$ and $q$ is
 $$d_{\AdS^3}(p,q)=\sup_{\gamma}\{\mathrm{length(\gamma)}|\gamma:[0,1]\to\AdS^3\setminus Q \text{ is a timelike curve, }\gamma(0)=p,\gamma(1)=q\}\,.$$
\end{defi}
 The distance between two such points $p,q$ is achieved along the timelike geodesic connecting $p$ and $q$. The timelike distance satisfies the reverse triangle inequality, meaning that, $$d_{\AdS^3}(p,r)\geq d_{\AdS^3}(p,q)+d_{\AdS^3}(q,r)\,,$$
 provided both pairs $(p,q)$ and $(q,r)$ are connected by a timelike curve. In a completely analogous way, we define the distance of a point $x$ from a totally geodesic spacelike plane $P$ as the supremum of the length of a timelike curve connecting $x$ to $P$.

We are now ready to introduce the notion of \emph{width} of the convex hull, as defined in \cite{bon_schl}.
\begin{defi} \label{defi width}
Given a curve $\Gamma=gr(\phi)$ in $\partial_\infty \AdS^3$, the width $w$ of the convex hull $\mathcal{CH}(\Gamma)$ is the supremum of the length of a timelike geodesic contained in $\mathcal{CH}(\Gamma)$.
\end{defi}

\begin{remark} \label{discussion width}
Note that the distance $d_{\AdS^3}(p,q)$ is achieved along the geodesic timelike segment connecting $p$ and $q$. Hence Definition \ref{defi width}
is equivalent to 
\begin{equation} \label{width defi}
w=\sup_{p\in\partial_-\mathcal{C}, q\in\partial_+\mathcal{C}}  d_{\AdS^3}(p,q)\,,
\end{equation}
where $\mathcal{C}=\mathcal{CH}(\Gamma)$ and $\partial_{\pm}\mathcal{C}$ denote the two components (one future and one past component) of the boundary of the convex hull of $\Gamma$ in $\AdS^3$.

In particular, we note that
\begin{equation} \label{width sup sum}
w=\sup_{x\in\mathcal{C}}\left(d_{\AdS^3}(x,\partial_-\mathcal{C})+d_{\AdS^3}(x,\partial_+\mathcal{C})\right).
\end{equation}
To stress once more the meaning of this equality, note that the supremum in (\ref{width sup sum}) cannot be achieved on a point $x$ such that the two segments realizing the distance from $x$ to $\partial_-\mathcal{C}$ and $\partial_+\mathcal{C}$ are not part of a unique geodesic line. Indeed, if at $x$ the two segments form an angle, the piecewise geodesic can be made longer by avoiding the point $x$, as in Figure \ref{fig:shorten}.
We also remark that if the distance between a point $x$ and $\partial_\pm\mathcal{C}$ is achieved along a geodesic segment $l$, then the maximality condition imposes that $l$ must be orthogonal to a support plane to $\partial_\pm\mathcal{C}$ at $\partial_\pm\mathcal{C}\cap l$.
\end{remark}

%\begin{floatingfigure}[r]{.30\textwidth}
\begin{figure}[htbp]
\centering
\includegraphics[height=3cm]{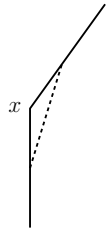}
\caption{A path through $x$ which is not geodesic does not achieve the maximum distance. \label{fig:shorten}}
\end{figure}

\subsection{An application of the maximum principle}

A key property used in this paper is that maximal surfaces with boundary at infinity a curve $\Gamma=gr(\phi)$ are contained in the convex hull of $\Gamma$. Although this fact is known, we prove it here by applying maximum principle to a simple linear PDE describing maximal surfaces.

Hereafter $\Hess \,u$ denotes the Hessian of a smooth function $u$ on the surface $S$, i.e. the (1,1) tensor 
$$\Hess \,u(v)=\nabla^S_v \grad u\,.$$
Finally, $\Delta_S$ denotes the Laplace-Beltrami operator of $S$, which can be defined as $$\Delta_S u=\mathrm{tr}(\Hess \,u)\,.$$

Proposition \ref{formule hess lap ads} was proved in \cite{bon_schl}. We give a proof here for the sake of completeness. We first observe that, given a point $x$ and a totally geodesic spacelike plane $P$, it is easy to check (as for Equations \eqref{lunghezza1} and \eqref{lunghezza2}) that the timelike distance $d_{\AdS^3}(x,P)$ of $x$ from $P=p^\perp$ satisfies
$$\sin d_{\AdS^3}(x,P)=|\langle x,p\rangle|\,.$$

\begin{prop} \label{formule hess lap ads}
Given a maximal surface $S\subset\AdS^3$ and a plane $P$, let $u:S\rar\R$ be the function $u(x)=\sin d_{\AdS^3}(x,P)$, where  $ d_{\AdS^3}(x,P)$ is considered as a signed distance. Let $N$ be the future unit normal to $S$ and $B=\nabla N$ the shape operator. Then
\begin{equation}\label{hessian ads}
\Hess\, u-u\,E=\sqrt{1-u^2+||\grad u||^2} B~,
\end{equation} 
where $E$ denotes the identity operator. As a consequence, $u$ satisfies the linear equation
\begin{equation}\label{lap 2u ads}
\Delta_S u-2u=0\,. \tag{L}
\end{equation} 
\end{prop}
\begin{proof}
Let us assume that $P$ is the plane dual to the point $p\in\AdS^3$. We will perform the computation in the double cover $\wAdS$. Then $u$ is the restriction to $S$ of the function $U$ defined by:
\begin{equation} \label{definition linear distance function}
U(x)=\sin d_{\AdS^3}(x,P)=\langle x,p\rangle\,.
\end{equation}
 Let $N$ be the unit normal vector field to $S$; we compute $\grad u$ by projecting the gradient $\nabla U$ of $U$ to the tangent plane to $S$:
\begin{gather}
\nabla U=p+\langle p,x \rangle x \label{bombo} \\
\grad u(x)=p+\langle p,x \rangle x + \langle p,N\rangle N
\end{gather}

\noindent Now $\Hess \,u(v)=\nabla^{S}_v \grad u$, where $\nabla^{S}$ is the Levi-Civita connection of $S$, namely the projection of the flat connection of $\R^{2,2}$, and so for any $v\in T_x\wAdS$ one gets:
$$\Hess\, u(v)=\langle p,x\rangle v + \langle p,N\rangle \nabla^S_v N=u(x)v+\langle \nabla U,N\rangle B(v)\,.$$
Moreover, $\nabla U=\grad u+\langle \nabla U,N\rangle N$ and thus
\begin{equation} \label{ciccio}
\langle \nabla U,N\rangle^2=||\grad u||^2-\langle\nabla U,\nabla U\rangle=1-u^2+||\grad u||^2\,,
\end{equation}
which proves (\ref{hessian ads}). By taking the trace, (\ref{lap 2u ads}) follows.
\end{proof}

\begin{cor} \label{maximal surface contained convex hull}
Let $S$ be an entire maximal surface in $\AdS^3$. Then $S$ is contained in the convex hull of $\partial_\infty S$.
\end{cor}
\begin{proof}
If $\Gamma=\partial_\infty S$ is a circle, then $S$ is a totally geodesic plane which coincides with the convex hull of $\Gamma$. Hence we can suppose $\Gamma$ is not a circle. Consider a plane $P_-$ which does not intersect $\Gamma$ and the function $u$ defined as in Equation \eqref{definition linear distance function} in Proposition \ref{formule hess lap ads}, with respect to $P_-$. Suppose their mutual position is such that $u\geq 0$ in the region of $S$ close to the boundary at infinity (i.e. in the complement of a large compact set). If there exists some point where $u<0$, then by Equation \eqref{lap 2u ads} at a minimum point $\Delta_S u=2u<0$, which gives a contradiction. The proof is analogous for a plane $P_+$ on the other side of $\Gamma$, by switching the signs. This shows that every convex set containing $\Gamma$ contains also $S$.
\end{proof}

\subsection{Two opposite examples}

We report here the two examples to have in mind for our study of maximal surface. The first is a very simple example, namely a totally geodesic plane, for which the principal curvatures vanish at every point. To some extent, the second example can be considered as the opposite of a totally geodesic plane, since in the second case the principal curvatures are $1$ and $-1$ at every point.

\begin{example}(Totally geodesic planes)
By definition, a totally geodesic plane $P$ has shape operator $B\equiv 0$. Hence a totally geodesic plane is a maximal surface, and $\partial_\infty P=gr(A)$, where $A\in\PSL(2,\R)$ is the trace on $\RP^1=\partial_\infty \Hyp^2$ of an isometry of $\Hyp^2$. Hence the width of the convex hull vanishes. It is also easy to see that, given a curve in $\partial_\infty\AdS^3$, if $w=0$ then the curve is necessarily the boundary of a totally geodesic plane. Therefore the unique maximal surface with zero width (up to isometries of $\AdS^3$) is a totally geodesic plane.
\end{example}

\begin{example}(Horospherical surfaces)
Consider a spacelike line $l$ in $\AdS^3$. The dual line $l^\perp$ is obtained as the intersection of all totally geodesic planes dual to points of $l$. Recall that timelike geodesics in $\AdS^3$ are closed and have length $\pi$. Hence one equivalently has:
$$l^\perp=\{\gamma(\pi/2)|\gamma:[0,\pi]\to\AdS^3\text{ is a timelike geodesic },\gamma(0)=\gamma(\pi)\in l\}\,.$$
Let us define the smooth surface
$$H=\{\gamma(\pi/4)|\gamma:[0,\pi/2]\to\AdS^3\text{ is a future-directed timelike geodesic},\gamma(0)\in l,\gamma(\pi/2)\in l^\perp\}\,.$$

See Figure \ref{fig:horosphere} for a schematic picture.

The group of isometries which preserve $l$ (and thus preserves also $l^\perp$) is isomorphic to $\R\times\R$, where $\R\times\{0\}$ fixes $l$ pointwise and acts on $l^\perp$ by translation (it actually acts as a rotation around the spacelike line $l$), while  $\{0\}\times\R$ does the opposite. The induced metric on $H$ is flat, and thus $H$ is isometric to the Euclidean plane. Moreover, for every point $x\in H$, the surface $H$ has an orientation-reversing, time-reversing isometry obtained by reflection in a plane $P$ tangent to a point $x\in H$, followed by rotation of angle $\pi/2$ around the timelike geodesic orthogonal to $P$ at $x$. This basically shows that the principal curvatures $\lambda_1,\lambda_2$ of $H$ are necessarily opposite to one another, hence $H$ is a maximal surface. Moreover, since by the Gauss equation
$$0=K_H=-1-\lambda_1\lambda_2=-1+\lambda^2\,,$$
it follows that the principal curvatures are necessarily $\pm 1$ at every point. Finally, by construction, the width of the convex hull of $H$ is precisely $\pi/2$.
\end{example}

\begin{figure}[htbp]
\centering
\begin{minipage}[c]{.45\textwidth}
\centering
\includegraphics[height=7cm]{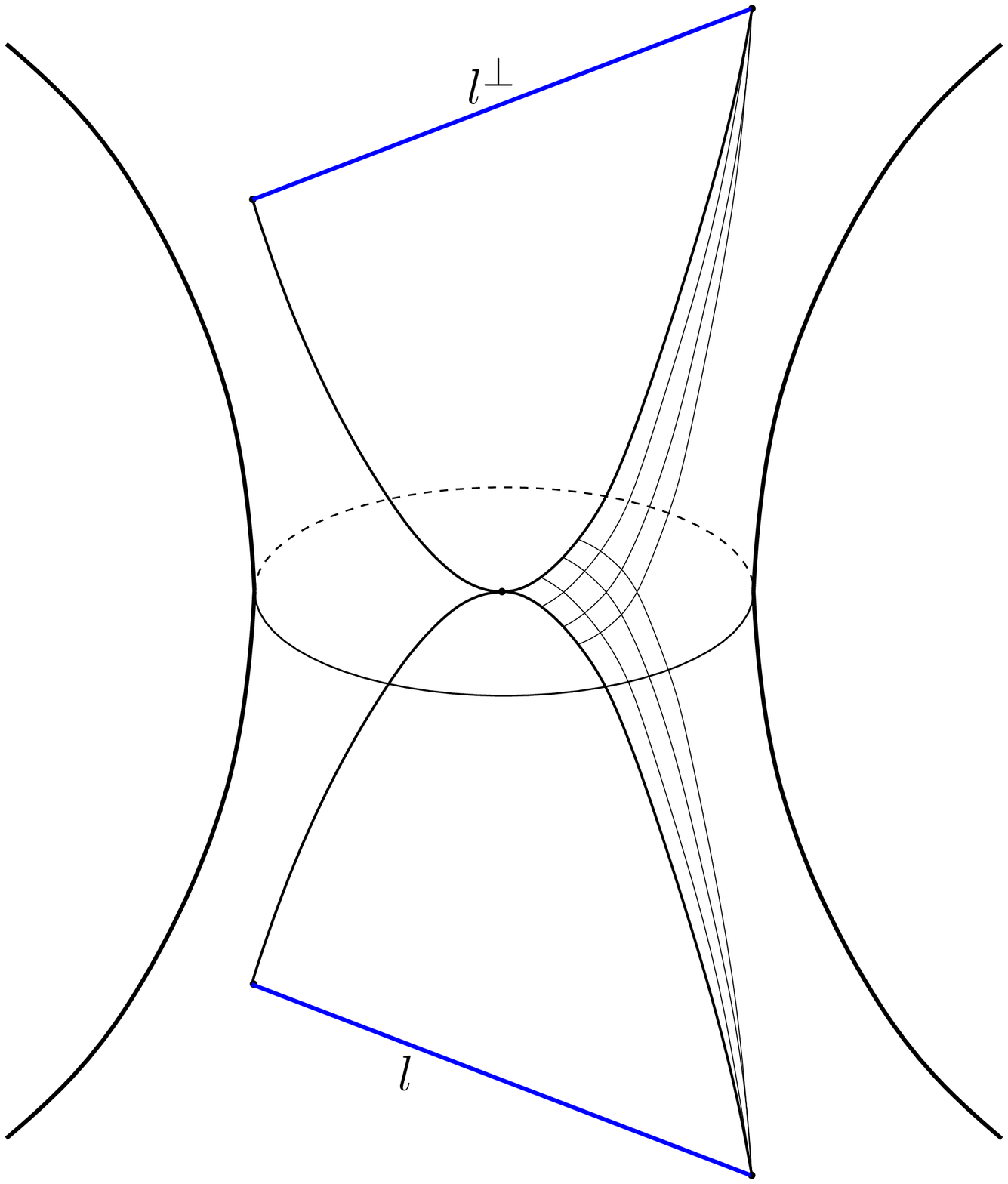} 
%\captionsetup{labelformat=empty}
%\caption{The lightcone of future null geodesic rays from a point and a totally geodesic spacelike plane $P$.} \label{fig:lightcone}
\end{minipage}%
\hspace{5mm}
\begin{minipage}[c]{.45\textwidth}
\centering
\includegraphics[height=7cm]{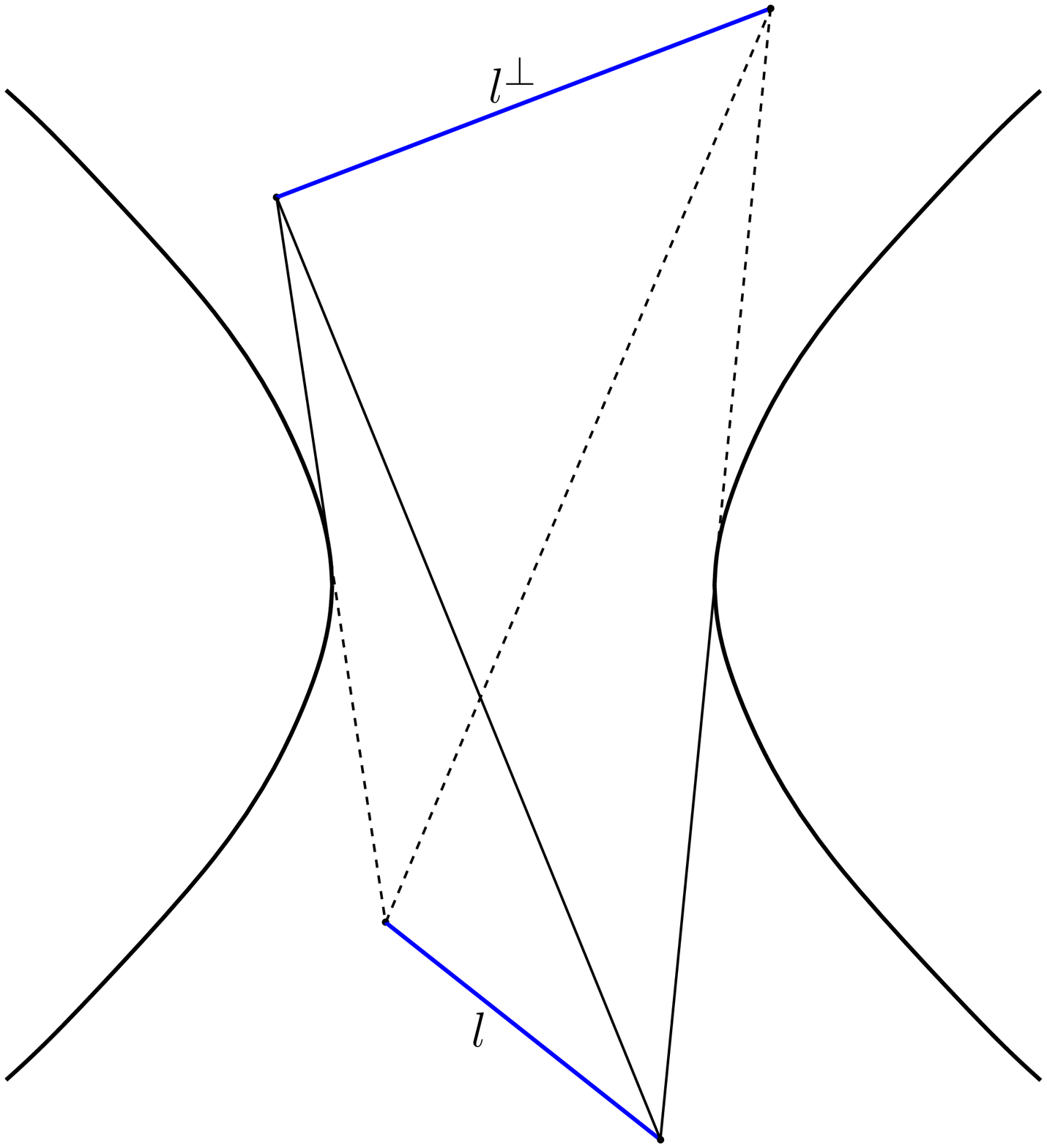}
%\captionsetup{labelformat=empty}
%\caption{Left and right projection from a point $\xi\in\partial_\infty\AdS^3$ to the plane $P=\left\{x_3=0\right\}$} \label{fig:projections}
\end{minipage}
\caption{A subset of a horospherical surface. The lines of curvature provide Euclidean coordinates for the induced metric, and are planar curves parallel to $l$ or $l^\perp$. The boundary at infinity is composed of four lightlike segments. \label{fig:horosphere}}
\end{figure}

Bonsante and Schlenker proved an important property of rigidity of maximal surfaces with large principal curvatures.

\begin{lemma}[{\cite[Lemma 5.5]{bon_schl}}] \label{lemma flat horo}
Given a maximal surface $S$ in $\AdS^3$ with nonpositive curvature, if the curvature is $0$ at some point, then $S$ is a subset of a horospherical surface.
\end{lemma}

The proof of Lemma \ref{lemma flat horo} follows from applying the maximum principle to Equation \eqref{quasi}, which is stated in the following lemma. Recall that an umbilical point on a surface is a point where the two principal curvatures are equal. Hence for a maximal surface, umbilical points are the points where the principal curvatures vanish.

\begin{lemma}[{\cite[Lemma 3.11]{Schlenker-Krasnov}}] \label{lemma eq quasilin}
Given a maximal surface $S$ in $\AdS^3$, with principal curvatures $\pm\lambda$, let $\chi:S\to\R$ be the function defined (in the complement of umbilical points) by
$\chi=-\ln |\lambda|$. Then $\chi$ satisfies the quasi-linear equation
\begin{equation}\label{quasi}
\Delta_S\chi=2(1-e^{-2\chi})\,. \tag{Q}
\end{equation}
\end{lemma}

\subsection{Uniformly negative curvature}

As a warm-up for what will come next, we give here a proof of Proposition \ref{estimate principal curvatures and width ads reverse}. Our proof was basically already implicit in \cite[Claim 3.21]{bon_schl}, and will be a consequence of the following easy lemma, which we prove here by completeness. See also \cite{Schlenker-Krasnov}.

\begin{lemma} \label{lemma formule prima seconda shape distanza costante}
Given a smooth spacelike surface $S$ in $\AdS^3$, let $S_\rho$ be the surface at timelike distance $\rho$ from $S$, obtained by following the normal flow. Then the pull-back to $S$ of the induced metric on the surface $S_\rho$ is given by
\begin{equation}
I_\rho=I((\cos(\rho) E+\sin(\rho) B)\cdot,(\cos(\rho) E+\sin(\rho) B)\cdot)\,.
\end{equation}
\noindent The second fundamental form and the shape operator of $S_\rho$ are given by
\begin{gather}
\II_\rho=I((-\sin(\rho) E+\cos(\rho) B)\cdot,(\cos(\rho) E-\sin(\rho) B)\cdot)\,, \\
B_\rho=(\cos(\rho) E+\sin(\rho) B)^{-1}(-\sin(\rho) E+\cos(\rho) B)\,. \label{shape operator constant distance}
\end{gather}
\end{lemma}
\begin{proof}
Let $\sigma$ be a smooth embedding of the maximal surface $S$, with oriented unit normal $N$. The geodesics orthogonal to $S$ at a point $x=\sigma(y)$ can be written as
$$\gamma_x(\rho)=\cos(\rho)\sigma(y)+\sin(\rho)N(x)\,.$$
Then we compute
\begin{align*}
I_\rho(v,w)=&\langle d\gamma_x(\rho)(v),d\gamma_x(\rho)(w) \rangle \\
=&\langle \cos(\rho)d\sigma_x(v)+\sin(\rho)d(N\circ\sigma)_x(v),\cos(r)d\sigma_x(w)+\sin(\rho)d(N\circ\sigma)_x(w)\rangle \\
=&I(\cos(\rho)v+\sin(\rho)B(v),\cos(\rho)w+\sin(\rho)B(w))\,.
\end{align*}
We have used the equation $B=\nabla N$. The formula for the second fundamental form follows from the fact that $\II_\rho=\frac{1}{2}\frac{dI_\rho}{d\rho}$, and the formula for $B_\rho$ from equating $B_\rho=I_\rho^{-1}\II_\rho$.
\end{proof}

It follows that, if the principal curvatures of a maximal surface $S$ are $\lambda_1=\lambda\in[0,1)$ and $\lambda_2=-\lambda$, then the principal curvatures of $S_\rho$ are $$\lambda_\rho=\frac{\lambda-\tan(\rho)}{1+\lambda\tan(\rho) }=\tan(\rho_0-\rho)\,,$$ where $\tan\rho_0=\lambda$, and $$\lambda'_\rho=\frac{-\lambda-\tan(-\rho)}{1-\lambda\tan(\rho) }=\tan(-\rho_0-\rho)\,.$$ In particular $\lambda_\rho$ and $\lambda'_\rho$ are non-singular for every $\rho$ between $-\pi/4$ and $\pi/4$.
It then turns out that $S_\rho$ is convex at every point for $\rho<-||\rho_0||_\infty$, and concave for $\rho>||\rho_0||_\infty$. This proves that the width $w$ is less than $2||\rho_0||_\infty=2\arctan||\lambda||_\infty$. Therefore 
$$\tan\left(\frac{\omega}{2}\right)\leq ||\lambda||_\infty\,,$$
and the statement of Proposition \ref{estimate principal curvatures and width ads reverse} follows.
 %Observe that the surfaces $S_\rho$ all have the same boundary at infinity, say $\Gamma=gr(\phi)$, and foliate the domain of dependence of $\Gamma$. The following is then proved:
%it follows than the width of the convex hull of $\partial_\infty(S)$ is less than $2\arctan\lambda$. 

\begin{repprop}{estimate principal curvatures and width ads reverse} 
Let $S$ be an entire maximal surface in $\AdS^3$ with $||\lambda||_\infty<1$ and let $w$ be the width of the convex hull.  Then 
$$\tan w\leq \frac{2||\lambda||_\infty}{1-||\lambda||_\infty^2}\,.$$
\end{repprop}

A direct consequence is that, if $S$ is an entire maximal surface with uniformly negative curvature (equivalently, with $||\lambda||_\infty<1$), then $w<\pi/2$, see also \cite[Corollary 3.22]{bon_schl}. Also the converse holds:

\begin{prop}[\cite{bon_schl}] \label{prop bs width neg curv}
Given any entire maximal surface $S$ of nonpositive curvature in $\AdS^3$, the width $w$ of the convex hull of $S$ is at most $\pi/2$. Moreover, $w<\pi/2$ if and only if $S$ has uniformly negative curvature.
\end{prop}

The proof of the converse implication in \cite{bon_schl} used quasisymmetric homemorphisms (see below) and minimal Lagrangian extensions (which will be discussed in Section \ref{section universal}). Observe that our Theorem \ref{theorem width lambda big} gives a quantitative version of the converse implication of Proposition \ref{prop bs width neg curv}.
%while the quantitative version of the other direction is provided by Proposition \ref{estimate principal curvatures and width ads reverse}.

\subsection{Quasisymmetric homeomorphisms of the circle} \label{subsec: quasisymmetric model}

Given an orientation-preserving homeomorphism $\phi:\RP^1\to \RP^1$, we define the \emph{cross-ratio norm} of $\phi$ as
$$||\phi||_{cr}=\sup_{cr(Q)=-1}\left|\ln\left|cr(\phi(Q))\right|\right|\,,$$
where $Q=(z_1,z_2,z_3,z_4)$ is any quadruple of points on $\RP^1$ and we use the following definition of cross-ratio:
$$cr(z_1,z_2,z_3,z_4)=\frac{(z_4-z_1)(z_3-z_2)}{(z_2-z_1)(z_3-z_4)}.$$
According to this definition, a quadruple $Q=(z_1,z_2,z_3,z_4)$ is symmetric (i.e. the hyperbolic geodesics connecting $z_1$ to $z_3$ and $z_2$ to $z_4$ intersect orthogonally) if and only if $cr(Q)=-1$. 

Observe that $||\phi||_{cr}\in[0,\infty]$ and $||\phi||_{cr}=0$ if and only if $\phi$ is a projective transformation, i.e. $\phi\in\PSL(2,\R)$. Indeed, by post-composing with a projective transformation, one can assume that $\phi$ fixes three points of $\RP^1$, and then the conclusion is straightforward. A homeomorphism is \emph{quasisymmetric} when it has finite cross-ratio norm:

\begin{defi} \label{defi quasisymmetric}
An orientation-preserving homeomorphism $\phi:\RP^1\to \RP^1$ is quasisymmetric if and only if $||\phi||_{cr}<+\infty$.
\end{defi}

Quasisymmetric homeomorphisms arise naturally in the context of quasiconformal mappings and universal Teichm\"uller space, which is actually one of the main themes of this paper. However, we defer the introduction of this point of view to Section \ref{section universal}, where some applications are given.

The first intimate correlation between the cross-ratio norm of a quasisymmetric homeomorphism $\phi$ and the width of the convex hull of $gr(\phi)$ was proved in \cite{bon_schl}:

\begin{theorem}[{\cite[Theorem 1.12]{bon_schl}}] \label{thmBS quasi width}
Given any orientation-preserving homeomorphism $\phi:\RP^1\to \RP^1$, let $w$ be the width of the convex hull of $gr(\phi)$. Then $w<\pi/2$ if and only if $\phi$ is quasisymmetric.
\end{theorem}

Again, Proposition \ref{estimate width cross ratio norm} will provide a more precise version of Theorem \ref{thmBS quasi width}, giving a quantitative inequality between the width and the cross-ratio norm.

\begin{remark}
Actually, Theorem  \ref{thmBS quasi width} holds under a more general hypothesis, namely for more general curves in $\partial_\infty\AdS^3$, which, roughly speaking, may also contain lightlike segments. We will not be interested in these objects in this paper.
\end{remark}

A refinement of Theorem \ref{thm bon schl existence} was given in \cite{bon_schl} under the assumption that the curve at infinity is the graph of a quasisymmetric homeomorphism. 

\begin{theorem}[\cite{bon_schl}] \label{teorema superfici massime quasisimmetrico}
Given any \textbf{quasisymmetric} homeomorphism $\phi:\RP^1\rar \RP^1$, there exists a \textbf{unique} entire maximal surface $S$ in $\AdS^3$ with \textbf{uniformly} negative curvature such that $\partial_\infty S=gr(\phi)$. 
\end{theorem}

%Universal Teichm\"uller space is then equivalently defined as the space of quasisymmetric homeomorphisms of the circle up to post-composition with M\"obius transformations:
%$$\Teich(\D)=\left\{\phi:S^1\to S^1\text{ quasisymmetric}\right\}/\mbox{M\"ob}(S^1)\,.$$
%Again, $\Teich(\D)$ can be identified to the space of quasisymmetric homeomorphisms of $S^1$ fixing $1$, $i$ and $-1$.

%The topology of $\Teich(\D)$ coincides with the topology induced by the distance induced on the space of quasisymmetric homeomorphisms by the cross-ratio norm. Namely, one can define the distance between two quasisymmetric homeomorphisms $\phi,\phi'$ as $||\phi^{-1}\circ \phi'||_{cr}$.

\subsection{Compactness properties} \label{subsection compactness}

One of the main tools used in \cite{bon_schl} is a result of compactness for maximal surfaces in $\AdS^3$. To conclude the preliminaries of this paper, we briefly discuss some properties of compactness for maximal surfaces and quasisymmetric homeomorphisms.

Given a spacelike plane $P_0$ in $\AdS^3$ and a point $x_0\in P_0$, let $l$ be the timelike geodesic through $x_0$ orthogonal to $P_0$. We define the solid cylinder $Cl(x_0,P_0,R_0)$ of radius $R_0$ above $P_0$ centered at $x_0$ as the set of points $x\in\AdS^3$ which lie on a spacelike plane $P_x$ orthogonal to $l$ such that $d_{P_x}(x,l\cap P_x)\leq R_0$. See also Figure \ref{fig:cylinder}.

\begin{figure}[htbp]
\centering
\begin{minipage}[c]{.45\textwidth}
\centering
\includegraphics[height=6cm]{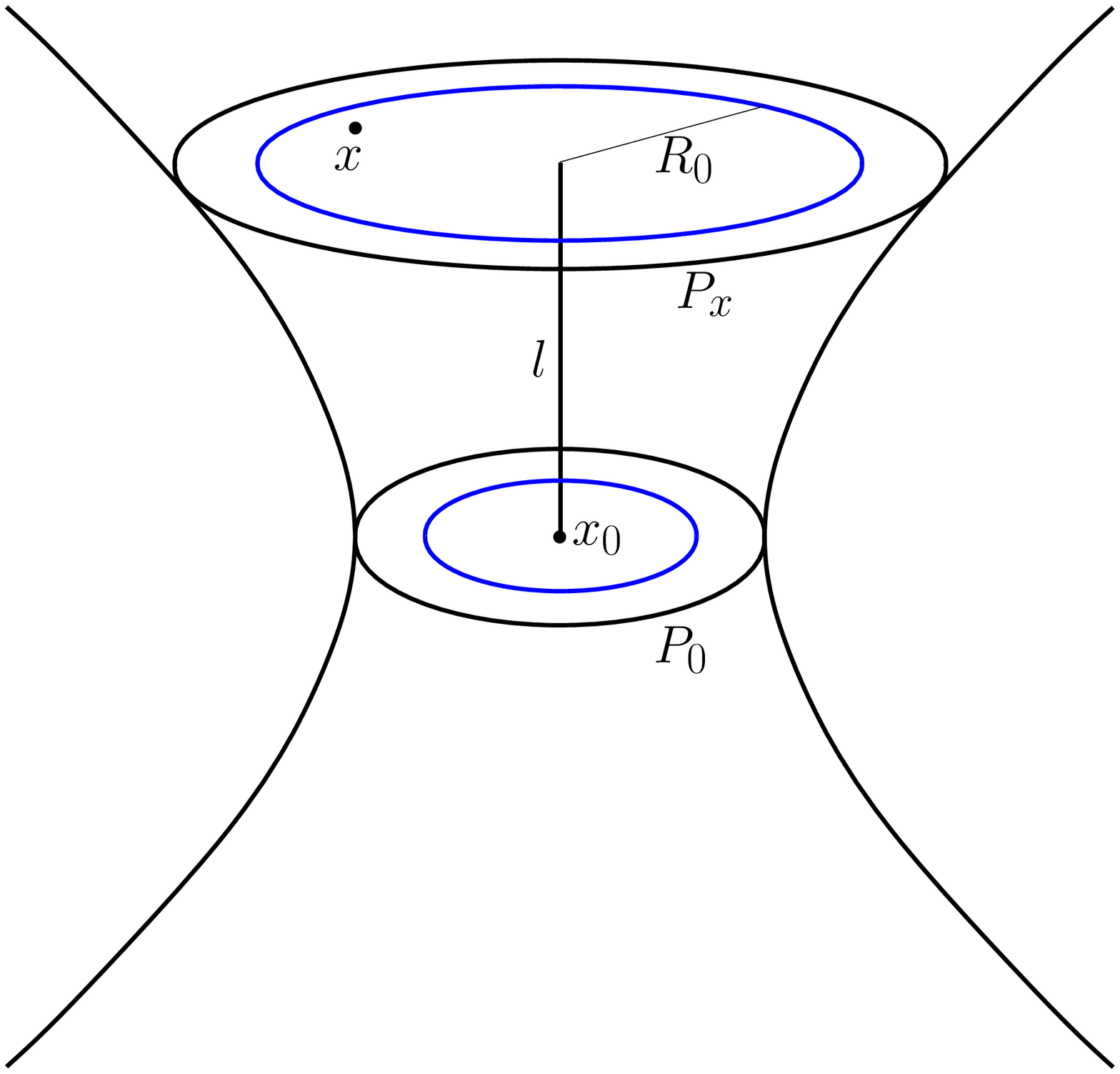} 
%\captionsetup{labelformat=empty}
\end{minipage}%
\hspace{5mm}
\begin{minipage}[c]{.50\textwidth}
\centering
\includegraphics[height=6cm]{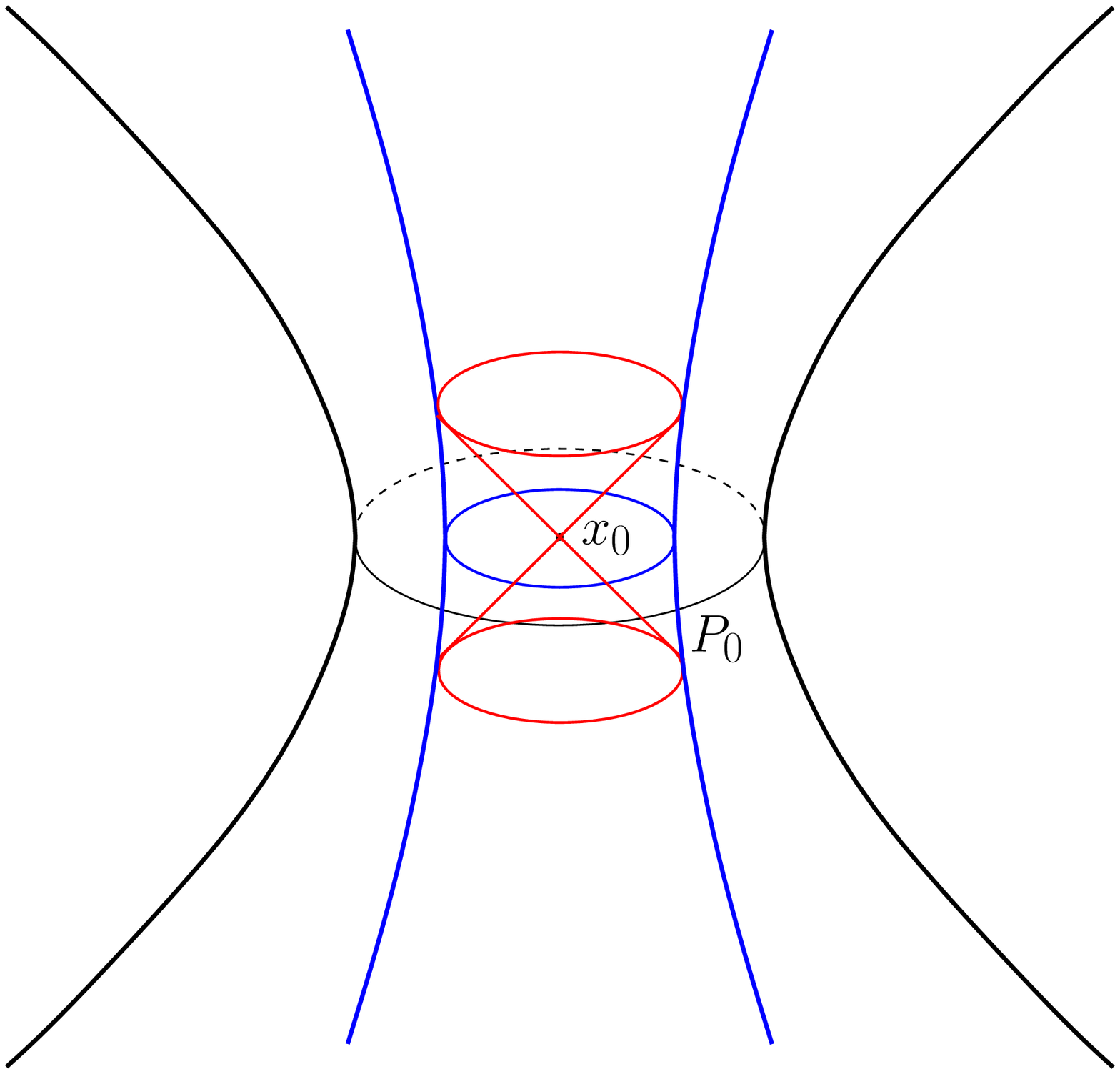}
%\captionsetup{labelformat=empty}
\end{minipage}
\caption{The definition of the solid cylinder $Cl(x_0,P_0,R_0)$. On the right, its intersection with the future and past lightcone over $x_0$. \label{fig:cylinder}}
\end{figure}

The solid cylinder $Cl(x_0,P_0,R_0)$ can also be conveniently described in the following way. Assuming (in the double cover $\wAdS$, for one moment) that $x_0=(0,0,1,0)$ and the normal vector to $P_0$ is $N_0=(0,0,0,1)$, let us consider the following coordinate system:
\begin{equation} \label{cylindric coordinates}
(r,\theta,\zeta)\mapsto [\cos\theta\sinh r,\sin\theta\sinh r,\cos\zeta\cosh r,\sin\zeta\cosh r]\,,
\end{equation}
defined for $r\in\R,\theta\in S^1,\zeta\in S^1$.
This means that the level sets with $\zeta=c$ are totally geodesic planes orthogonal to the timelike line
$$\{[0,0,\cos\zeta,\sin\zeta]:\zeta\in S^1\}\,,$$
which passes through the point $x_0=(0,0,1,0)$ with future-directed normal $N_0=(0,0,0,1)$. It is easy to see that, by this description, the solid cylinder $Cl(x_0,P_0,R_0)$ is determined by the relation $r\leq R_0$.

From the tools in the paper \cite{bon_schl}, the following lemma is proved:

\begin{lemma} \label{lemma bon schl}
Given a spacelike plane $P_0$ and a point $x_0\in P_0$, every sequence $S_n$ of entire maximal surfaces with nonpositive curvature, tangent to $P_0$ at $x_0$, admits a subsequence converging $C^{\infty}$ to an entire maximal surface on the solid cylinders $Cl(x_0,P_0,R)$, for every $R>0$. 
\end{lemma}

A somehow similar property of compactness for quasisymmetric homeomorphisms will be used in several occasions. See also \cite{basmajian} for a discussion.

\begin{theorem} \label{Compactness property of quasisymm homeo}
Let $k>0$ and $\phi_n:\RP^1\to \RP^1$ be a family of orientation-preserving quasisymmetric homeomorphisms of the circle, with $||\phi_n||_{cr}\leq k$. Then there exists a subsequence $\phi_{n_k}$ for which one of the following holds:
\begin{itemize}
\item The homeomorphisms $\phi_{n_k}$ converge uniformly to a quasisymmetric homeomorphism $\phi:\RP^1\to \RP^1$, with $||\phi||_{cr}\leq k$;
\item The homeomorphisms $\phi_{n_k}$ converge uniformly on the complement of any open neighborhood of a point of $\RP^1$ to a constant map $c:\RP^1\to \RP^1$.
\end{itemize}
\end{theorem}

\begin{remark} \label{remark compactness}
It is also not difficult to prove that, given a sequence of entire maximal surfaces $S_n$ which converges uniformly on compact cylinders $Cl(x_0,P_0,R)$ to an entire maximal surface $S_\infty$, then the asymptotic boundaries $\partial_\infty S_n$ converge (in the Hausdorff convergence, for instance) to the asymptotic boundary $\partial_\infty S_\infty$ of the limit surface.
\end{remark}

%---------------------------

\section{Cross-ratio norm and the width of the convex hull} \label{sec cross ratio width}

The purpose of this section is to investigate the relation between the width of the convex hull of $gr(\phi)$ and the cross-ratio norm of $\phi$ when $\phi$ is a quasisymmetric homeomorphism. The main result is thus Proposition \ref{estimate width cross ratio norm}. To some extent, Proposition \ref{estimate width cross ratio norm} is a quantitative version of Theorem \ref{thmBS quasi width}.

\begin{repprop}{estimate width cross ratio norm}
Given any quasisymmetric homeomorphism $\phi$ of $\RP^1$, let $w$ be the width of the convex hull of the graph of $\phi$ in $\partial_\infty\AdS^3$. Then
\begin{equation} \label{estimate width quasisymmetric norm}
%\tanh\left(\frac{||\phi||_{cr}}{4}\right)\leq
\tan(w)\leq \sinh{\left(\frac{||\phi||_{cr}}{2}\right)}\,.
\end{equation}
%Moreover, the inequality is sharp.
\end{repprop}
\begin{proof}
To prove the upper bound on the width, suppose the width of the convex hull $\mathcal{C}$ of $gr(\phi)$ is $w\in(0,\pi/2)$. Let $k=||\phi||_{cr}$. We can find a sequence of pairs $(p_n,q_n)$ such that $d_{\AdS^3}(p_n,q_n)\nearrow w$, with $p_n\in\partial_-\mathcal{C}$, $q_n\in\partial_+\mathcal{C}$. We can assume the geodesic connecting $p_n$ and $q_n$ is orthogonal to $\partial_-\mathcal{C}$ at $p_n$; indeed one can replace $p_n$ with a point in $\partial_-\mathcal{C}$ which maximizes the distance from $q_n$, if necessary (see Remark \ref{discussion width}). Let us now apply isometries $T_n$ so that $T_n(p_n)=p=[\hat p]\in\AdS^3$, for $\hat p=(0,0,1,0)\in\wAdS$, and $T_n(q_n)$ lies on the timelike geodesic through $p$ orthogonal to $P_-=(0,0,0,1)^\perp$. 

The curve at infinity $gr(\phi)$ is mapped by $T_n$ to a curve $gr(\phi_n)$, where $\phi_n$ is obtained by pre-composing and post-composing $\phi$ with M\"obius transformations (this is easily seen from the description of $\isom(\AdS^3)$ as $\PSL(2,\R)\times\PSL(2,\R)$). Hence $\phi_n$ is still quasisymmetric with norm $||\phi_n||_{cr}=||\phi||_{cr}=k$.

\begin{figure}[b]
\centering
\begin{minipage}[c]{.45\textwidth}
\centering
\includegraphics[height=5.5cm]{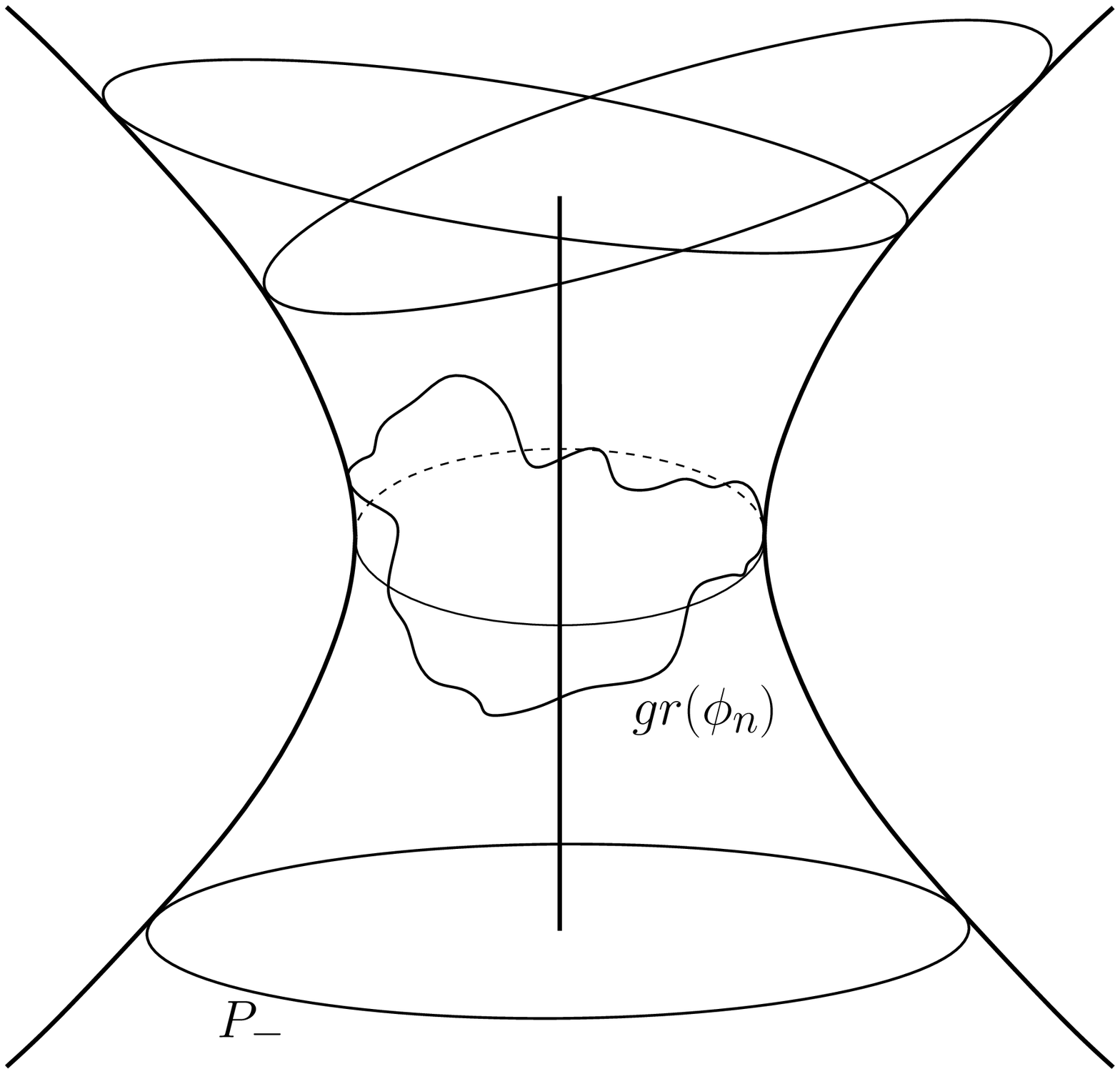} 
%\captionsetup{labelformat=empty}
\end{minipage}%
\hspace{5mm}
\begin{minipage}[c]{.50\textwidth}
\centering
\includegraphics[height=5.5cm]{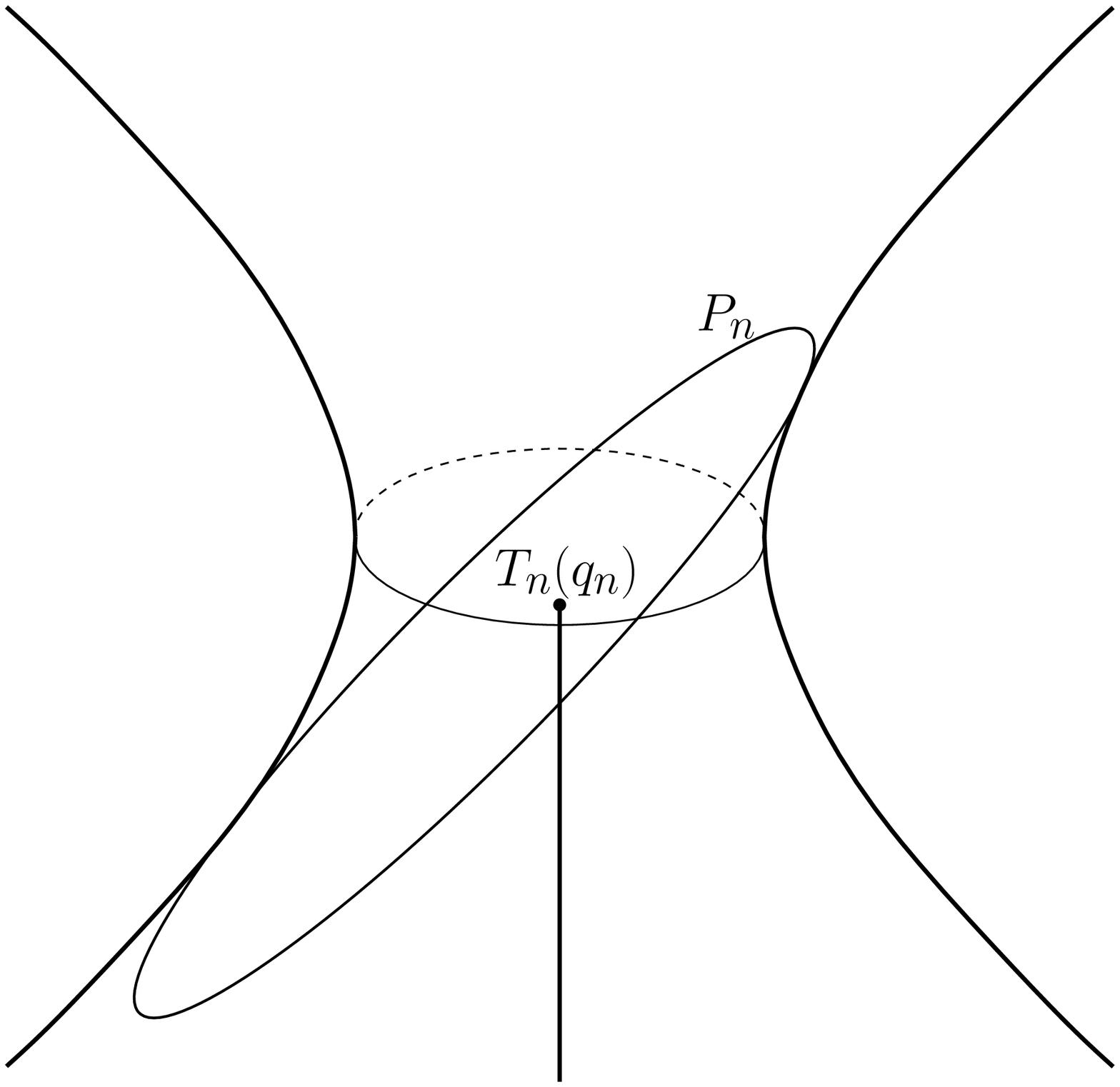}
%\captionsetup{labelformat=empty}
\end{minipage}
\caption{The curves $gr(\phi_n)$ are contained in a bounded region in an affine chart, hence they cannot diverge to a constant map. This is easily seen, for instance, by putting the plane $P_-$ at infinity and observing that the plane $P_n$ has to be spacelike, disjoint from $P_-$, and to intersect the point $T_n(q_n)$ which is in the lower half-plane. \label{fig:boundedregion}}
\end{figure}

It is easy to see that $\phi_n$ cannot converge to a map sending the complement of a point in $\RP^1$ to a single point of $\RP^1$. Indeed, the curves $gr(\phi_n)$ are all contained between $P_-$ and a spacelike plane $P_n$ disjoint from $P_-$, which contains the point $T_n(q_n)$. Moreover the distance of $p$ from $T_n(q_n)\in P_n$ is at most $w$. This shows that the curves $gr(\phi_n)$ all lie in a bounded region in an affine chart of $\AdS^3$; this would not be the case if $\phi_n$ were converging on the complement of one point to a constant map. See Figure \ref{fig:boundedregion}.

 Hence, by the convergence property of $k$-quasisymmetric homeomorphisms (Theorem \ref{Compactness property of quasisymm homeo}), $\phi_n$ converges to a $k$-quasisymmetric homeomorphism $\phi_\infty$, so that $w$ equals the width of the convex hull of $gr(\phi_\infty)$. Let us denote by $\mathcal{C}_\infty$ the convex hull of $gr(\phi_\infty))$.

\begin{figure}[hbtp]
\centering
\begin{minipage}[c]{.45\textwidth}
\centering
\includegraphics[height=5.5cm]{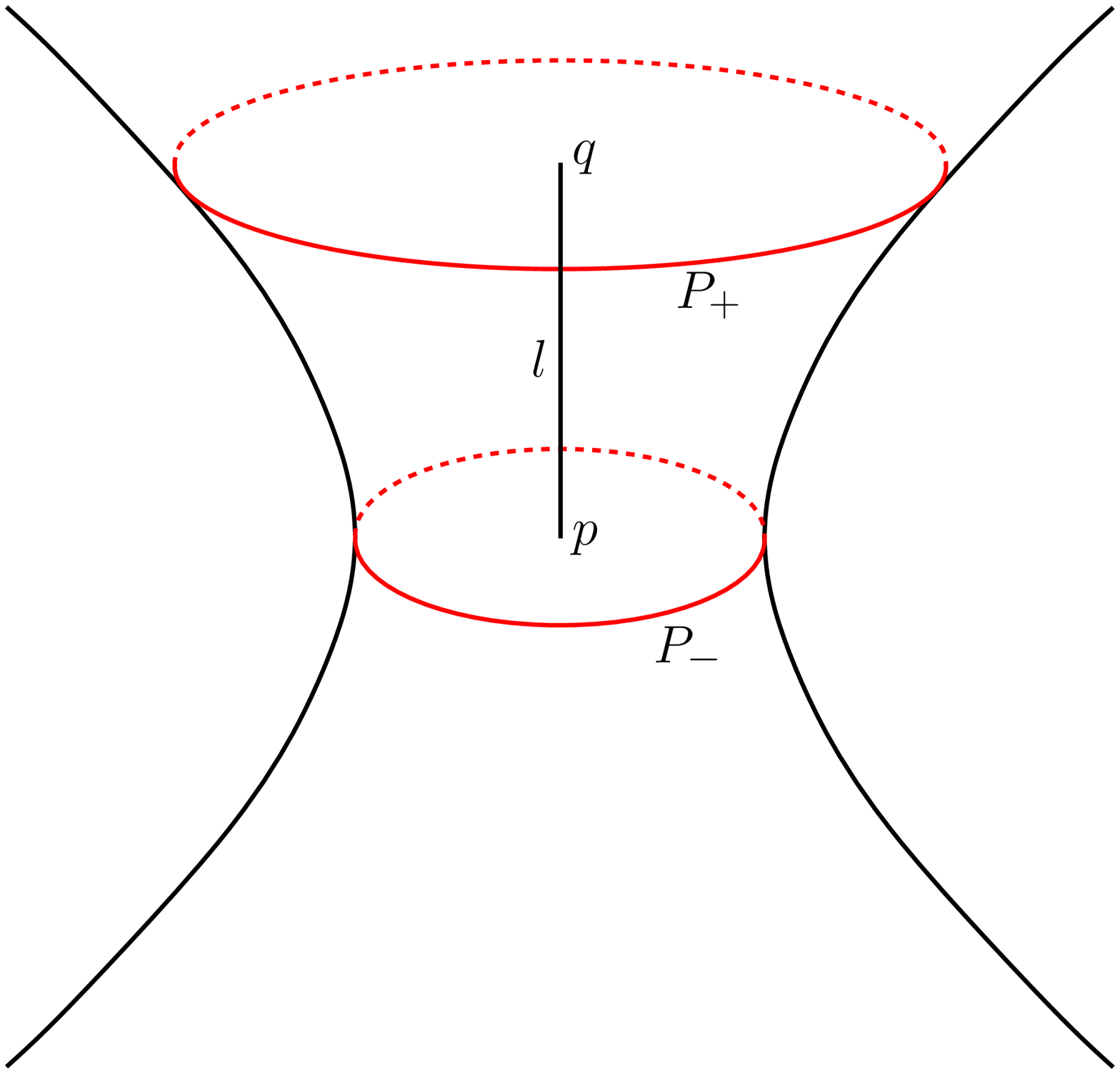} 
%\captionsetup{labelformat=empty}
\caption{The setting of the proof of Proposition ~\ref{estimate width cross ratio norm}.} \label{fig:setting}
\end{minipage}%
\hspace{5mm}
\begin{minipage}[c]{.50\textwidth}
\centering
\includegraphics[height=5.5cm]{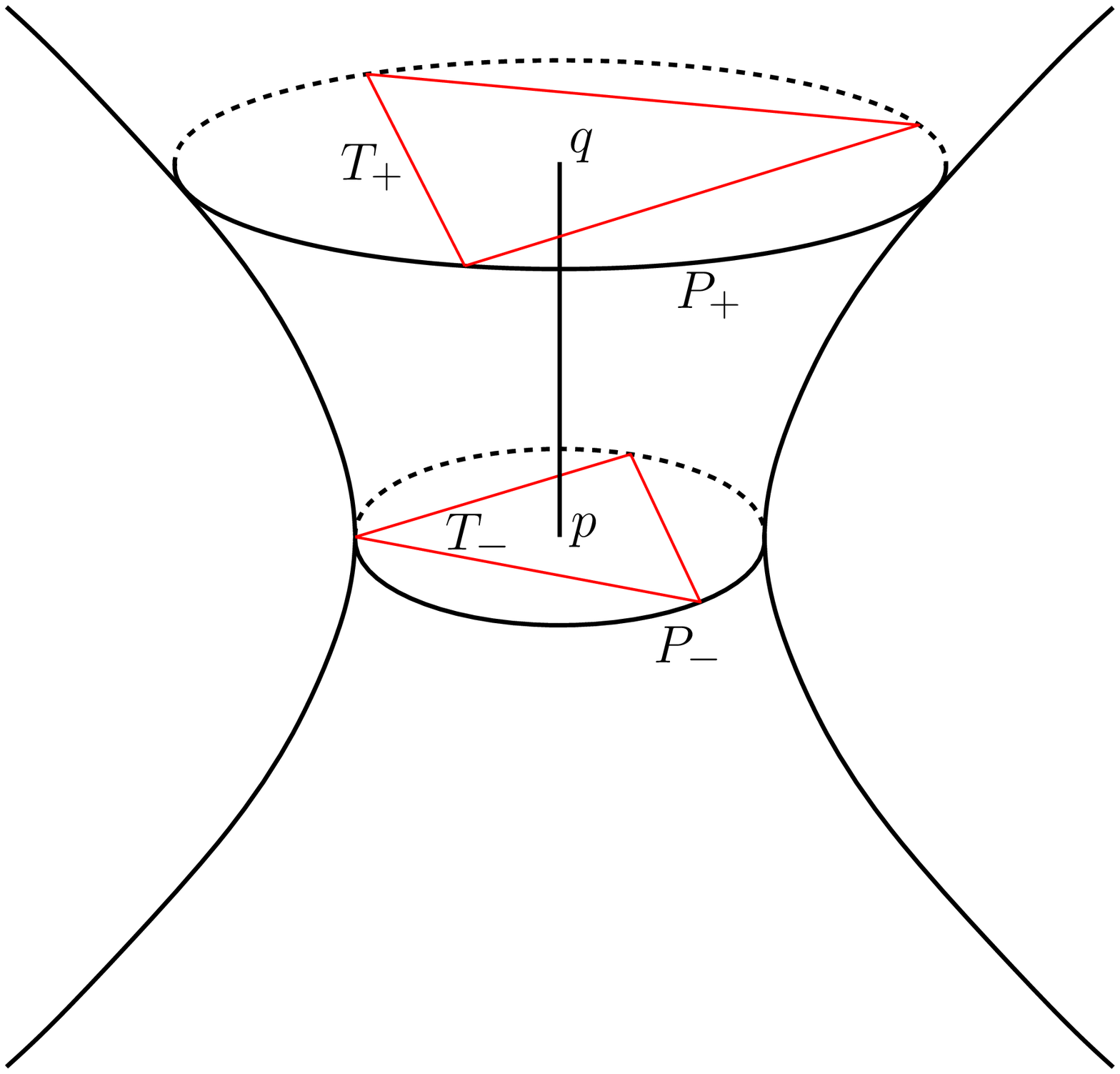}
%\captionsetup{labelformat=empty}
\caption{The point $p$ is contained in the convex envelope of three (or two) points in $\partial_\infty(P_-)$; analogously $q$ in $P_+$.} \label{fig:triangles}
\end{minipage}
\end{figure}

%In the rest of the proof, we derive the estimate (\ref{estimate width quasisymmetric norm}).
We will mostly refer to the coordinates in the affine chart $\{x^3\neq0\}$, namely $(x,y,z)=(x^1/x^3,x^2/x^3,x^4/x^3)$. Our assumption is that the point $p$ has coordinates $(0,0,0)$ and $P_-=\{(x,y,0):x^2+y^2<1\}$ is the totally geodesic plane through $p$ which is a support plane for $\partial_-\mathcal{C}_\infty$. The geodesic line $l$ through $p$ orthogonal to $P_-$ is $\{(0,0,z)\}$. By construction, the width of $\mathcal{C}_\infty$ equals $d_{\AdS^3}(p,q)$, where $q=(0,0,h)=l\cap \partial_+\mathcal{C}_\infty$. It is then an easy computation to show that $h=\tan w$. Hence the plane 
$$P_+=\{(x,y,h):(x,y)\in\R^2,x^2+y^2<1+h^2\}\,,$$
which is the plane orthogonal to $l$ through $q$, is a support plane for $\partial_+\mathcal{C}_\infty$. See Figure \ref{fig:setting}.

Since $\partial_-\mathcal{C}_\infty$ and $\partial_+\mathcal{C}_\infty$ are pleated surfaces, $\partial_-\mathcal{C}_\infty$ contains an ideal triangle $T_-$, such that $p\in T_-$ (possibly $p$ is on the boundary of $T_-$). The ideal triangle might also be degenerate if $p$ is contained in an entire geodesic, but this will not affect the argument. Hence we can find three geodesic half-lines in $P_-$ connecting $p$ to $\partial_\infty\AdS^3$ (or an entire geodesic connecting $p$ to two opposite points in the boundary, if $T_-$ is degenerate). Analogously we have an ideal triangle $T_+$ in $P_+$, compare Figure \ref{fig:triangles}. The following sublemma will provide constraints on the position the half-geodesics in $P_+$ can assume. See Figure \ref{fig:sector} and \ref{fig:sectorabove} for a picture of the ``sector'' described in Lemma \ref{lemma sector}.

\begin{sublemma} \label{lemma sector}
Suppose $\partial_-\mathcal{C}_\infty\cap P_-$ contains a half-geodesic $$g=\left\{t(\cos\theta,\sin\theta,0):t\in[0,1)\right\}$$ from $p$, asymptotic to the point at infinity $\eta=(\cos\theta,\sin\theta,0)$. Then $\partial_+\mathcal{C}_\infty\cap P_+$ must be contained in $P_+\setminus S(\eta)$, where $S(\eta)$ is the sector $\{x\cos\theta+y\sin\theta>1\}$.
\end{sublemma}

\begin{figure}[htbp]
\centering
\begin{minipage}[c]{.40\textwidth}
\centering
\includegraphics[height=5.5cm]{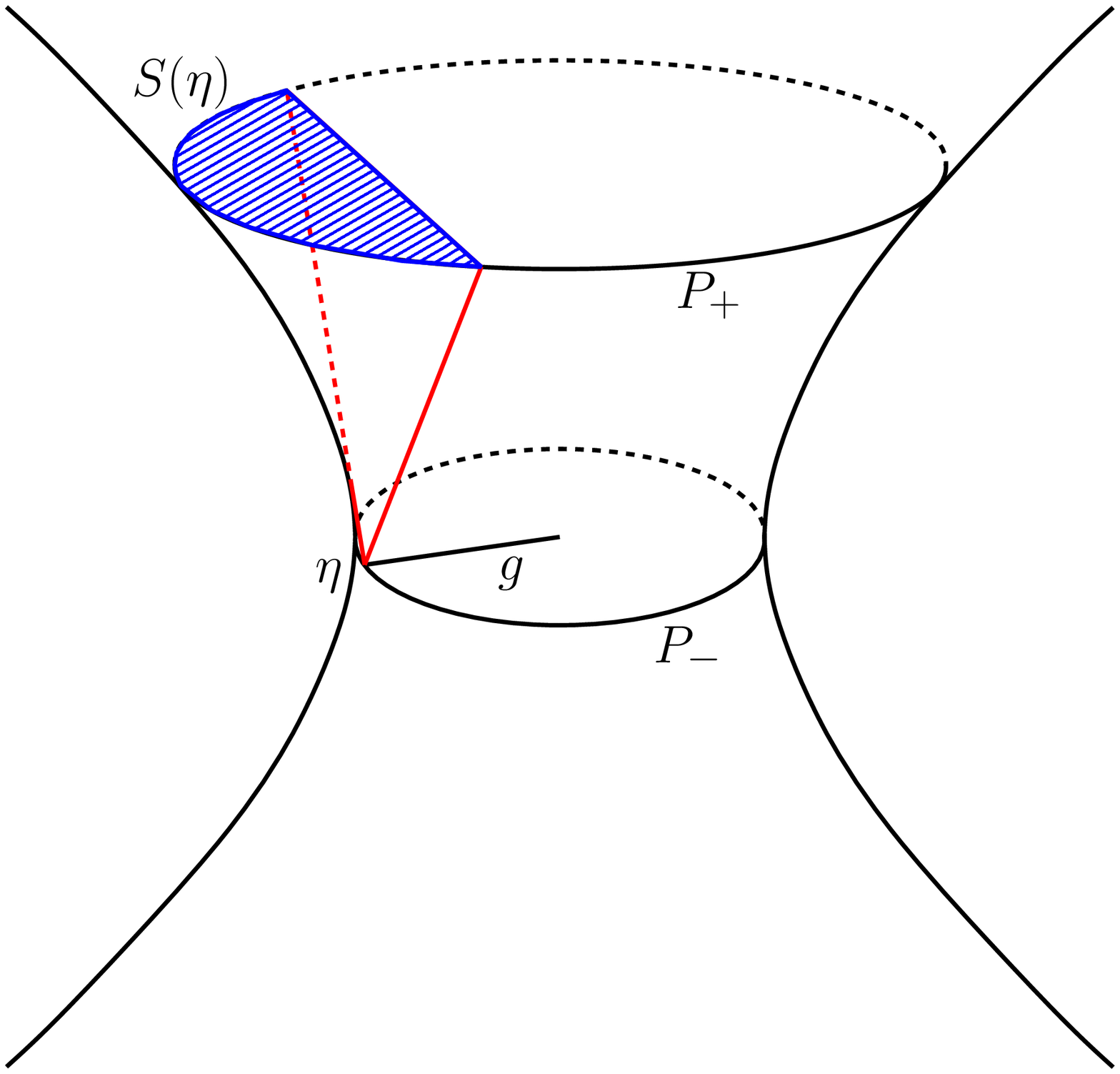}
%\captionsetup{labelformat=empty}
\caption{The sector $S(\eta)$ as in Sublemma \ref{lemma sector}.} \label{fig:sector}
\end{minipage}%
\hspace{5mm}
\begin{minipage}[c]{.50\textwidth}
\centering
\includegraphics[height=4.5cm]{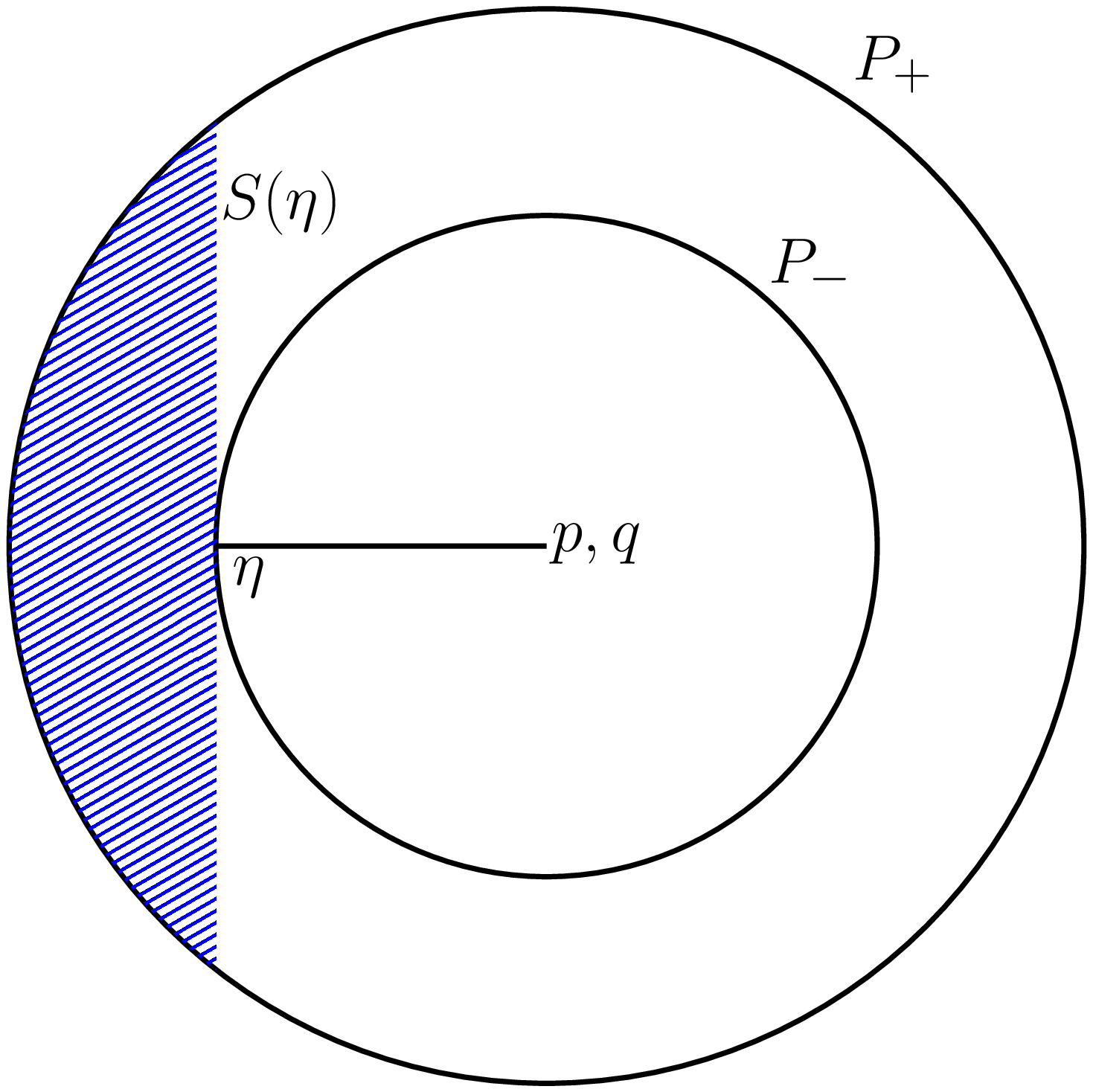}
%\captionsetup{labelformat=empty}
\caption{The $(x,y)$-plane seen from above. The sector $S(\eta)$ is bounded by the chord in $P_+$ tangent to the concentric circle, which projects vertically to $P_-$.} \label{fig:sectorabove}
\end{minipage}
\end{figure}

\begin{proof}
The computation will be carried out using the coordinates of the double cover $\wAdS$ of $\AdS^3$. It suffices to check the assertion when $\theta=\pi$, since in the statement there is a rotational symmetry along the vertical axis. The half-geodesic $g$ is parametrized by $$g(t)=[\sinh(t),0,\cosh(t),0]\,,$$ for $t\in(-\infty,0]$. Since the width is less than $\pi/2$, every point in $\partial_+\mathcal{C}_\infty\cap P_+$ must lie in the region bounded by $P_-$ and the dual plane $g(t)^\perp$. Indeed for every $t$, $g(t)^\perp$ is the locus of points at timelike distance $\pi/2$ from $g(t)$. We have $$P_+=\left\{[\cos(\alpha)\sinh(r),\sin(\alpha)\sinh(r),\cos(w)\cosh(r),\sin(w)\cosh(r)]:r\geq 0,\alpha\in[0,2\pi)\right\}.$$ 
Hence the intersection $P_+\cap g(t)^\perp$ is given by imposing that a point of $P_+$ has zero product with points $g(t)$, which gives the condition $$\sinh(t)\cos(\alpha)\sinh(r)=\cosh(t)\cos(w)\cosh(r)\,.$$
Thus points in the intersection are of the form (in the affine coordinates of $\{x^3\neq0\}$): $$\left(\frac{1}{\tanh(t)},\frac{\tan(\alpha)}{\tanh(t)},\tan(w)\right)\,.$$
 Therefore, points in $\partial_+\mathcal{C}_\infty\cap P_+$ need to have $x\geq 1/\tanh(t)$, and since this holds for every $t\leq0$, we have $x\geq -1$.
\end{proof}

By the Sublemma \ref{lemma sector}, if $p$ is contained in the convex envelope of three points $\eta_1,\eta_2,\eta_3$ in $\partial_\infty(P_-)$, then any point at infinity of $\partial_+\mathcal{C}_\infty\cap P_+$ is necessarily contained in $P_+\setminus(S(\eta_1)\cup S(\eta_2)\cup S(\eta_3))$. We will use this fact to choose two pairs of points, $\eta,\eta'$ in $\partial_\infty(P_-)$ and $\xi,\xi'$ in $\partial_\infty(P_+)$, in a convenient way. This is the content of next sublemma. See Figure \ref{fig:connectedcomponents}.

\begin{sublemma} \label{sublemma eucl}
Suppose $p$ is contained in the convex envelope of three points $\eta_1,\eta_2,\eta_3$ in $\partial_\infty(P_-)$. Then $gr(\phi_\infty)$ must contain (at least) two points $\xi,\xi'$ of $\partial_\infty(P_+)$ which lie in different connected components of $\partial_\infty(P_+)\setminus(S(\eta_1)\cup S(\eta_2)\cup S(\eta_3))$.
\end{sublemma}
\begin{proof}
The proof is simple 2-dimensional Euclidean geometry. Recall that the point $q$, which is the ``center'' of the plane $P_+$, is in the convex hull of $gr(\phi_\infty)$. If the claim were false, then one connected component of $\partial_\infty(P_+)\setminus(S(\eta_1)\cup S(\eta_2)\cup S(\eta_3))$ would contain a sector $S_0$ of angle $\geq \pi$. But then the points $\eta_1,\eta_2,\eta_3$ would all be contained in the complement of $S_0$. This contradicts the fact that $p$ is in the convex hull of $\eta_1,\eta_2,\eta_3$.
\end{proof}

\begin{remark}
If $p$ is in the convex envelope of only two points at infinity, which means that $P_-$ contains an entire geodesic, the previous statement is simplified, see Figure \ref{fig:entireline}.
\end{remark}

\begin{figure}[htbp]
\centering
\begin{minipage}[c]{.50\textwidth}
\centering
\includegraphics[width=4.2cm]{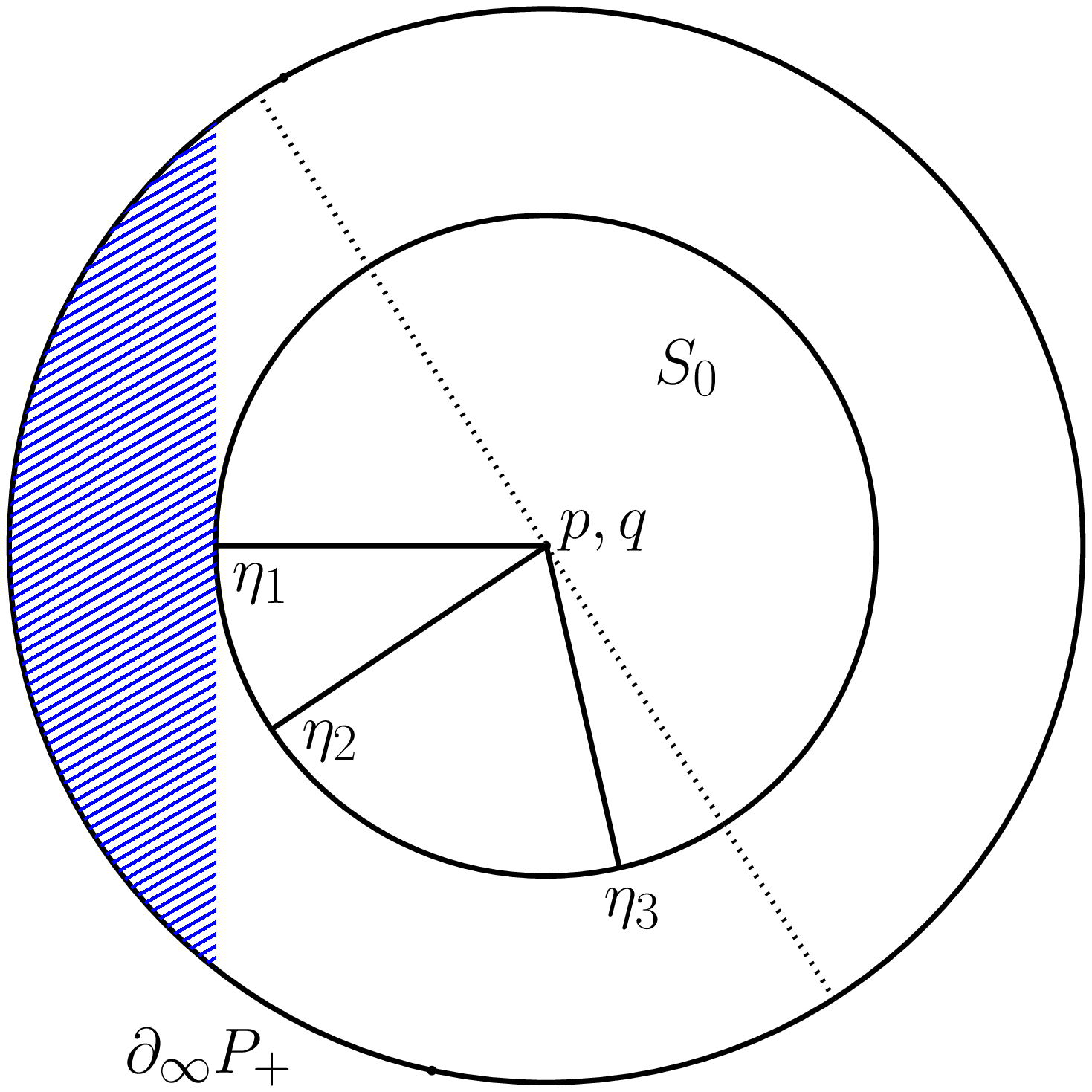}
%\captionsetup{labelformat=empty}
\caption{The proof of Sublemma \ref{sublemma eucl}. Below, the choice of points $\eta,\eta',\xi,\xi'$.} \label{fig:connectedcomponents}
\vspace{2mm}
\includegraphics[width=4.3cm]{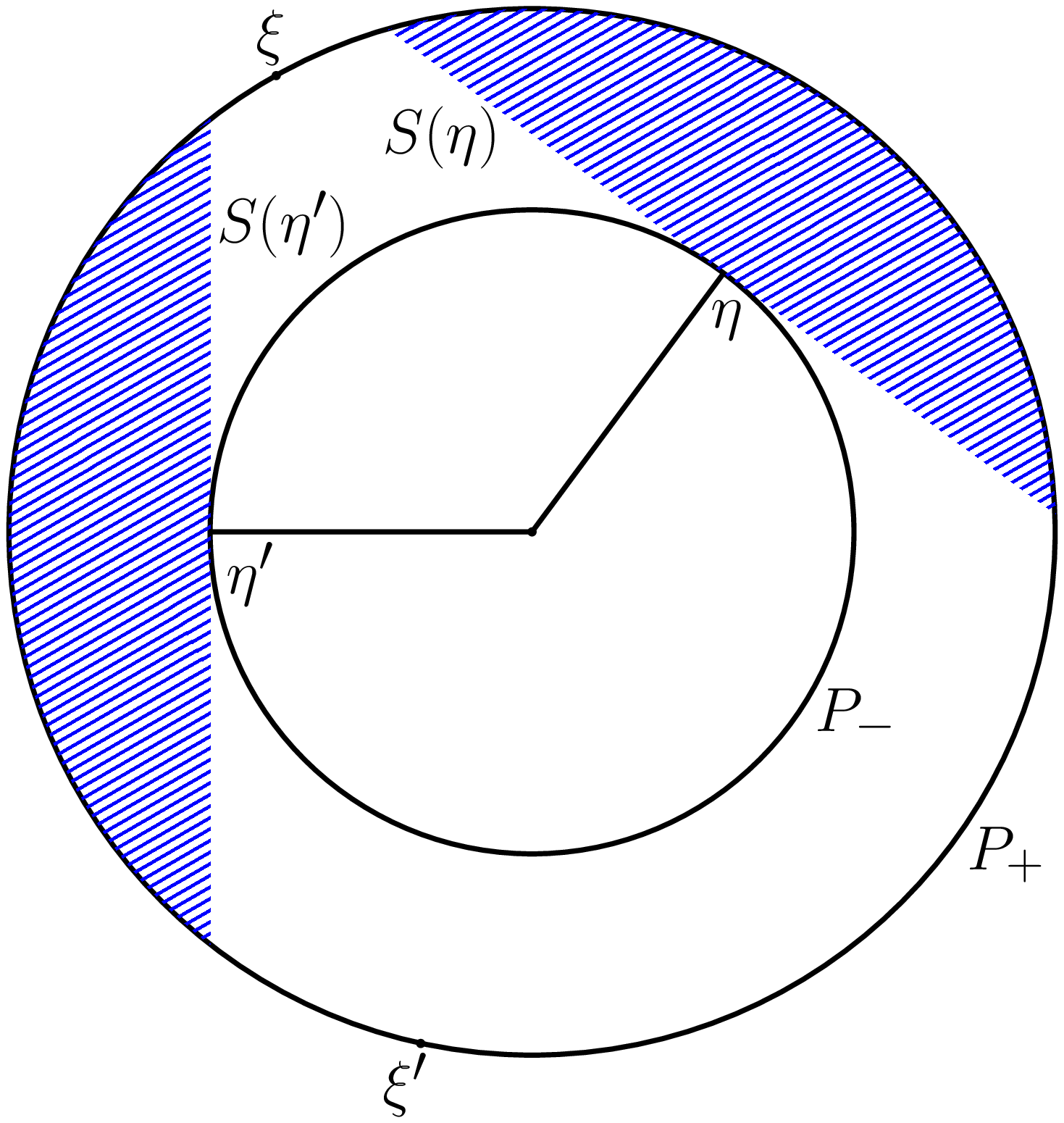}
\end{minipage}%
%\hspace{5mm}
\begin{minipage}[c]{.45\textwidth}
\centering
\includegraphics[height=6.3cm]{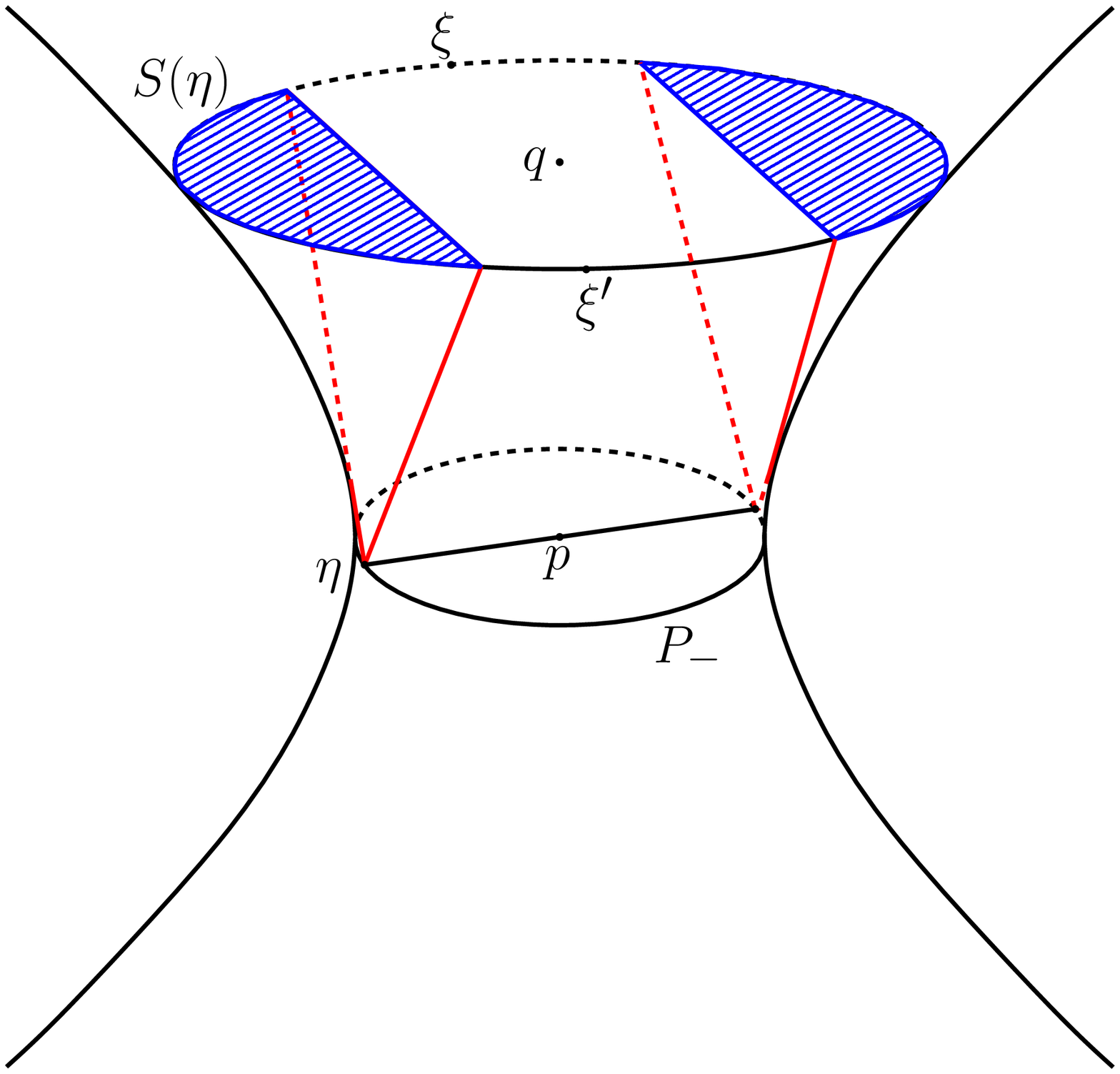}
%\captionsetup{labelformat=empty}
\caption{%If $p$ is contained in an entire geodesic line contained in $P_-$, it is easier to see that $q$ is in the convex hull of at least two point $\xi$ and $\xi'$, in different connected components.
The same statement of Sublemma \ref{sublemma eucl} is simpler if $p$ is contained in an entire geodesic line contained in $P_-$.} \label{fig:entireline}
\end{minipage}
\end{figure}

%if $\partial_-\mathcal{C}_\infty\cap P_-$ contains three half-geodesics whose endpoints at infinity are $\eta_1,\eta_2,\eta_3$, then an ideal triangle contained in $\partial_+\mathcal{C}_\infty\cap P_+$ is necessarily contained in $P_+\setminus(S(\eta_1)\cup S(\eta_2)\cup S(\eta_3))$. It follows that we can choose two half-geodesics from $p$ in $P_-$ (with endpoints at infinity $\eta$ and $\eta'$) and two half-geodesics from $q$ in $P_+$ (with endpoints $\xi$ and $\xi'$) such that $\xi$ and $\xi'$ belong to two different components of $\partial_\infty(P_+)\setminus(S(\eta)\cup S(\eta'))$. 

Let us now choose two points $\eta,\eta'\in \partial_\infty(P_-)$ among $\eta_1,\eta_2,\eta_3$, and $\xi,\xi'\in \partial_\infty(P_+)$ in such a way that $\xi$ and $\xi'$ lie in two different connected components of $\partial_\infty(P_+)\setminus(S(\eta_1)\cup S(\eta_2))$. The strategy will be to use this quadruple to show that the cross-ratio distortion of $\phi_\infty$ is not too small, depending on the width $w$. However, such quadruple is not symmetric in general. Hence $\xi'$ will be replaced later by another point $\xi''$. First we need some tool to compute the left and right projections to $\partial_\infty\Hyp^2$ of the chosen points.

We use the plane $P_-$ to identify $\partial_\infty\AdS^3$ with $\partial_\infty\Hyp^2\times \partial_\infty\Hyp^2$. Let $\pi_l$ and $\pi_r$ denote left and right projection to $\partial_\infty(P_-)$, following the left and right ruling of $\partial_\infty\AdS^3$. In what follows, angles like $\theta_l$, $\theta_r$ and similar symbols will always be considered in $(-\pi,\pi]$.

\begin{sublemma} \label{lemma left right projection}
Suppose $\xi\in\partial_\infty(P_+)$, where the length of the timelike geodesic segment orthogonal to $P_-$ and $P_+$ is $w$. If $\pi_l(\xi)=(\cos(\theta_l),\sin(\theta_l),0)$, then $\pi_r(\xi)=(\cos(\theta_l-2w),\sin(\theta_l-2w),0)$.
\end{sublemma}
\begin{proof}
By the description of the left ruling (see Section \ref{subsection Hyp AdS}), recalling $h=\tan(w)$, it is easy to check that
\begin{align*}
\xi=&(\cos(\theta_l),\sin(\theta_l),0)+h(\sin(\theta_l),-\cos(\theta_l),1)=(\cos(\theta_l)+h\sin(\theta_l),\sin(\theta_l)-h\cos(\theta_l),h) \\
=&(\sqrt{1+h^2}\cos(\theta_l-w),\sqrt{1+h^2}\sin(\theta_l-w),h)\,.
\end{align*}
By applying the same argument to the right projection, the claim follows.
\end{proof}

We can assume $\eta'=(-1,0,0)$. We shall adopt in this part the complex notation, i.e. $\Hyp^2$ is thought of in the disc model as a subset of $\C$, where $\C$ is identified to the plane $\{z=0\}$ in the affine chart. In this way, $\eta'$ corresponds to $(-1,-1)\in\partial_\infty\Hyp^2\times \partial_\infty\Hyp^2$. Let $\eta=(e^{i\theta_0},e^{i\theta_0})$; by symmetry, we can assume $\theta_0\in[0,\pi)$; in this case we need to consider the point $\xi=(e^{i\theta_l},e^{i\theta_r})$ constructed above, with $\theta_r\in[\theta_0,\pi)$. More precisely, Sublemma \ref{lemma left right projection} shows $\theta_r=\theta_l-2w$; by Sublemma \ref{lemma sector} we must have $\theta_l-w\notin(\theta_0-w,\theta_0+w)\cup(\pi-w,\pi)\cup(-\pi,-\pi+w)$ and thus, by choosing $\xi$ in the correct connected component (i.e. switching $\xi$ and $\xi'$ if necessary), necessarily $\theta_l\in [\theta_0+2w,\pi]$ (see Figure \ref{fig:xi}). 

We remark again that the quadruple $Q=\pi_l(\xi',\eta,\xi,\eta')$ will not be symmetric in general, so we need to consider a point $\xi''$ instead of $\xi'$ so as to obtain a symmetric quadruple. However, if $\theta_0\in(-\pi,0)$, then one would consider the point $\xi'$ in the connected component having $\theta_r\in(-\pi,\theta_0)$ - and then a point $\xi''$ in the other connected component so as to have a symmetric quadruple - and obtain the same final estimate.

So let $\xi''=(e^{i\theta''_l},e^{i\theta''_r})$ be a point on $gr(\phi)$ so that the quadruple $Q=\pi_l(\xi'',\eta,\xi,\eta')$ is symmetric; we are going to compute the cross-ratio of $\phi(Q)=\pi_r(\xi'',\eta,\xi,\eta')$. However, in order to avoid dealing with complex numbers, we first map $\partial_\infty\Hyp^2=\partial_\infty(P_-)$ to $\RP^1\cong \R\cup\left\{\infty\right\}$ using the M\"obius transformation
$$z\mapsto\frac{z-1}{i(z+1)}$$
which maps $e^{i\theta}$ to $\tan (\theta/2)\in\R$ if $\theta\neq\pi$, and $-1$ to $\infty$. We need to compute
\begin{equation} \label{stimare cross ratio norm}
\left|\ln\left|cr(\phi(Q))\right|\right|=\left|\ln\left|{\frac{\tan(\theta_r/2)-\tan(\theta_0/2)}{\tan(\theta_0/2)-\tan(\theta''_r/2)}}\right|\right|
\end{equation}
and in particular we want to show this is uniformly away from 1. By construction $\theta_r<\theta_l$ (see also Figure \ref{fig:slopes}), and since $P_-$ does not disconnect $gr(\phi)$, also $\theta''_r<\theta''_l$. Hence we have 
\begin{equation} \label{stimare cross ratio norm 2} \tan(\theta_0/2)-\tan(\theta''_r/2)\geq \tan(\theta_0/2)-\tan(\theta''_l/2)\,.
\end{equation}
\noindent The condition that $(\theta_l'',\theta_0,\theta_l,\infty)$ forms a symmetric quadruple translates on $\R$ to the condition that
\begin{equation} \label{stimare cross ratio norm 3}
\tan(\theta_0/2)-\tan(\theta''_l/2)=\tan(\theta_l/2)-\tan(\theta_0/2)\,.
\end{equation}
\noindent Using (\ref{stimare cross ratio norm 2}) and (\ref{stimare cross ratio norm 3}) in the argument of the logarithm in (\ref{stimare cross ratio norm}), we obtain:
$$\frac{\tan(\theta_r/2)-\tan(\theta_0/2)}{\tan(\theta_0/2)-\tan(\theta''_r/2)}\leq\frac{\tan((\theta_l/2)-w)-\tan(\theta_0/2)}{\tan(\theta_l/2)-\tan(\theta_0/2)}=:S(\theta_l).$$
%By Lemma \ref{lemma sector}, we must have $\theta_l\in [\theta_0+2w,\pi]$. (see picture FIGURA)

\begin{figure}[b]
\centering
\begin{minipage}[c]{.45\textwidth}
\centering
\includegraphics[height=5.8cm]{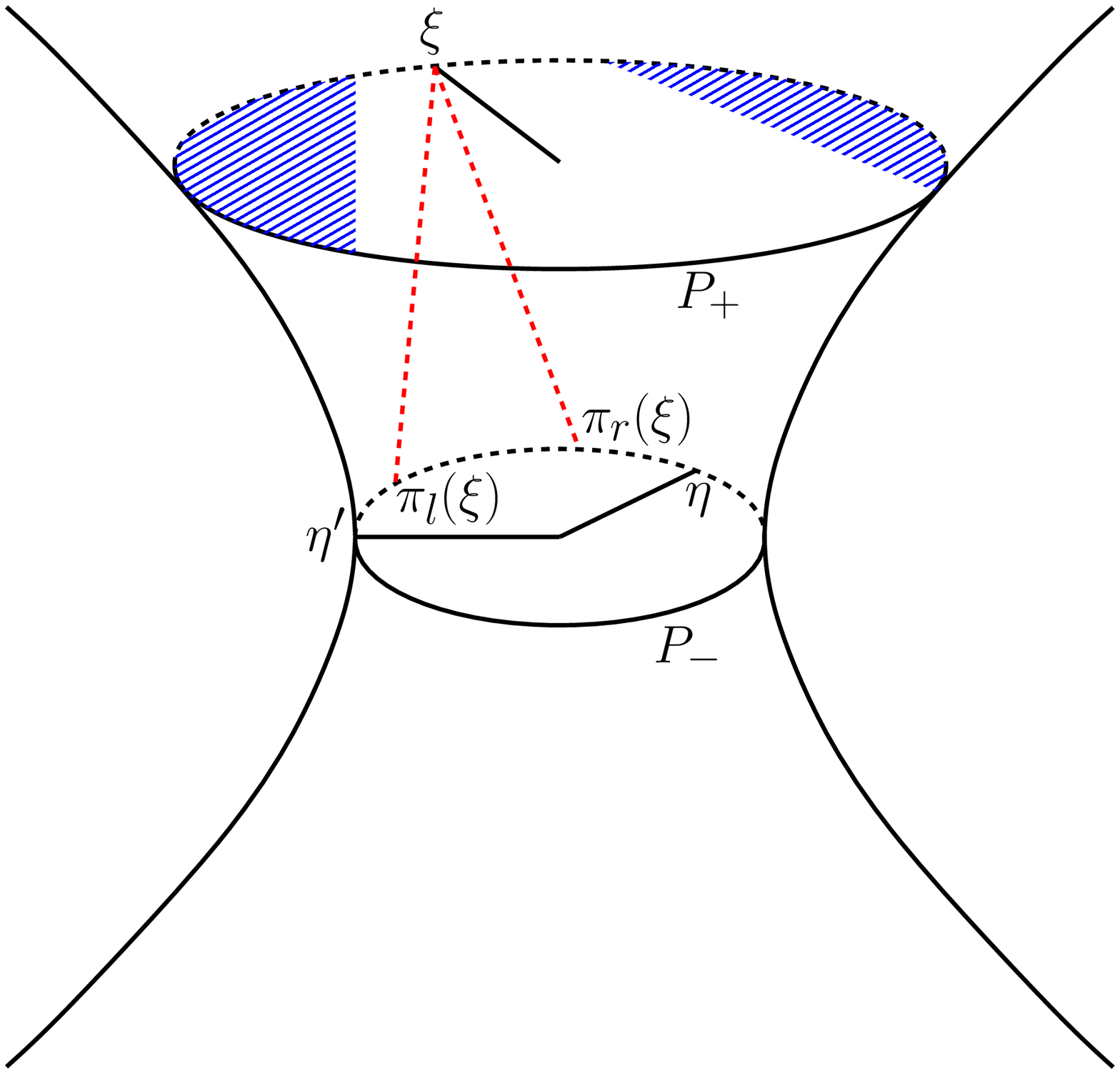} 
%\captionsetup{labelformat=empty}
\caption{The choice of points $\eta,\xi,\eta'$ in $\partial_\infty\AdS^3$, endpoints at infinity of geodesic half-lines in the boundary of the convex hull.} \label{fig:xi}
\end{minipage}%
\hspace{5mm}
\begin{minipage}[c]{.50\textwidth}
\centering
\includegraphics[height=6.2cm]{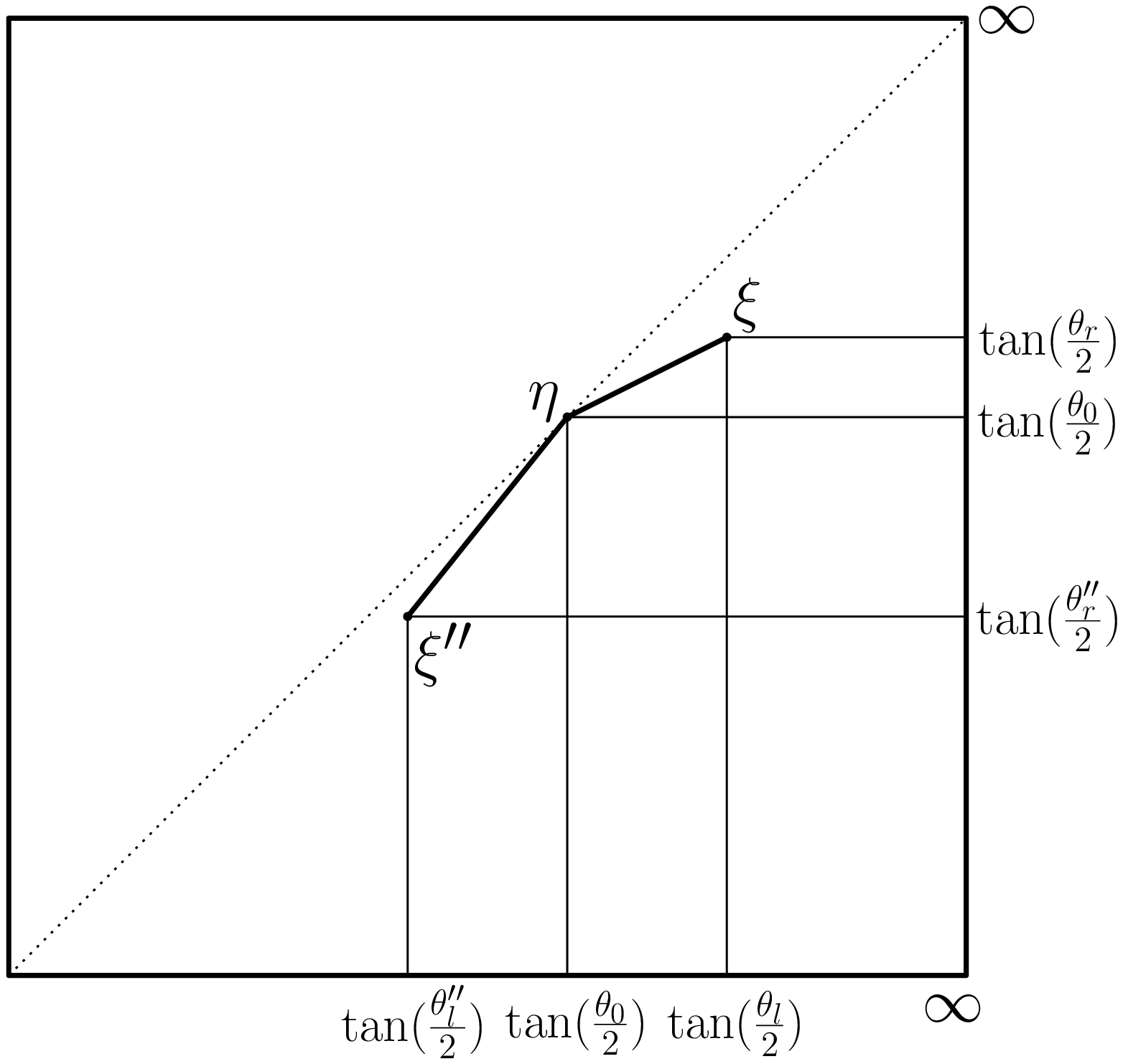}
%\captionsetup{labelformat=empty}
\caption{We give an upper bound on the ratio between the slopes of the two thick lines. The dotted line represents the plane $P_-$.} \label{fig:slopes}
\end{minipage}
\end{figure}

Note that $S(\theta_l)<1$ on $[\theta_0+2w,\pi]$ and $S(\theta_l)\rar 0$ when $\theta_l\rar\theta_0+2w$ or $\theta_l\rar\pi$: this corresponds to the fact that $gr(\phi_\infty)$ tends to contain a lightlike segment. On the other hand $S(\theta_l)$ is positive on $[\theta_0+2w,\pi]$ and the maximum $S_{max}$ is achieved at some interior point of the interval. A computation gives
$$\left| cr(\phi(Q))\right| \leq S_{max}=\left(\frac{\cos(\theta_0/2+w)}{\cos(\theta_0/2)+\sin(w)}\right)^2.$$
The RHS quantity depends on $\theta_0$, but is maximized on $[0,\pi-2w]$ for $\theta_0=0$, where it assumes the value $(1-\sin(w))/(1+\sin(w))$. This gives
$$e^{||\phi_\infty||_{cr}}\geq\left|\frac{1}{cr(\phi(Q))}\right| \geq \frac{1+\sin(w)}{1-\sin(w)}\,.$$
From this we deduce
$$\sin(w)\leq\frac{e^{||\phi_\infty||_{cr}}-1}{e^{||\phi_\infty||_{cr}}+1}=\tanh{\frac{||\phi_\infty||_{cr}}{2}}$$
or equivalently $$\tan(w)\leq \sinh{\frac{||\phi_\infty||_{cr}}{2}}\,.$$
Since $||\phi_\infty||_{cr}\leq ||\phi||_{cr}$, the proof is concluded.
\end{proof}

By using a very similar analysis, though simpler, we can prove an inequality in the converse direction. 

\begin{repprop}{estimate width cross ratio norm below}
Given any quasisymmetric homeomorphism $\phi$ of $\RP^1$, let $w$ be the width of the convex hull of the graph of $\phi$ in $\partial_\infty\AdS^3$. Then
$$\tanh\left(\frac{||\phi||_{cr}}{4}\right)\leq\tan w\,.$$
\end{repprop}

\begin{proof}
Suppose $||\phi||_{cr}>k$. Then we can find a quadruple of symmetric points $Q$ such that $|cr(\phi(Q))|=e^{k}$. Consider the points $\xi',\eta,\xi,\eta'$ on $\partial_\infty\AdS^3$ such that their left and right projection are $Q$ and $\phi(Q)$, respectively. 

%We use the same notation as above, for example $\xi=(e^{i\theta_l},e^{i\theta_r})$. 

Recall that the isometries of $\AdS^3$ act on $\partial_\infty\AdS^3\cong \RP^1\times \RP^1$ as a pair of M\"obius transformations, therefore they preserve the cross-ratio of both $Q$ and $\phi(Q)$. Thus we can suppose $Q=(-1,0,1,\infty)$ and $\phi(Q)=(-e^{k/2},0,e^{-k/2},\infty)$ when the quadruples are regarded as composed of points on $\R\cup\left\{\infty\right\}$. 

Passing to the coordinates in $S^1$ (by the map $\theta\in S^1\mapsto\tan(\theta/2)\in\R$) for this quadruple of points at infinity, it is easy to see that - in the affine chart $\{x^3\neq 0\}$ - the position of the four points has an order 2 symmetry obtained by rotation around the $z$-axis. See Figure \ref{fig:opposite}. This is ensured by the special renormalization chosen for $Q$ and $\phi(Q)$.

Hence the 
 geodesic line $g_1$ with endpoints at infinity $\eta$ and $\eta'$ is contained in the plane $P_-$ as in the first part of the proof. More precisely, in the usual affine chart $\{x^3\neq 0\}$, 
 $$g_1=\{(\tanh(t),0,0):t\in\R\}\,.$$ 
The geodesic line $g_2$ connecting $\xi$ and $\xi'$ has the form
$$g_2(s)=\left\{\left(\frac{\cos(\alpha)\tanh(s)}{\cos(w')},\frac{\sin(\alpha)\tanh(s)}{\cos(w')},\tan(w')\right): s\in\R\right\}\,.$$
 The lines $g_1$ and $g_2$ are in the convex hull of $gr(\phi)$ and have the common orthogonal segment $l$ which lies in the $z$-axis in the usual affine chart (Figure \ref{fig:opposite}), the feet of $l$ being achieved for $t=0$ and $s=0$. 
 
 The distance between $g_1$ and $g_2$ is achieved along this common orthgonal geodesic and its value is $w'$. Recalling Sublemma \ref{lemma left right projection} and the computation in its proof, we find $\alpha=\theta_l-w'=\pi/2-w'$ and $\theta_r=\theta_l-2w'$. Since $\tan(\theta_r/2)=e^{-k/2}$ and $\theta_l=\pi/2$, one can compute
$$w'=\pi/4-\arctan(e^{-k/2})\,.$$
It follows that
$$\tan w\geq\tan w'=\frac{1-e^{-k/2}}{1+e^{-k/2}}=\tanh\left(\frac{k}{4}\right).$$
Since this is true for an arbitrary $k\leq||\phi||_{cr}$, the inequality $$\tan w\geq\tanh\left(\frac{||\phi||_{cr}}{4}\right)$$ holds.
\end{proof}

\begin{figure}[h!]
\centering
\includegraphics[height=6.5cm]{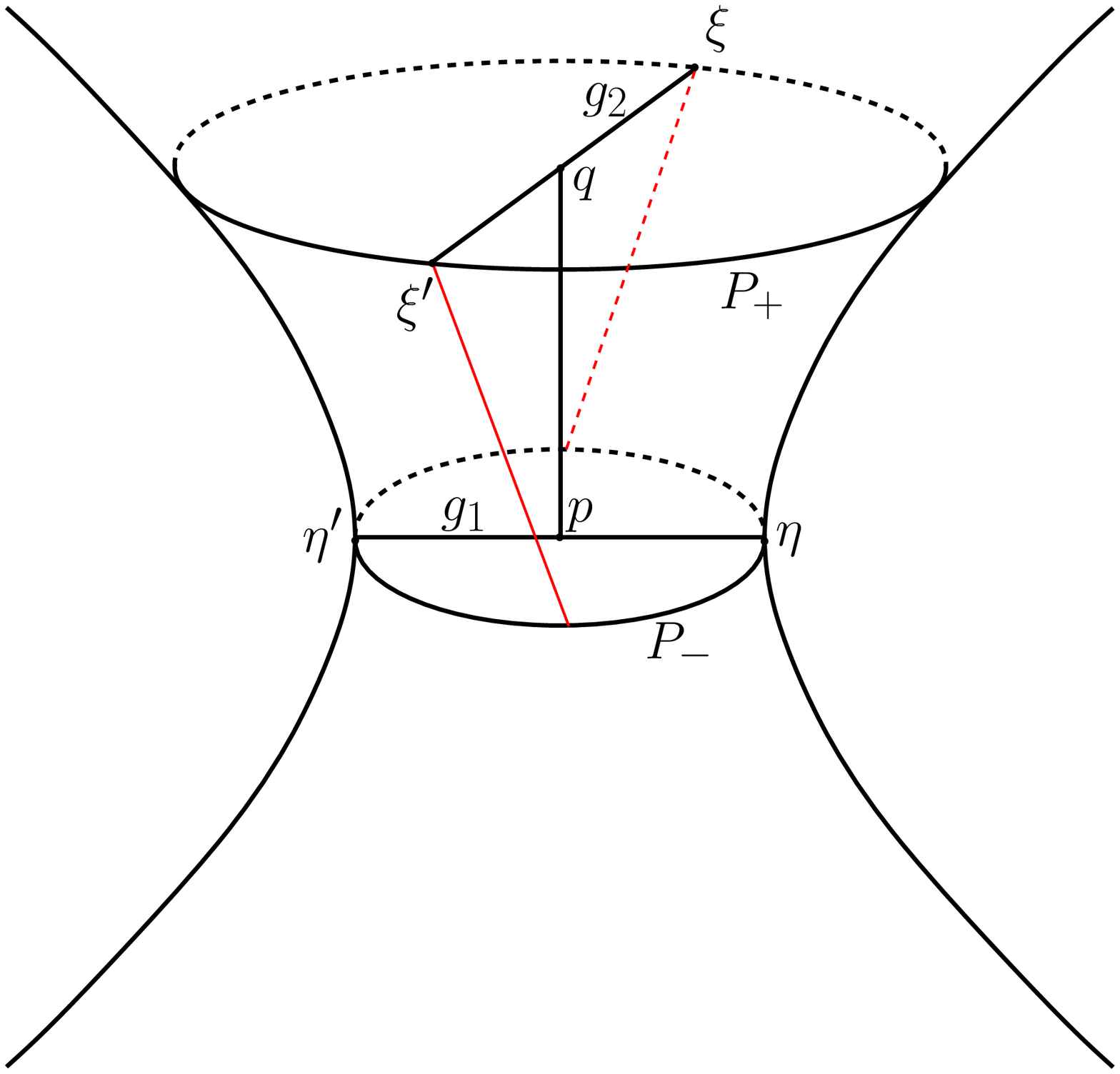}
\caption{The distance between the two lines $g$ and $g'$ is achieved along the common orthogonal geodesic. \label{fig:opposite}}
\end{figure}

%can be computed to be 
%$$S_{max}=\frac{(\tan(w)(1+\tan(\theta_0/2)^2)-\sqrt{(1+\tan(w)^2)(1+\tan(\theta_0/2)^2)})}{(\tan(w)(1+\tan(\theta_0/2)^2)+\sqrt{(1+\tan(w)^2)(1+\tan(\theta_0/2)^2)})}$$

%To conclude, we prove the converse estimate:
%\begin{prop}
%\end{prop}
%\begin{proof}
%Suppose $||\phi||_{cr}> M$. Then we can find a quadruple of points $Q=(z_1,z_2,z_3,z_4)$ such that $\frac{cr(\phi(Q))}{cr(Q)}>e^M$ or $\frac{cr(\phi(Q))}{cr(Q)}<e^{-M}$. By composing with M\"obius transformations, we can assume $Q=(\infty,-1,0,z)$, so that $m=cr(Q)$, and $\phi(Q)=(\infty,-1,0,\phi(m))$, with $\phi(z)=cr(\phi(Q))$. This means that in $\partial_\infty\AdS^3$, $gr(\phi)$ passes through the points $(\infty,\infty),(0,0),(-1,-1)\in\R P^1\times \R P^1$ (we denote $[a:b]=b/a$ in $\R P^1$). The coordinates of the projection of these points to the affine chart $x_3=1$ are respectively:
%\begin{align*}
%\eta_p=(-1,0,1,0) \\
%\eta'_p=(1,0,1,0) \\
%\zeta_p=(0,-1,1,0) 
%\end{align*}
%\noindent where the pair $()$
%\end{proof}

%------------------------------------------------------

%------------------------------------------------------

\section{Maximal surfaces with small principal curvatures} \label{sec estimate small}

Let $S$ be a maximal surface in $\AdS^3$. Let $P_-$ be a spacelike plane which does not intersect the convex hull. We want to use the fact that the function $u(x)=\sin d_{\AdS^3}(x,P_-)$, satisfies the equation  
\begin{equation}
\Delta_S u-2u=0\,. \tag{\ref{lap 2u ads}}
\end{equation}  given in Proposition \ref{formule hess lap ads}. This will enable us to use Equation \eqref{hessian ads} to give estimates on the principal curvatures of $S$.
%Note that, by Gauss equation in the $\AdS^3$ setting, a maximal surface with principal curvatures $\pm\lambda$ has curvature given by $K_S=-1+\lambda^2$. It is proved in \cite{bon_schl} that, if $\partial_\infty(S)$ is the graph of a quasisymmetric homeomorphism and the principal curvatures of $S$ are bounded, then $K_S$ is uniformly negative, which means that $||\lambda||_\infty<1$. This is a substantial difference with the case of hyperbolic minimal surfaces, where the principal curvatures can be larger than 1. 

%From this point, we will always assume that $S$ is a maximal surface spanning the graph of a quasisymmetric homeomorphism, which is a compression disc for $\AdS^3$, with bounded principal curvatures; hence $S$ is complete (recall Remark \ref{remark completeness}) and the curvature is bounded by $-1\leq K_S<0$. However, when $||\lambda||_\infty$ approaches $1$, the curvature becomes close to $0$. We will use uniform estimates on the norm of the gradient of $u$. 

\subsection{Uniform gradient estimates}

We start by obtaining some technical estimates on the gradient of $u$, which will have as a consequence that a maximal surface cannot be very ``tilted'' with respect to a plane outside the convex hull.

\begin{lemma} \label{lemma gradient bounds}
The universal constant $L={2(1+\sqrt{2})}$ is such that, for every point $x$ on a maximal surface of nonpositive curvature in $\AdS^3$, $||\grad u||^2<L$.
\end{lemma}
\begin{proof}
Let $\gamma$ be a path on $S$ obtained by integrating the gradient vector field; more precisely, we impose $\gamma(0)=x$ and $$\gamma'(t)=-\frac{\grad u}{||\grad u||}\,.$$
Observe that $$u(\gamma(t))-u(x)=\int_0^t du(\gamma'(s))ds=\int_0^t -\langle\grad u(s),\frac{\grad u(s)}{||\grad u(s)||}\rangle ds=-\int_0^t ||\grad u(s)|| ds\,.$$
We denote $y(t)=||\grad u(\gamma(t))||$. We will show that $y(0)$ is bounded by a universal constant, using the fact that $u(\gamma(t))$ cannot become negative on $S$ (recall Corollary \ref{maximal surface contained convex hull}). We have
\begin{equation}
\ddt y(t)^2=2\langle\nabla_{\gamma'(t)}\grad u(\gamma(t)),\grad u(\gamma(t)) \rangle=2\langle\Hess u(\gamma'(t)),\grad u(\gamma(t))\rangle\,.
\end{equation}
Since, by equation \eqref{hessian ads}, 
$\Hess u-u E=\sqrt{1-u^2+||\grad u||^2}B$ and $||B(v)||\leq||v||$,
\begin{align*}
-\ddt y(t)^2\leq\left|\ddt y(t)^2\right| \leq &2||\Hess u(\gamma'(t))||||\grad u(\gamma(t))|| \\
\leq &2\left(u(\gamma(t))+\sqrt{1-u(\gamma(t))^2+y(t)^2}\right)y(t)\,.
\end{align*}
Using that
$$\left(u(\gamma(t))+\sqrt{1-u(\gamma(t))^2+y(t)^2}\right)^2\leq 2\left(u(\gamma(t))^2+1-u(\gamma(t))^2+y(t)^2\right)
=2\left(1+y(t)^2\right)
$$
we obtain
\begin{equation}\label{disequazione gradiente}
-\ddt y(t)\leq \sqrt{2}\sqrt{1+y(t)^2}\,.
\end{equation}
It follows that
$$\int_0^t\frac{y'(s)}{\sqrt{1+y(s)^2}}ds=\arcsinh y(t)-\arcsinh y(0)\geq \sqrt{2}t\,,$$
and therefore by a direct computation,
\begin{equation} \label{disequazione gradiente 2}
y(t)\geq y(0)\cosh(\sqrt{2}t)-\sqrt{1+y(0)^2}\sinh(\sqrt{2}t)\,.
\end{equation}
Now $$u(\gamma(t))-u(x)=-\int_0^t y(s)ds\leq\frac{1}{\sqrt{2}}\left(  -y(0)\sinh(\sqrt{2}t)+\sqrt{1+y(0)^2}(\cosh(\sqrt{2}t)-1)  \right)=:F(t)\,.$$
We must have $u(\gamma(t))\geq 0$ for every $t$; so we impose that $F(t)\geq -u(x)$ for every $t$. The minimum of $F$ is achieved for $$\tanh(\sqrt{2}t_{min})=\frac{y(0)}{\sqrt{1+y(0)^2}}\,.$$
Therefore 
$$F(t_{min})=-\frac{1}{\sqrt{2}}\left(1+\sqrt{1+y(0)^2}\right)\geq -u(x)$$ which is equivalent to $y(0)^2\leq2(u(x)^2+\sqrt{2}u(x))$. Recalling $u\in[-1,1]$, $||\grad u(x)||^2\leq 2(1+\sqrt{2})$ independently on the maximal surface $S$ and on the support plane $P_-$.
\end{proof}

We now apply the above uniform gradient estimate to prove a fact which will be of use shortly. Given two unit timelike vectors $v,v'\in T_x\AdS^3$ (both future-directed, or both past-directed), we define the hyperbolic angle between $v$ and $v'$ as the number $\alpha\geq0$ such that $\cosh\alpha=|\langle v,v'\rangle|$. Compare with Figure \ref{fig:angle} below.

\begin{lemma} \label{boundedness angle}
There exists a constant $\bar\alpha$ such that the following holds for every maximal surface $S$ in $\AdS^3$ and every totally geodesic plane $P_-$ in the past of $S$ which does not intersect $S$. Let $l$ be a geodesic line orthogonal to $P_-$ and let $x=l\cap S$. Suppose $x$ is at timelike distance less than $\pi/4$ from $P_-$. Then the hyperbolic angle $\alpha$ at $x$ between $l$ and the normal vector to $S$ is bounded by $\alpha\leq \bar\alpha$.
\end{lemma}
\begin{proof}
We use the same notation as Proposition \ref{formule hess lap ads}. It is clear that the tangent direction to $l$ is given by the vector $\nabla U$, where $U(x)=\sin d_{\AdS^3}(x,P_-)=\langle x,p\rangle$ is defined on the entire $\AdS^3$ and $p$ is the point dual to $P_-$. Recall $u$ is the restriction of $U$ to $S$.
From Equations \eqref{bombo} and \eqref{ciccio} of Proposition \ref{formule hess lap ads}, we have
$$\langle\nabla U,\nabla U\rangle=-1+u^2=||\grad u||^2-\langle \nabla U,N\rangle^2\,.$$
It follows that the angle $\alpha$ at $x$ between the normal to the maximal surface $S$ and the geodesic $l$ can be computed as
$$(\cosh\alpha)^2=\langle \frac{\nabla U(x)}{||\nabla U(x)||},N\rangle ^2=\frac{1-u(x)^2+||\grad u(x)||^2}{1-u(x)^2}$$
and so $\alpha$ is bounded by Lemma \ref{lemma gradient bounds} and the assumption that $u(x)^2\leq 1/2$.
\end{proof}

\subsection{Schauder estimates} \label{subsection schauder ads}
We move to proving Schauder-type estimates on the derivatives of the function $u=\sin d_{\AdS^3}(\cdot,P_-)$, expressed in suitable coordinates, of the form $$||u||_{C^2(B_0(0,\frac{R}{2}))}\leq C ||u||_{C^0(B_0(0,R))}$$
where the constant does not depend on $S$ and $P_-$. Here $B_0(0,R)$ denotes the Euclidean ball centered at $0$ of radius $R$. %We again prove this estimate by using a compactness argument.

Denote by $w$ the width of the convex hull of $\partial_\infty S=gr(\phi)$; recall that $w\leq\pi/2$. Let $x$ be a point of $S$. By Remark \ref{discussion width}, we have that $d_{\AdS^3}(x,\partial_-\mathcal{C})+d_{\AdS^3}(x,\partial_+\mathcal{C})\leq w$, therefore one among $d_{\AdS^3}(x,\partial_-\mathcal{C})$ and $d_{\AdS^3}(x,\partial_+\mathcal{C})$ must be smaller than $\pi/4$. Composing with an isometry of $\AdS^3$ (which possibly reverses time-orientation), we can assume $d_{\AdS^3}(x,\partial_-\mathcal{C})\leq d_{\AdS^3}(x,\partial_+\mathcal{C})$, which implies that $x$ has distance less than $\pi/4$ from $P_-$. This assumption will be important in the following.

Recall from Subsection \ref{subsection compactness} that, if $l$ is a timelike line orthogonal to the plane $P_0$ at $x_0$, the solid cylinder $Cl(x_0,P_0,R_0)$ is the set of points $x\in\AdS^3$ which lie on a spacelike plane $P_x$ orthogonal to $l$ such that $d_{P_x}(x,l\cap P_x)\leq R_0$. See also Figure \ref{fig:cylinder}.

\begin{repprop}{schauder estimate ads}
There exists a radius $R>0$ and a constant $C>0$ such that for every choice of:
\begin{itemize}
\item A maximal surface $S\subset\AdS^3$ with $\partial_\infty S$ the graph of an orientation-preserving homeomorphism; 
\item A point $x\in S$;
\item A plane $P_-$ disjoint from $S$ with $d_{\AdS^3}(x,P_-)\leq \pi/4$,
\end{itemize}
the function $u(\bullets)=\sin d_{\AdS^3}(\bullets,P_-)$ expressed in terms of normal coordinates centered at $x$,
namely $$u(z)=\sin d_{\AdS^3}(\exp_x(z),P_-)$$
where $\exp_x:\R^2\cong T_{x}S\rar S$ denotes the exponential map, satisfies the Schauder-type inequality
\begin{equation} \label{schauder}
||u||_{C^2(B_0(0,\frac{R}{2}))}\leq C ||u||_{C^0(B_0(0,R))}\,.
\end{equation} 
\end{repprop}

\begin{proof}
Fix a radius $R_0>0$. First, we show that there exists a radius $R>0$ such that the image of the Euclidean ball $B_0(0,R)$ under the exponential map at every point $x\in S$, for every surface $S$, is contained in the solid cylinder $Cl(x,T_x S,R_0)$. Indeed, suppose this does not hold, namely
\begin{equation} \label{infimum radii compactness}
\inf_{x\in S} \sup\left\{R:\exp_x(B_0(0,R))\subset Cl(x,T_x S,R_0) \right\}=0\,.
\end{equation}
Then one can find a sequence $S_n$ of maximal surfaces and points $x_n$ such that, if $R_n$ is the supremum of those radii $R$ for which $\exp_{x_n}(B_0(0,R))$ is contained in the respective cylinder of radius $R_0$, then $R_n$ goes to zero. We can compose with isometries of $\AdS^3$ so that all points $x_n$ are sent to the same point $x_0$ and all surfaces are tangent at $x_0$ to the same plane $P_0$. By Lemma \ref{lemma bon schl}, there exists a subsequence converging inside $Cl(x_0,P_0,R_0)$ to a maximal surface $S_\infty$. Therefore the infimum in the LHS of Equation \eqref{infimum radii compactness} cannot be zero, since for the limiting surface $S_\infty$ there is a radius $R_\infty$ such that $\exp_x(B_0(0,R_\infty))\subset Cl(x,T_x S,R_0)$.

We use a similar argument to prove the main statement. We can consider $P_-$ a fixed plane, and a point $x\in S$ lying on a fixed geodesic $l$ orthogonal to $P_-$. Suppose the claim does not hold, namely there exists a sequence of surfaces $S_n$ in the future of $P_-$ such that for the function $u_n(z)=\sin d_{\AdS^3}(\exp_{x_n}(z),P_n)$,
$$||u_n||_{C^2(B_0(0,\frac{R}{2}))}\geq n ||u||_{C^0(B_0(0,R))}\,.$$
Let us compose each $S_n$ with an isometry $T_n\in\isom(\AdS^3)$ so that $S_n'=T_n(S_n)$ is tangent at a fixed point $x_0$ to a fixed plane $P_0$, whose normal unit vector is $N_0$. 

We claim that the sequence of isometries $T_n$ is bounded in $\isom(\AdS^3)$, since $T_n^{-1}$ maps the element $(x_0,N_0)$ of the tangent bundle $T\AdS^3$ to a bounded region of $T\AdS^3$. Indeed, by our assumptions, $T_n^{-1}(x_0)=x_n$ lies on a geodesic $l$ orthogonal to $P_-$ and has distance less than $\pi/4$ (in the future) from $P_-$; moreover by Lemma \ref{boundedness angle} the vector $(d T_n)^{-1}(N_0)$ forms a bounded angle with $l$.

By Lemma \ref{lemma bon schl}, up to extracting a subsequence, we can assume $S_n'\rar S_\infty'$ on $Cl(x_0,P_0,R_0)$ with all derivatives. Since we can also extract a converging subsequence from $T_n$, we assume $T_{n}\rar T_\infty$, where $T_\infty$ is an isometry of $\AdS^3$. Therefore $T_n(P_-)$ converges to a totally geodesic plane $P_\infty$. 

Using the first part of this proof and Lemma \ref{lemma bon schl}, on the image of the ball $B_0(0,R)$ under the exponential map of $S'_n$,  the coefficients of the Laplace-Beltrami operators $\Delta_{S_n'}$ (in normal coordinates on $B_0(0,R)$) converge to the coefficients of $\Delta_{S_\infty'}$. Hence the operators $\Delta_{S_n'}-2$ are uniformly strictly elliptic with uniformly bounded coefficients. By classical Schauder estimates (see \cite{giltrud}), using the fact that $u_n$ solves the equation $\Delta_{S_n'}(u_n)-2u_n=0$, there exists a constant $c$ such that $$||u_n||_{C^2(B_0(0,\frac{R}{2}))}\leq c ||u_n||_{C^0(B_0(0,R))}\,,$$ for every $n$. This gives a contradiction.
%It follows that the coefficients of the Laplace-Beltrami operators $\Delta_{S_n'}$ on a Euclidean ball $B_0(0,R)$ of the tangent plane at $x_0$, for the coordinates given by the exponential map, converge to the coefficients of $\Delta_{S_\infty'}$, and therefore the operators $\Delta_{S_n'}-2$ are uniformly strictly elliptic with uniformly bounded coefficients. Since the functions $u_n$ satisfy $\Delta_{S_n'}(u_n)-2u_n=0$, by Schauder estimates (see \cite{giltrud}), there exists a constant $c$ such that $$||u_n||_{C^2(B_0(0,\frac{R}{2}))}\leq c ||u_n||_{C^0(B_0(0,R))}$$ for all $n$, and this gives a contradiction.
%the restriction of $u_n$ to the image under the exponential map of $S'_n$ of the ball $B_0(0,R)$ converges to the function $u_\infty=\sin d_{\AdS^3}(P_\infty,\cdot)$ on $S_\infty$ (in normal coordinates on $B_0(0,R)$). $u_\infty$ is solution of $\Delta_{S_\infty}(u_\infty)-2u_\infty=0$ and thus satisfies an inequality $||u_\infty||_{C^2(B_0(0,\frac{R}{2}))}\leq c' ||u_\infty||_{C^0(B_0(0,R))}$  for some constant $c'$, thus giving a contradiction.
\end{proof}

\begin{remark}
The statements of Lemma \ref{boundedness angle} and Proposition \ref{schauder estimate ads} could be improved so as to be stated in terms of the choice of any radius $R>0$, any number $w_0<\pi/2$ (replacing $\pi/4$), where the constant $C$ would depend on such choices. Similarly for Proposition \ref{boundedness projection} below. However, these details would not improve the final statement of Theorem \ref{estimate principal curvatures and width ads} and thus are not pursued here.
\end{remark}

We are therefore in a good point to obtain an estimate of the second derivatives of $u$ (and thus for the principal curvatures of $S$) in terms of the width $w$. However, let us remark that in Anti-de Sitter space the projection from a spacelike curve or surface to a totally geodesic spacelike plane is not distance-contracting. Hence we need to give an additional computation in order to ensure (by substituting the radius $R$ in Proposition \ref{schauder estimate ads} by a smaller one if necessary) that the projection from the geodesic balls $B_S(x,R)$ to $P_-$ has image contained in a uniformly bounded set. This is proved in the next Proposition, see also Figure \ref{fig:projectionsads}.

\begin{prop} \label{boundedness projection}
There exist constant radii $R_0'$ and $R'$ such that for every maximal surface of nonpositive curvature $S$ in $\AdS^3$, every point $x_0\in S$ and every totally geodesic plane $P_-$ which does not intersect $S$, such that the distance of $x_0$ from $P_-$ is at most $\pi/4$, the orthogonal projection $\pi|_S:S\rar P_-$ maps $S\cap Cl(x_0,T_{x_0} S,R_0')$ to $B_{P_-}(\pi(x_0),R')$.
\end{prop}
\begin{proof}
We can suppose $T_{x_0} S$ is the intersection of the plane $\left\{x_4=0\right\}$ with $\AdS^3$ and $x_0=[0,0,1,0]$. As in the coordinate system \eqref{cylindric coordinates}, the points $x$ in $Cl(x_0,T_{x_0} S,R_0')$ have coordinates $$x=[\cos\theta \sinh r,\sin\theta \sinh r, \cos \zeta \cosh r,\sin \zeta \cosh r]\,,$$ 
for $r\leq R_0'$. Let us denote by $\Ip(p)$ (resp. $\mathrm{I}^-(x_0)$) the cone of points connected to $p$ by a future-directed (resp. past-directed) timelike path in $\AdS^3\setminus Q$, where $Q$ is the plane at infinity in the affine chart. 

Since $S$ is spacelike, $S\cap Cl(x_0,T_{x_0} S,R_0')$ is contained in $Cl(x_0,T_{x_0} S,R_0')\setminus (\Ip(x_0)\cup \mathrm{I}^-(x_0))$. See also Figure \ref{fig:cylinder}. Hence $|\langle x,x_0\rangle|>1$ (recall Equation \eqref{lunghezza1} in Section \ref{subsection Hyp AdS}), which is equivalent to 
\begin{equation} \label{relazione altezza raggio}
|\cos \zeta|>\frac{1}{\cosh r}\,.
\end{equation}
Let $l$ be the geodesic through $x_0$ orthogonal to $P_-$. We will conduct the computation in the double cover $\wAdS\subset \R^{2,2}$. We can assume $l$ has tangent vector at $x_0$ given by $l'(0)=(\sinh\alpha,0,0,\cosh\alpha)$, where of course $\alpha$ is the angle between $l$ and the normal to $S$ at $x_0$. Therefore $$l(t)=(\cos t) x_0+(\sin t) l'(0)=(\sin t\sinh\alpha,0,\cos t,\sin t\cosh\alpha)\,.$$
Let $w_1=d_{\AdS^3}(x_0,P_-)$, so $P_-=p^\perp$, where $p$ is the point 
$$p=l'(-w_1)=(\cos w_1\sinh\alpha,0,\sin w_1,\cos w_1\cosh\alpha)\,.$$
The projection of $x$ to $P_-$ is given by
$$\pi(x)=\frac{x+\langle x,p\rangle p}{\sqrt{1-\langle x,p\rangle^2}}$$
provided $\langle x,p\rangle^2<1$, which is the condition for $x$ to be in the domain of dependence of $P_-$. (We say that $x$ is in the domain of dependence of $P_-$ if the dual plane of $x$ is disjoint from $P_-$.) The distance $d$ between $\pi(x)$ and $\pi(x_0)=l(-w_1)$ is given by the expression
\begin{equation} \label{expression distance projection bound}
\cosh d=\left|\langle \pi(x),l(-w_1)\rangle\right|=\left|\frac{\langle x,l(-w_1)\rangle}{\sqrt{1-\langle x,p\rangle^2}}\right|\,.
\end{equation}
Now, we have 
\begin{align*}
|\langle x,p\rangle|=&|\cos\theta \sinh r \cos w_1\sinh\alpha-\cos \zeta \cosh r \sin w_1-\sin \zeta \cosh r \cos w_1\cosh\alpha| \\
\leq& \sinh r\sinh \alpha+\frac{\sqrt{2}}{2}\cosh r+\sinh r\cosh\alpha=\frac{\sqrt{2}}{2}\cosh r+(\sinh r) e^{\alpha}\,.
\end{align*}
In the last line, we have used that $|\sin \zeta|=\sqrt{1-(\cos \zeta)^2}\leq \tanh r$, by Equation (\ref{relazione altezza raggio}), and that $\sin w_1<\sqrt{2}/2$. Since the hyperbolic angle $\alpha$ is uniformly bounded by Lemma \ref{boundedness angle} (Figure \ref{fig:angle}), it follows that if $r\leq R_0'$ for $R_0'$ sufficiently small, ${\sqrt{1-\langle x,p\rangle^2}}$ is uniformly bounded from below. Moreover, 
\begin{align*}
|\langle x,l(-w_1)\rangle|=&|-\cos\theta \sinh r \sin w_1\sinh\alpha-\cos \zeta \cosh r \cos w_1+\sin \zeta \cosh r \sin w_1\cosh\alpha| \\
\leq& \sinh r\sinh \alpha+\cosh r+\cosh r\cosh\alpha
\end{align*}
is uniformly bounded. This shows, from Equation \eqref{expression distance projection bound}, that $\cosh d\leq \cosh R'$ for some constant radius $R'$ (depending on $R_0'$). This concludes the proof.
\end{proof}

\begin{figure}[htbp]
\centering
\begin{minipage}[c]{.46\textwidth}
\centering
\includegraphics[height=6.5cm]{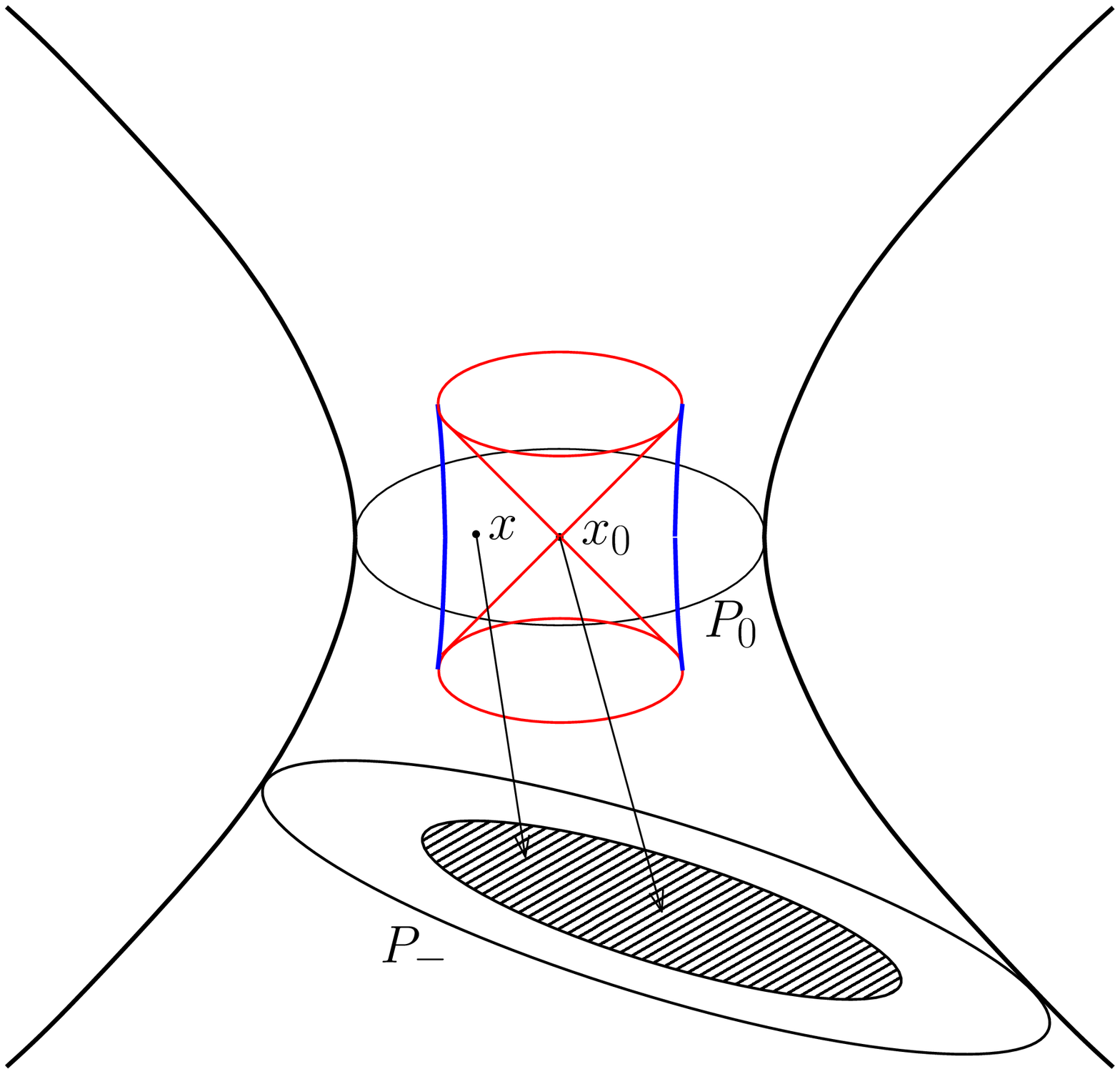} 
%\captionsetup{labelformat=empty}
\caption{Projection from points in $Cl(x_0,T_{x_0} S,R_0')$ which are connected to $x_0$ by a spacelike geodesic have bounded image.} \label{fig:projectionsads}
\end{minipage}%
\hspace{6mm}
\begin{minipage}[c]{.46\textwidth}
\centering
\includegraphics[height=6.5cm]{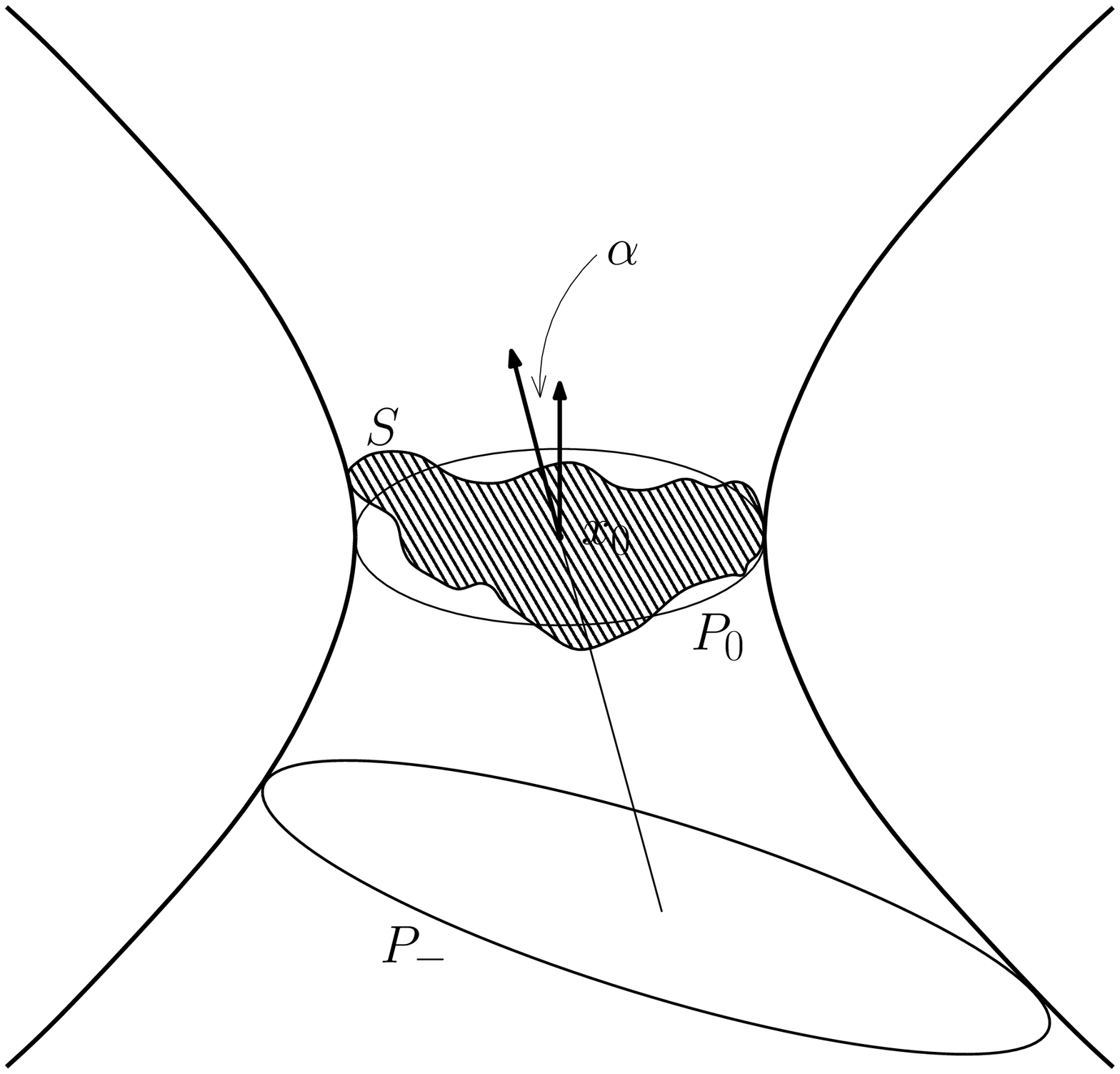}
%\captionsetup{labelformat=empty}
\caption{The key point is that the hyperbolic angle $\alpha$ is uniformly bounded, by Lemma \ref{boundedness angle}.} \label{fig:angle}
\end{minipage}
\end{figure}

Therefore, replacing $R_0$ in Lemma \ref{lemma bon schl} with $\min\left\{R_0,R_0'\right\}$, we have that the geodesic balls of radius $R$ ($R$ as in Proposition \ref{schauder estimate ads}) on $S$ centered at $x$ project to $P_-$ with image contained in $B_{P_-}(\pi(x),R')$. The radii $R$ and $R'$ are fixed, not depending on $S$. 

\subsection{Principal curvatures}

In this subsection we finally prove the estimate on the supremum of the principal curvatures of $S$ in terms of the width. Recall the statement of the main theorem of this section:

\begin{reptheorem}{estimate principal curvatures and width ads}
There exists a constant $C_1$ such that, for every maximal surface $S$ with $||\lambda||_\infty<1$ and width $w$, 
$$||\lambda||_\infty\leq C_1\tan w\,.$$
\end{reptheorem}

We take an arbitrary point $x\in S$. By Remark \ref{discussion width}, we know that there are two disjoint planes $P_-$ and $P_+$ with $d_{\AdS^3}(x,P_-)+d_{\AdS^3}(x,P_+)=w_1+w_2\leq w$ where $w$ is the width. As in the previous subsection, we will assume $P_-$ is a fixed plane in $\AdS^3$, upon composing with an isometry. Figure \ref{fig:lemmadistance} gives a picture of the situation of the following lemma.

\begin{lemma} \label{distance and width ads}
Let $p\in P_-$, $q\in P_+$ be the endpoints of geodesic segments $l_1$ and $l_2$ from $x\in S$ orthogonal to $P_-$ and $P_+$, of length $w_1$ and $w_2$, with $w_1\leq w_2$. Let $p'\in P_-$ a point at distance $R'$ from $p$ and let $d=d_{\AdS^3}((\pi|_{P_+})^{-1}(p'),P_-)$. Then
\begin{equation}
\tan d\leq (1+\sqrt{2})\cosh R' \tan(w_1+w_2).
\end{equation}
\end{lemma}
\begin{proof}
As in the previous proof, we do the computation in $\wAdS$. We assume $x=(0,0,1,0)$ and $l_1$ is the geodesic segment parametrized by $l_1(t)=(\cos t)x-(\sin t)(0,0,0,1)$, so that the plane $P_-$ is dual to $p_-=(0,0,\sin w_1,\cos w_1)$. Points on the plane $P_-$ at distance $R'$ from $\pi(x)=l_1(w_1)=(0,0,\cos w_1,-\sin w_1)$ have coordinates
$$p'=(\cos\theta\sinh R',\sin\theta\sinh R',\cosh R'\cos w_1,-\cosh R'\sin w_1)\,.$$
We also assume $l_2$ has initial tangent vector $l_2'(0)=(\sinh\alpha,0,0,\cosh\alpha)$, where $\alpha$ is the hyperbolic angle between $(0,0,0,1)$ and  $l_2'(0)$, so that $l_2(t)=(\cos t)x+(\sin t)(\sinh\alpha,0,0,\cosh\alpha)$. Note that $l_2'(w_2)=(\cos w_2\sinh\alpha,0,-\sin w_2,\cos w_2\cosh\alpha)=:p_+$ is the unit vector orthogonal to $P_+$, by construction.

We derive a condition which must necessarily be satisfied by $\alpha$, because $P_-$ and $P_+$ are disjoint. Indeed, we must have
$$|\langle p_-,p_+\rangle|=-\sin w_1\sin w_2+\cos w_1\cos w_2\cosh\alpha \leq 1\,,$$
which is equivalent to
\begin{equation} \label{inequality cosh alpha w12}
\cosh\alpha<\frac{1+\sin w_1\sin w_2}{\cos w_1\cos w_2}\,.
\end{equation}
Let us now write
$$(\tanh\alpha)^2=\left(1+\frac{1}{\cosh\alpha}\right)\left(1-\frac{1}{\cosh\alpha}\right)\leq2(\cosh\alpha-1)\,,$$
and therefore, using (\ref{inequality cosh alpha w12}),
\begin{equation} \label{inequality tanh alpha w12}
(\tanh\alpha)^2< 2\left(\frac{1-\cos(w_1+w_2)}{\cos w_1\cos w_2}\right)\leq 2\left(\frac{1-(\cos(w_1+w_2))^2}{\cos w_1\cos w_2}\right) \leq 2\frac{(\sin(w_1+w_2))^2}{\cos w_1\cos w_2}\,.
\end{equation}

To compute $d$, we now write explicitly the geodesic $\gamma$ starting from $p'$ and orthogonal to $P_-$. We find $d$ such that $\gamma(d)\in P_+$ and this will give the expected inequality. We have 
$$\gamma(d)=(\cos d)p'+(\sin d)(0,0,\sin w_1,\cos w_1)$$
and $\gamma(d)\in P_+$ if and only if $\langle \gamma(d),p_+\rangle=0$, which gives the condition
\begin{multline*}
\cos d(\cosh R'(\cos w_1 \sin w_2+\cos w_2 \sin w_1\cosh\alpha)+\sinh R'(\cos\theta \cos w_2\sinh\alpha)) \\
+\sin d(\sin w_1 \sin w_2-\cos w_1 \cos w_2\cosh\alpha)=0\,.
\end{multline*}
We express
\begin{align*}
\tan d=&\cosh R'\frac{\cos w_1 \sin w_2+\cos w_2 \sin w_1\cosh\alpha}{\cos w_1 \cos w_2\cosh\alpha-\sin w_1 \sin w_2} \\
+&\sinh R'\frac{\cos\theta \cos w_2\sinh\alpha}{\cos w_1 \cos w_2\cosh\alpha-\sin w_1 \sin w_2}~.
\end{align*}
The first term in the RHS is easily seen to be less than $\cosh R'\tan(w_1+w_2)$. We turn to the second term. Using (\ref{inequality tanh alpha w12}), it is bounded by
$$\sinh R'\tanh\alpha\frac{\cos w_2}{\cos{(w_1+w_2)}}\leq\sqrt{2}\sinh R'\tan(w_1+w_2)\left(\frac{\cos w_2}{\cos w_1}\right)^{\frac{1}{2}}\,.$$
In conclusion, having assumed $w_1\leq w_2$, we can put $\cos(w_2)/\cos(w_1)\leq 1$, sum the two terms and get
$$\tan d\leq(1+\sqrt{2})\cosh R'\tan(w_1+w_2)\,.$$
\end{proof}

\begin{figure}[htbp]
\centering
\includegraphics[height=6.5cm]{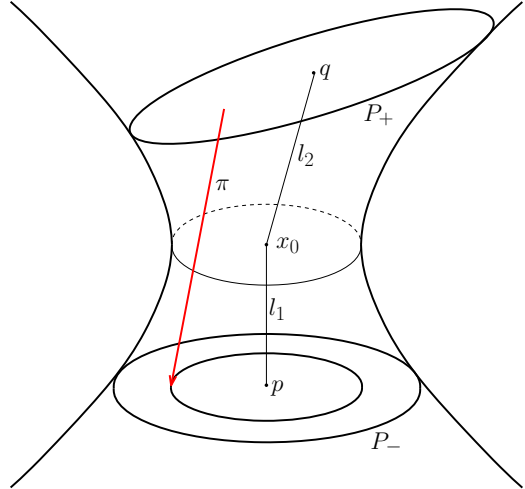}
\caption{The setting of Lemma \ref{distance and width ads}. We assume $w_1=d_{\AdS^3}(x_0,p)<d_{\AdS^3}(x_0,q)=w_2$. \label{fig:lemmadistance}}
\end{figure}

\begin{proof}[Proof of Theorem \ref{estimate principal curvatures and width ads}]
Let $x\in S$ and consider the point $x_-$ of $\partial_-\mathcal{C}$ which minimizes the distance from $x$, where $\mathcal{C}$ is the convex hull of $S$. Let $P_-$ be the plane through $x_-$ orthogonal to the geodesic line containing $x$ and $x_-$ (recall Remark \ref{discussion width}). The plane $P_-$ is then a support plane of $\partial_-\mathcal{C}$. 
We construct analogously the support plane $P_+$ for $\partial_+ \mathcal{C}$. As discussed in Remark \ref{discussion width}, 
$$d_{\AdS^3}(x,P_-)+d_{\AdS^3}(x,P_+)\leq w\,.$$
Moreover, we can assume (upon composing with a time-orientation-reversing isometry, if necessary) that $d_{\AdS^3}(x,P_-)\leq d_{\AdS^3}(x,P_+)$. As a consequence, $d_{\AdS^3}(x,P_-)\leq \pi/4$.

Let us now consider the function $u$, defined in normal coordinates:
$$u=\sinh d_{\AdS^3}(\exp_x(\bullets),P_-)\,.$$
By Equation (\ref{hessian ads}), we have the following expression for the shape operator of $S$:
 $$B=\frac{1}{\sqrt{1-u^2+||\grad u||^2}}(\Hess\, u-u\,E)\,.$$
In normal coordinates at $x$ the Hessian of $u$ is given just by the second derivatives of $u$; in Proposition \ref{schauder estimate ads} we showed the second derivatives of $u$ are bounded, up to a factor, by $||u||_{C^0(B_S(x,R))}$. By Proposition \ref{boundedness projection}, $||u||_{C^0(B_S(x,R))}$ is smaller than the supremum of the hyperbolic sine of the distance $d$ from $P_-$ of points of $S$ which project to $B_{P_-}(\pi(x),R')$. Therefore we have the following estimate for the principal curvatures at $x$:
$$|\lambda|\leq C'\frac{||u||}{\sqrt{1-||u||^2}}\leq C' \tan\left(\sup \{d_{\AdS^3}(p,P_-):p\in (\pi_S)^{-1}(B_{P_-}(\pi(x),R'))\}\right)\,.$$
The constant $C'$ involves the constant which appears in Equation \eqref{schauder} in Proposition \ref{schauder estimate ads}. The quantity in brackets in the RHS is certainly less than 
$$\left(\sup \{d_{\AdS^3}(p,P_-):p\in (\pi_{P_+})^{-1}(B_{P_-}(\pi(x),R'))\}\right)\,.$$
Thus, applying Lemma \ref{distance and width ads} we obtain:
$$||\lambda||_\infty\leq C_1\tan w\,.$$
The constant $C$ then involves $C'$ and $\cosh R'$. Such inequality holds independently on the point $x$ and thus concludes the proof.
\end{proof}

%------------------------------------------------------

\section{Maximal surfaces with large principal curvatures} \label{sec estimate large}

The main goal of this section is to prove Theorem \ref{theorem width lambda big}. Recall that we denote by $||\lambda||_\infty$ the supremum of the principal curvatures of a maximal surface $S$ with $\partial_\infty S=gr(\phi)\subset\partial_\infty\AdS^3$, and by $w$ the width of the convex hull of $\partial_\infty S$.

\begin{reptheorem}{theorem width lambda big}
There exist universal constants $M>0$ and $\delta\in(0,1)$ such that, if $S$ is an entire maximal surface in $\AdS^3$ with $\delta\leq ||\lambda||_\infty<1$ and width $w$, then
$$\tan w\geq \left(\frac{1}{1-||\lambda||_\infty}\right)^{1/M}\,.$$
\end{reptheorem}

Observe that it clearly suffices to prove that for every $\lambda_0\geq\delta$, if there exists a point $x_0\in S$ such that $\lambda(x_0)=\lambda_0$, then
\begin{equation} \label{quiquoqua}
\tan w\geq \left(\frac{1}{1-\lambda_0}\right)^{1/M}\,.
\end{equation}
Indeed, if $||\lambda||_\infty$ is not achieved on $S$, one has that Equation \eqref{quiquoqua} holds for every $\lambda_0<||\lambda||_\infty$, and by continuity \eqref{quiquoqua} holds for $\lambda_0=||\lambda||_\infty$ as well.

\subsection{Uniform gradient estimates}

We define the function $v:S\to\R$ by $$v(x)=\ln\left(\frac{1}{1-\lambda(x)}\right)\,,$$ where $\lambda(x)$ is the positive eigenvalue of the shape operator. Observe that $v(x)\to\infty$ as $\lambda(x)$ approaches $1$, while $v(x)=0$ if $x$ is an umbilical point, and the statement of Theorem \ref{theorem width lambda big} amounts to proving that $\tan w\geq e^{v(x_0)/M}$. The first important step is to show that, if $v$ is large at some point, then it remains large on a geodesic ball of $S$ with large radius.

\begin{prop} \label{prop gradient estimate}
There exists a universal constant $M$ such that $||\grad v||\leq M$, for every maximal surface $S$ in $\AdS^3$ with uniformly negative curvature.
\end{prop}

We will prove some \emph{a priori} estimates for the function 
 $\chi=-\ln|\lambda|$, which is defined in the complement of umbilical points
and it takes values in $[0,+\infty)$. Recall that by Lemma \ref{lemma eq quasilin} the function $\chi$ satisfies the quasi-linear equation:
\begin{equation} \label{ueq}
\Delta_S\chi=2(1-e^{-2\chi})\,, \tag{\ref{quasi}}
\end{equation}
while by the Gauss equation the curvature of $S$ is given by: 
\begin{equation}\label{eq:curv}
  K=-1+e^{-2\chi}\,.
\end{equation}

We start by computing the gradient and the Laplacian of the function
$||\grad \chi||^2$.

\begin{lemma}
The following identities hold:
\begin{equation}\label{eq:grad}
\grad(||\grad \chi||^2)=2\Hess\chi(\grad \chi)\,,
\end{equation}
\begin{equation}\label{eq:lap}
\Delta_S(||\grad \chi||^2)=(10 e^{-2\chi}-2)||\grad \chi||^2+2||\Hess\chi||^2\,.
\end{equation}
\end{lemma}
\begin{proof}
In order to compute the gradient, notice that if $\xi$ is any vector tangent to $S$, then
$$
\xi(||\grad \chi||^2)=2\langle \nabla_\xi(\grad\chi), \grad\chi\rangle=
2\langle \Hess\chi(\xi), \grad\chi\rangle=2\langle\xi, \Hess\chi(\grad\chi)\rangle\,.
$$
This proves the first identity. By taking the covariant derivative of $\grad||\grad \chi||^2$ we get
\begin{align*}
\Hess(||\grad \chi||^2)&=2\nabla(\Hess\chi)(\grad\chi)+2(\Hess\chi)^2\\
                      &=2\nabla_{\bullet}(\nabla\grad \chi)(\grad \chi)+2(\Hess\chi)^2\\
                      &=2\nabla_{\grad \chi}(\Hess \chi)+2R(\bullet, \grad \chi)\grad \chi+2(\Hess\chi)^2\,.
\end{align*}
Taking the trace we get
$$
\Delta_S(||\grad \chi||^2)=2\grad\chi(\Delta_S \chi)+2K||\grad \chi||^2+2||\Hess\chi||^2\,,
 $$
and using (\ref{ueq}) and (\ref{eq:curv}) we deduce that
\begin{align*}
\Delta_S(||\grad \chi||^2)&=4\langle\grad\chi, \grad(1-e^{-2\chi})\rangle
+2(-1+e^{-2\chi})||\grad \chi||^2+2||\Hess \chi||^2\\
& =4\cdot 2e^{-2\chi}||\grad \chi||^2+2(-1+e^{-2\chi})||\grad \chi||^2+2||\Hess \chi||^2\\
& =(10e^{-2\chi}-2)||\grad \chi||^2+2||\Hess \chi||^2\,.
\end{align*}
This concludes the proof.
\end{proof}

\begin{lemma} \label{lm:bound}
Let $g:(0,+\infty)\rightarrow\mathbb R$ any  smooth function
and consider the function $\psi:S\to\R$ defined by
$$\psi=e^{g(\chi)}||\grad \chi||^2\,.$$
Then we have
\begin{equation}\label{eq:estimate}
\Delta_S \psi\geq a(\chi)\psi^2+b(\chi)\psi+c(\chi)\,,
\end{equation}
where
\begin{equation}
\begin{array}{l}
a(\chi)=g''e^{-g}\,,\\
b(\chi)=6g'(\chi)(1-e^{-2\chi})+(10 e^{-2\chi}-2)\,,\\
c(\chi)=e^g(\Delta_S \chi)^2\,.
\end{array}
\end{equation}
\end{lemma}
\begin{proof}
Let us put $\alpha=e^g$, then (using \eqref{eq:grad}) we have
\[
\grad\psi=\alpha'||\grad \chi||^2\grad \chi+2\alpha\Hess\chi(\grad \chi)\,.
\]
Differentiating again we get
\[
\begin{split}
\Hess\,\psi& =\alpha''||\grad \chi||^2\grad\chi\otimes\grad \chi +\alpha'||\grad \chi||^2\Hess \chi + 2\alpha' \grad\chi\otimes\left(\Hess\chi(\grad \chi)\right)
\\
&\qquad+
2\alpha'\left(\Hess\chi(\grad \chi)\right)\otimes\grad \chi+\alpha\Hess(||\grad \chi||^2)\,,
\end{split}
\]
so by using \eqref{quasi} and \eqref{eq:lap} we get
\begin{align*}\label{eq:lapgrad}
\Delta_S\psi& =\alpha''||\grad \chi||^4+
2\alpha'(1-e^{-2\chi})||\grad \chi||^2+4 \alpha'\langle\grad \chi, \Hess\chi(\grad \chi)\rangle \\
&\qquad+
\alpha(10 e^{-2\chi}-2)||\grad \chi||^2+2\alpha||\Hess \chi||^2 \\
&=\alpha''||\grad \chi||^4+(2\alpha'(1-e^{-2\chi})+\alpha(10e^{-2\chi}-2))||\grad \chi||^2+\\
&\qquad+4 \alpha'\langle\grad \chi, \Hess\chi(\grad \chi)\rangle+2\alpha||\Hess \chi||^2\,.
\end{align*}
The main term to be estimated  is the scalar product
$\langle\grad \chi, \Hess\chi(\grad \chi)\rangle$.
Let $\mu_1, \mu_2 $ be the eigenvalues of $\Hess\chi$ and
$x_1,x_2$ be the coordinates of $\grad \chi$ with respect to
an orthonormal basis of eigenvectors of $\Hess \chi$.
Then we have
\[
\langle\grad \chi, \Hess\chi(\grad\chi)\rangle=\mu_1 x_1^2+\mu_2x_2^2\,.
\]
Notice that $\Delta_S\chi=\mu_1+\mu_2$, so if we put $\xi=\mu_1-\mu_2$
we deduce that
\[
\langle\grad \chi, \Hess\chi(\grad \chi)\rangle=
\frac{\Delta_S\chi}{2}||\grad \chi||^2+\frac{\xi}{2}(x_1^2-x_2^2)\,.
\]
So we get
\begin{equation}\label{eq:est}
4\alpha'\langle\grad \chi, \Hess\chi(\grad \chi)\rangle=4\alpha'(1-e^{-2\chi})||\grad \chi||^2+2\alpha'\xi(x_1^2-x_2^2)\,.
\end{equation}
Putting \eqref{eq:est} into the equality we obtained previously yields
  \begin{align*}\label{eq:lapgrad2}
\Delta_S\psi& 
=\alpha''||\grad \chi||^4+(2\alpha'(1-e^{-2\chi})+\alpha(10e^{-2\chi}-2))||\grad \chi||^2\\
&\qquad+4\alpha'(1-e^{-2\chi})||\grad \chi||^2+2\alpha'\xi(x_1^2-x_2^2) +2\alpha||\Hess \chi||^2\\
&=\alpha''||\grad \chi||^4+(6\alpha'(1-e^{-2\chi})+\alpha(10e^{-2\chi}-2))||\grad \chi||^2 \\
&\qquad+
2\alpha'\xi(x_1^2-x_2^2) +2\alpha||\Hess \chi||^2\,.
\end{align*}
Now observe that
 $$
 |2\alpha'\xi(x_1^2-x_2^2)|\leq2|\alpha'||\xi|||\grad \chi||^2\leq
 \frac{(\alpha')^2}{\alpha}||\grad \chi||^4+\alpha\xi^2\,,
 $$
hence we obtain
\begin{align*}
\Delta_S \psi& \geq
\alpha''||\grad \chi||^4+(6\alpha'(1-e^{-2\chi})+\alpha(10e^{-2\chi}-2))||\grad \chi||^2\\
&\qquad- \frac{(\alpha')^2}{\alpha}||\grad u||^4-\alpha\xi^2 +2\alpha||\hess u||^2=\\
&=\left(\alpha''- \frac{(\alpha')^2}{\alpha}\right)||\grad \chi||^4+
(6\alpha'(1-e^{-2\chi})+\alpha(10 e^{-2\chi}-2))||\grad \chi||^2 \\
&\qquad +\alpha(2||\Hess \chi||^2-\xi^2)\,.
\end{align*}
But $||\Hess \chi||^2=\mu_1^2+\mu_2^2$, so that $$2||\Hess \chi||^2-\xi^2=2(\mu_1^2+\mu_2^2)-(\mu_1-\mu_2)^2=(\mu_1+\mu_2)^2=(\Delta_S \chi)^2\,.$$
We finally obtain
$$
\Delta_S \psi \geq
\left(\alpha''- \frac{(\alpha')^2}{\alpha}\right)||\grad\chi||^4+
(6\alpha'(1-e^{-2\chi})+\alpha(10 e^{-2\chi}-2))||\grad\chi||^2+
\alpha(\Delta_S\chi)^2\,.
$$
Observing that $\alpha'=g'e^g$ and $\alpha''= (g''+(g')^2)e^g$, we get 
$\alpha''-(\alpha')^2/\alpha=g'' e^g$ and thus
\begin{align*}
\Delta_S \psi& \geq
g'' e^g||\grad \chi||^4+
(6g'(1-e^{-2\chi})+(10 e^{-2\chi}-2))e^g||\grad \chi||^2+
\alpha(\Delta_S\chi)^2 \\
&= g'' e^{-g}\psi^2+(6g'(1-e^{-2\chi})+(10 e^{-2\chi}-2))\psi+
e^g(\Delta_S\chi)^2\,,
\end{align*}
as in the statement.
\end{proof}

\begin{remark} \label{rk:bound}
In the hypothesis of Lemma \ref{lm:bound}, suppose that $g$ is a convex function
so that $a(\chi)>0$.  If $\sup_{\chi>0} (-{b(\chi)}/{a(\chi)})= M_1<+\infty$ we have that 
$\Delta_S \psi>0$ whenever  $\psi>M_1$ . 
\end{remark}

\begin{lemma}\label{lm:grl}
There is a constant $M_2$ such that $||\grad\lambda||<M_2$ for any maximal surface of nonpositive curvature in $\AdS^3$.
\end{lemma}
\begin{proof}
Take a point $p\in S$ and consider normal coordinates $x,y$ centered at $p$.
We have that
\[
   B(\partial_x)= a(x,y)\partial_x+b(x,y)\partial_y\,,
\]
where $a,b$ are smooth functions in a neighborood of $p$.
Since $B$ is traceless self-adjoint, it turns out that
\[
  \lambda=\sqrt{a^2+b^2} \,,
\]
so that
\[
\grad\lambda=\frac{1}{\sqrt{a^2+b^2}}(a\grad a+b\grad b)\,.
\]
In particular
\[
||\grad\lambda||^2\leq ||\grad a||^2+||\grad b||^2\,.
\]
On the other hand, at the point $p$ we have
\[
  (\nabla B)(\partial_x)=\nabla (B(\partial_x))=(\grad a)\otimes(\partial_x)+(\grad b)\otimes(\partial_y)\,,
\]
so
we deduce that at the point $p$
\[
  ||\grad\lambda||^2\leq||\nabla B||^2\,.
\]
By applying Lemma \ref{lemma bon schl}, it is not difficult to show that there exists a universal 
 constant $M_2$ such that $||\nabla B||\leq M_2$ for any point of any maximal surface
 with nonpositive curvature.
\end{proof}

\begin{proof}[Proof of Proposition \ref{prop gradient estimate}]
We will derive an a-priori bound on $||\grad v||$ where
$$v=-\ln(1-|\lambda|)=-\ln(1-e^{-\chi})\,.$$
Indeed we have that
\begin{equation} \label{eq gradient v lam}
||\grad v||^2=\frac{1}{(1-e^{-\chi})^2}||\grad\lambda||^2=\frac{e^{-2\chi}}{(1-e^{-\chi})^2}||\grad \chi||^2\,,
\end{equation}
hence in particular $||\grad v||\leq M_2$ at umbilical points, where $M_2$ is the constant given
by Lemma \ref{lm:grl}.
On the other hand, we have $||\grad v||^2=e^g||\grad \chi||^2$ where
$$g(\chi)=-2(\chi+\ln(1-e^{-\chi}))\,.$$
By a direct computation
$$
  g'=-2-2\frac{e^{-\chi}}{1-e^{-\chi}}=-\frac{2}{1-e^{-\chi}}\,,$$
$$  g''=\frac{2e^{-\chi}}{(1-e^{-\chi})^2}\,.
$$
Applying Lemma \ref{lm:bound} we have
$$
\Delta_S(||\grad v||^2)\geq a(\chi)||\grad v||^4+b(\chi)||\grad v||^2+c(\chi)\,,
$$
where $c(\chi)\geq 0$ and
\begin{align*}
a(\chi)&=\left(\frac{2e^{-\chi}}{(1-e^{-\chi})^2}\right)(1-e^{-\chi})^2e^{2\chi}=2e^\chi\,, \\
b(\chi)&=\left(-12\left(\frac{1}{1-e^{-\chi}}\right)(1-e^{-2\chi}) +(10e^{-2\chi}-2)\right)=10e^{-2\chi}-12 e^{-\chi}-14\,.
\end{align*}
Observe that $-b(\chi)/a(\chi)\rightarrow 0$ as $\chi\rightarrow +\infty$ and
$-b(\chi)/a(\chi)\rightarrow 8$ as $\chi\rightarrow 0$.
By Remark \ref{rk:bound}, taking $M=\sup_{\chi>0} (-b(\chi)/a(\chi))$ we have
that $\Delta_S(||\grad v||^2)|>0$ whenever $||\grad v||^2> M$.

Summarizing we have that
\begin{itemize}
\item $||\grad v||^2$ is bounded by a constant $M_2$ at umbilical points;
\item if at some non-umbilical point $||\grad v||^2>M_1$ then $\Delta_S(||\grad v||^2)>0$.
\end{itemize}
We now claim that
\begin{equation}\label{eq:nablav}
\sup||\grad v||^2\leq M^2:= \max(M_1, M_2)\,.
\end{equation}
To prove the claim, suppose to have a sequence of points $x_n$ such that
$||\grad v||^2(x_n)$ converges to $\sup_S||\grad v||^2$. The latter is finite by Equation \eqref{eq gradient v lam}, since $||\grad\lambda||$ is bounded and $e^{-\chi}$ stays uniformly away from $1$ by the hypothesis of uniformly negative curvature of $S$.
Take a sequence of isometries $T_n$ of $\AdS^3$ so that
   $T_n(x_n)$ is a fixed point $x_0$ and the maximal surface $S_n=T_n(S)$ is tangent to a fixed spacelike plane through $x_0$.
   
By Lemma \ref{lemma bon schl}, up to a subsequence the surfaces $S_n$ converge $C^\infty$ on compact sets to a maximal surface $S_\infty$. 

If $x_0$ is an umbilical point for $S_\infty$, then by Equation \eqref{eq gradient v lam} $||\grad v_n(x_0)||^2$ is bounded by $(1+\epsilon)M_2$ for $n$ large. Hence $||\grad v(x_n)||$ is bounded by $(1+\epsilon)M_2$, and thus $\sup_S||\grad v||^2$ is bounded by $(1+\epsilon)M_2$. (It is actually bounded by $M_2$ itself, since in the argument $\epsilon$ is arbitrary.) On the other hand, if $x_0$ is not an umbilical point for $S_\infty$,  by the $C^\infty$ convergence 
$$
  ||\grad v_\infty||^2(x_0)=\sup_S||\grad v||^2=\sup_{S_\infty}||\grad v_\infty||^2\,,
$$
therefore $x_0$ is an interior maximum point for $||\grad v_\infty||^2$.
Hence $\Delta_S||\grad v_\infty||^2(x_0)\leq 0$ and
$||\grad v_\infty||^2\leq M_1$ by the first part of the proof.
\end{proof}

\begin{repprop}{cor estimate francesco}
There exists a universal constant $M$ such that, for every maximal surface $S$ of nonpositive curvature in $\AdS^3$ and every pair of points $p,q\in S$,
$$1-\lambda(q)\leq e^{Md_S(p,q)}(1-\lambda(p))\,.$$
\end{repprop}
\begin{proof}
Using Proposition \ref{prop gradient estimate}, let $\gamma:[0,d_S(p,q)]$ be a unit speed parameterization of the geodesic segment connecting $p$ and $q$, and get
$$|v(p)-v(q)|=\left|\int_{[0,d_S(p,q)]}dv(\dot\gamma(t))dt\right|\leq\int_{[0,d_S(p,q)]}||\grad v||dt\leq Md_S(p,q)\,.$$
Therefore
\begin{equation} \label{eq stime vpq}
v(q)\geq v(p)-Md_S(p,q)\,,
\end{equation}
from which the statements follows, by recalling that $v$ is the function on $S$ defined so that $e^{v(x)}=1/(1-\lambda(x))$.
\end{proof}

In particular, we will apply Proposition \ref{cor estimate francesco} in the following form. 
\begin{cor} \label{cor autovalore grande}
There exists a universal constant $M$ such that, for every maximal surface $S$ in $\AdS^3$ and every pair $x_0\in S$,
$v(x)\geq {v(x_0)/2}$ for every point $x$ in the geodesic ball $B_S(x_0,v(x_0)/2M)$.
\end{cor}
\begin{proof}
Follows directly from Equation \eqref{eq stime vpq} in the proof of Proposition \ref{cor estimate francesco}, with the choice $d(x_0,x)\leq v(x_0)/2M $.
\end{proof}

\subsection{Barriers for the lines of curvature}

We need also to deduce that, on a large ball $B_S(x_0,v(x_0)/2M)$ as estimated in Proposition \ref{cor estimate francesco}, the lines of curvature of a maximal surface $S$ are closer and closer to being geodesics, in the sense that their intrinsic acceleration is small, as $\lambda(x_0)$ tends to $1$. 

\begin{cor} \label{cor acceleration}
For every $\delta\in(0,1)$ there exists a constant $M(\delta)$ such that, for every maximal surface $S$ in $\AdS^3$ and every point $x_0\in S$ with $\lambda(x_0)\geq\delta$, the lines of curvature of $S$ have intrinsic acceleration, inside the ball $B_S(x_0,v(x_0)/2M)$ for the intrinsic metric of $S$, bounded by:
$$||\nabla^S_{\dot\gamma_c}\dot\gamma_c(t)||\leq M(\delta)e^{-v(\gamma_c(t))}\,,$$
where $\gamma_c$ is a unit-speed parametrization of any portion of line of curvature of $S$ contained inside the ball $B_S(x_0,v(x_0)/2M)$.
\end{cor}
\begin{proof}
It turns out (see for instance \cite{Schlenker-Krasnov}) that the second fundamental form $\II$ of $S$ is the real part of a holomorphic quadratic differential for the complex structure underlying the induced metric on $S$. Let us denote by $q$ this holomorphic quadratic differential, so that $\II=Re(q)$. 
By Corollary \ref{cor autovalore grande}, assuming $\lambda(x_0)\geq\delta>0$, the eigenvalues of $S$ are nonzero on the geodesic ball $B_S(x_0,v(x_0)/2M)$, and thus also $q\neq 0$ on the same ball. Hence one can find a conformal chart for $B_S(x_0,v(x_0)/2M)$ for which $q=dz^2$.

In this coordinate the first fundamental form $I$ of $S$ has the form $e^{\ph}|dz|^2$, for some real function $\ph$. We claim that $\ph$ coincides with the function $\chi=-\ln|\lambda|$. Indeed, observe that the shape operator of $S$ has the form
$$B=
e^{-\ph}\II=e^{-\ph}Re(q)=e^{-\ph}\begin{pmatrix}
1 & 0 \\
0 & -1
\end{pmatrix}
=
\begin{pmatrix}
e^{-\ph} & 0 \\
0 & -e^{-\ph}
\end{pmatrix}\,.
$$
Since the eigenvalues of $B$ are $\pm\lambda$, assuming $\lambda>0$, we get $\lambda=e^{-\ph}$ and therefore $\ph=\chi$ as claimed.

In such coordinates, the lines of curvature of $S$ are the lines with constant coordinates $x=\Re(z)$ or $y=\Im(z)$. Let us denote by $e_x,e_y$ the orthonormal frame given by such lines of curvature. By a direct computation, one checks that
$$\nabla^S_{e_x}e_x=\left(-\frac{1}{2}e^{-\frac{\chi}{2}}\partial_y\chi\right)e_y\,,\quad\quad\nabla^S_{e_y}e_y=\left(-\frac{1}{2}e^{-\frac{\chi}{2}}\partial_x\chi\right)e_x\,.$$
Hence one has
$$||\nabla^S_{e_x}e_x||^2=\frac{1}{4}e^{-\chi}(\partial_y\chi)^2\,,  \quad\quad ||\nabla^S_{e_y}e_y||^2=\frac{1}{4}e^{-\chi}(\partial_x\chi)^2\,.$$
Observe that the gradient of $\chi$, for the induced metric on $S$, has squared norm $$||\grad\chi||^2=e^{-\chi}((\partial_x\chi)^2+(\partial_y\chi)^2)\,,$$
and thus one directly obtains $$||\nabla^S_{e_x}e_x||^2\leq\frac{1}{4}||\grad\chi||^2\,.$$
On the other hand, by Equation \eqref{eq gradient v lam}, we have 
$$||\grad \chi||^2=e^{2\chi}e^{-2v}||\grad v||^2\,.$$
Since by hypothesis $\lambda(x_0)$ is bounded away from zero by $\delta$, by Corollary \ref{cor autovalore grande} $e^\chi$ is uniformly bounded by some constant $C=C(\delta)$ on $B_S(x_0,v(x_0)/2M)$. Hence
$$||\nabla^S_{e_x}e_x||\leq C(\delta)Me^{-v}\,,$$
where $M$ is the constant of Proposition \ref{prop gradient estimate}, and the same holds for $||\nabla^S_{e_y}e_y||$. Upon relabeling the constant $M$, this concludes the proof.
\end{proof}

In the following, we will always fix $\delta\in(0,1)$ and denote by $M$ a larger constant satisfying the statement of both Proposition \ref{prop gradient estimate} and Corollary \ref{cor acceleration}.

Observe that, given a totally geodesic plane $P$, the surface at timelike distance $d$ from $P$ (in the past, say) is a complete convex constant mean curvature umbilical surface with shape operator $(\tan d)E$ at every point. This follows for instance by applying Equation \eqref{shape operator constant distance} of Lemma \ref{lemma formule prima seconda shape distanza costante}. Therefore, given a surface $S$ with future unit normal vector $N_0$ at the point $x_0\in S$, consider the totally geodesic plane $P$ which contains the point $y_0=\exp_{x_0}((\arctan{\bar\lambda})N_0)$ and is orthogonal to the timelike line through $x_0$ and $y_0$. Thus the surface at distance $\bar d=\arctan\bar\lambda$ in the past from $P$, which we denote by $U_{\bar\lambda}(x_0,N_0)$, is an umbilical constant mean curvature surface tangent to $S$ at $x_0$. The shape operator of $U_{\bar\lambda}(x_0,N_0)$ is $\bar\lambda E$.  See Figure \ref{fig:umbilic}.

\begin{figure}[htbp]
\centering
\includegraphics[height=6.5cm]{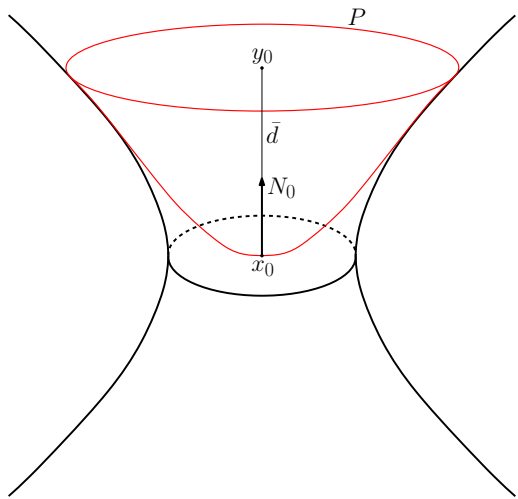}
\caption{The umbilical surface $U_{\bar\lambda}(x_0,N_0)$, for $\bar\lambda=\tan \bar d$, constructed as a parallel surface of the totally geodesic plane $P$. \label{fig:umbilic}}
\end{figure}

We shall denote by $l_c^+(x_0,a)$ (resp. $l_c^-(x_0,a)$) the segment of the line of curvature of $S$ for the positive (resp. negative) eigenvalue, which contains $x_0$ and whose extrema are at distance $a$ from $x_0$ for the induced metric.

The following lemma is a subtle application of a maximum principle argument. See also Figure \ref{fig:linescurvature} (for the statement) and Figure \ref{fig:comparison} (for the proof).

\begin{figure}[b]
\centering
\begin{minipage}[c]{.46\textwidth}
\centering
\includegraphics[height=6.5cm]{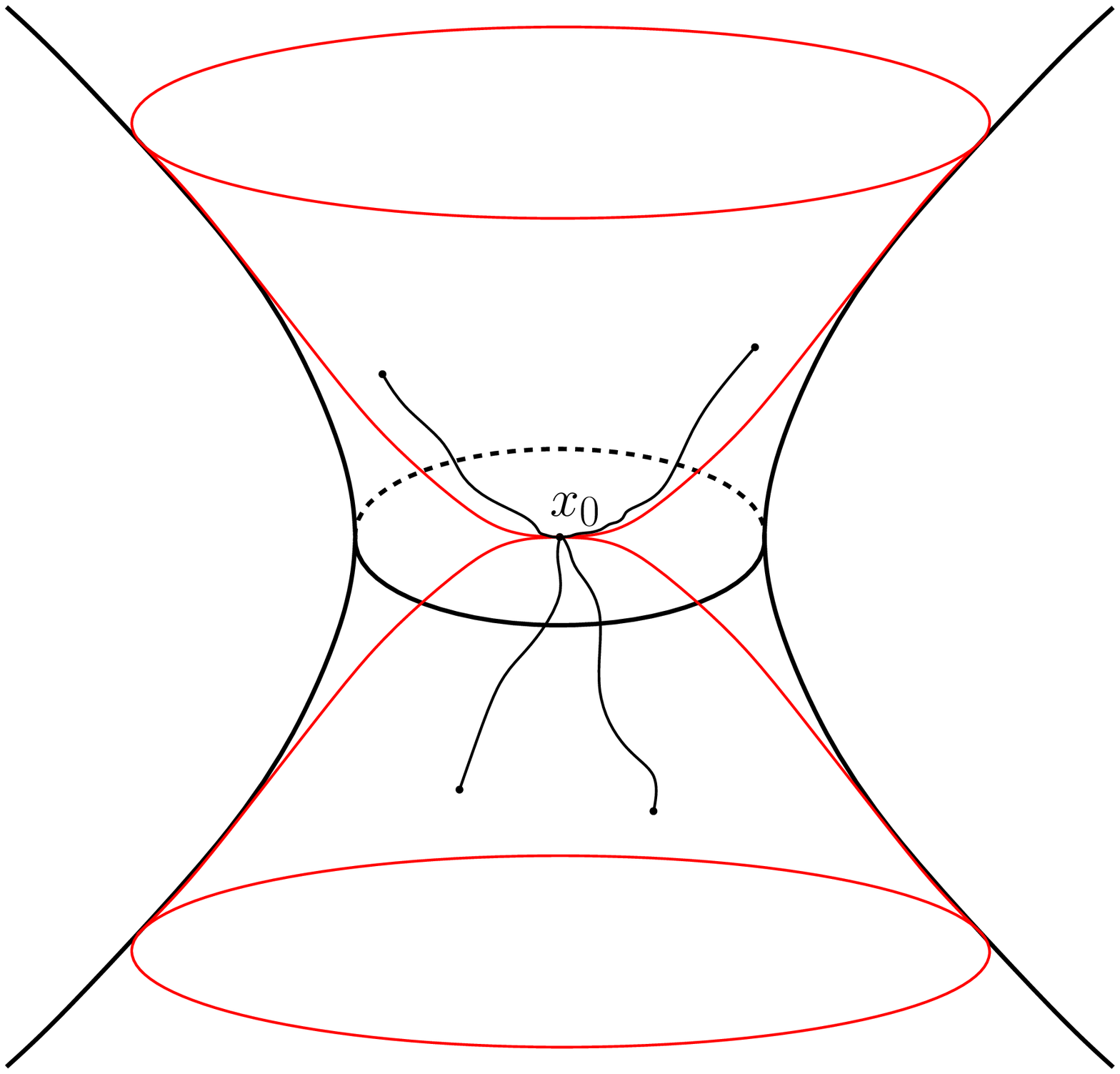} 
%\captionsetup{labelformat=empty}
\caption{Lines of curvature through $x_0$ and the surfaces $U^+_{\lambda_1}(x_0,N_0)$ and $U^-_{\lambda_1}(x_0,N_0)$} \label{fig:linescurvature}
\end{minipage}%
\hspace{6mm}
\begin{minipage}[c]{.46\textwidth}
\centering
\includegraphics[height=6.5cm]{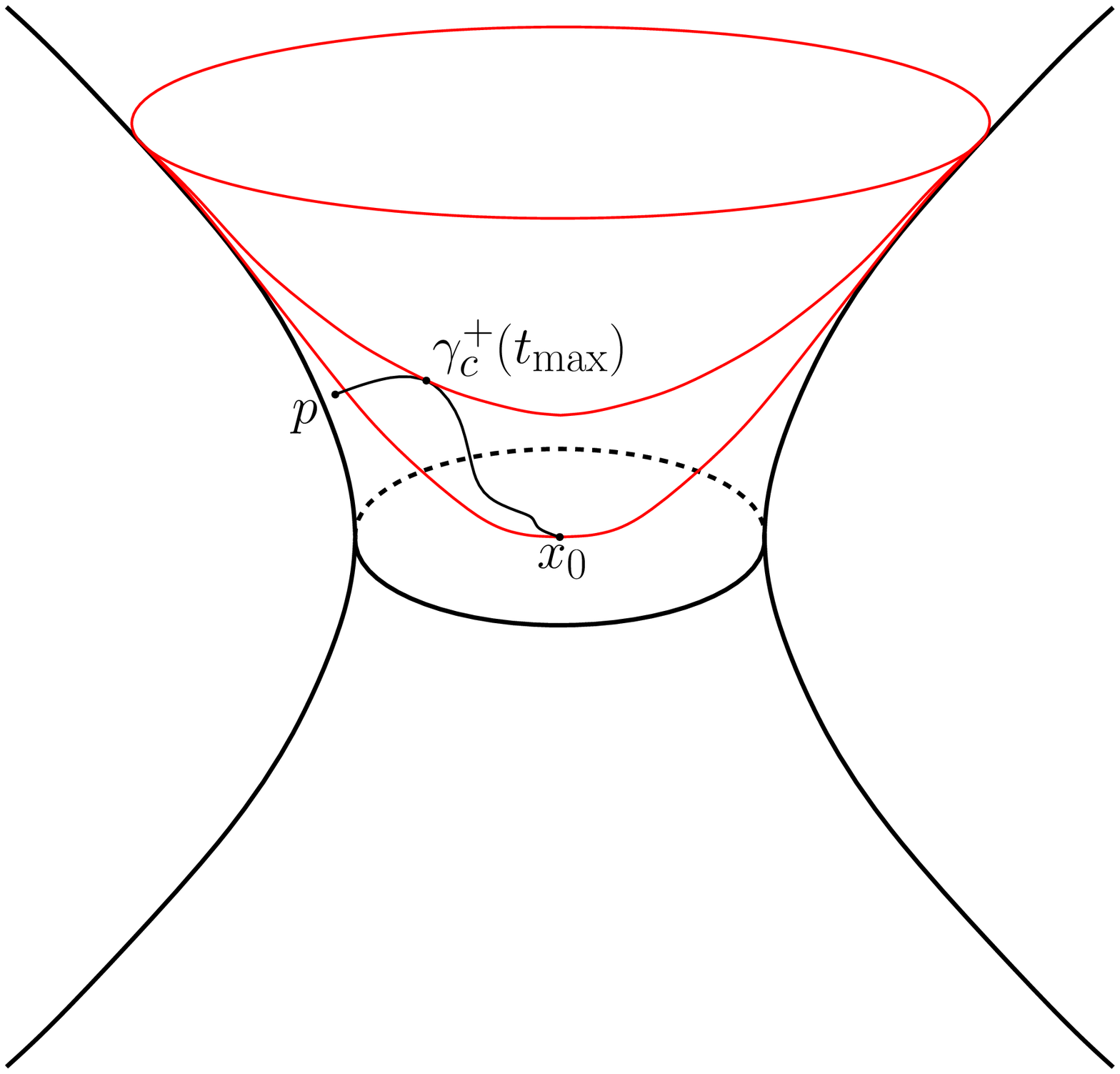}
%\captionsetup{labelformat=empty}
\caption{The proof of Lemma \ref{lemma calotta} is an argument by contradiction.} \label{fig:comparison}
\end{minipage}
\end{figure}

\begin{lemma} \label{lemma calotta}
There exists a constant $\delta\in(0,1)$ as follows. Suppose $S$ is a maximal surface with future unit normal vector $N_0$ at $x_0$ and with $\lambda(x_0)=1-e^{-v(x_0)}\geq \delta$. Then $l_c^+(x_0,a)$ is entirely contained in the convex side of the surface $U_{\lambda_1}(x_0,N_0)$, for $a=v(x_0)/2M$ and $\lambda_1=1-e^{-v(x_0)/4}$.
\end{lemma}
\begin{proof}
Choose $\delta>0$ as in Corollary \ref{cor acceleration}, and suppose \emph{ab absurdum} that $p\in l_c^+(x_0,a)$ is strictly in the past of $U_{\lambda_1}(x_0,N_0)$.

Recall that, for a spacelike curve $\gamma$ in a Lorentzian manifold, the curvature of $\gamma$ is defined as $\kappa=\sqrt{|\langle \nabla^S_{\dot\gamma}\dot\gamma,\nabla^S_{\dot\gamma}\dot\gamma\rangle|}$. If $\gamma_c^+:[-a,a]\to l_c^+(x_0,a)$ is a unit-speed parameterization of $l_c^+(x_0,a)$, we have
$$\nabla_{\dot\gamma_c^+}\dot\gamma_c^+=\nabla^S_{\dot\gamma_c^+}\dot\gamma_c^+ +\lambda N\,,$$
where $N$ is the unit future-directed normal vector field of the maximal surface $S$. Hence
$$\kappa^2=\lambda^2-||\nabla^S_{\dot\gamma_c^+}\dot\gamma_c^+||^2\,.$$
On the other hand, let $T$ be the timelike plane spanned by $\dot\gamma_c^+$ and $N_0$, and let $\gamma_1$ be a unit-speed parameterization of the spacelike curve $U_{\lambda_1}(x_0,N_0)\cap T$. Such curve is a geodesic of $U_{\lambda_1}(x_0,N_0)$, by a classical argument of symmetry. 

We first claim that $\gamma_c^+(-\epsilon,\epsilon)$ is contained in the future of $U_{\lambda_1}(x_0,N_0)$ for some $\epsilon>0$. Indeed, if this were not the case, the curvature of $\gamma_1$ at $x_0$ should be larger than the curvature of $\gamma_c^+$. Since $\gamma_1$ is geodesic for $U_{\lambda_1}(x_0,N_0)$, the curvature of $\gamma_1$ is
$$\kappa_1=\lambda_1\,.$$
By Corollary \ref{cor autovalore grande}  we have $v(x)\geq v(x_0)/2$ on $B_S(x_0,v(x_0)/2M)$, 
and by Corollary \ref{cor acceleration}, the intrinsic acceleration of $l_c^+(x_0,a)$ is bounded by 
$||\nabla^S_{\dot\gamma_c^+}\dot\gamma_c^+||\leq M e^{-v}\leq M e^{-v(x_0)/2}$.
Hence
\begin{align*}
\kappa^2&=\lambda^2-||\nabla^S_{\dot\gamma_c^+}\dot\gamma_c^+||^2\geq (1-e^{-v(x_0)/2})^2-Me^{-v(x_0)/2} \\
&=1-(2+M)e^{-v(x_0)/2}+e^{-v(x_0)}\\ 
&\geq 1-e^{-v(x_0)/4}
\geq(1-e^{-v(x_0)/4})^2=\lambda_1^2=\kappa_1^2\,.
\end{align*}
We have replaced $\delta$ by a larger number if necessary, and used the assumption $\lambda(x_0)\geq \delta$.
This gives a contradiction and concludes the claim.

Now consider the function $d:[-a,a]\to\R$, where $d(t)$ is the signed distance of $\gamma_c^+(t)$ from the surface $U_{\lambda_1}(x_0,N_0)$. The function $d$ is positive in the interval $(-\epsilon,\epsilon)$ by the previous claim, and negative at the point $t_0$ such that $\gamma_c^+(t_0)=p$. Hence $d$ must achieve a maximum $d_{\text{max}}=d(t_{\text{max}})>0$. At the point $\gamma_c^+(t_{\text{max}})$, the curve $\gamma_c^+$ is therefore tangent to the surface $V$ at distance $d_{\text{max}}$ from $U_{\lambda_1}(x_0,N_0)$. Again by Lemma \ref{lemma formule prima seconda shape distanza costante}, if $d_1$ is such that $\lambda_1=\tan d_1$, then $V$ is an umbilical constant mean curvature convex surface, whose shape operator at every point is $\tan(d_1-d_{\text{max}})E$. Denote $\lambda_2=\tan(d_1-d_{\text{max}})$ and observe that $\lambda_2<\lambda_1$.

By a similar argument as the previous claim, we compare the curve $\gamma_c^+$ to the curve $\gamma_2$, which parameterizes the intersection of $V$ with the timelike plane spanned by $\dot\gamma_c^+(t_{\text{max}})$ and $N(\gamma_c^+(t_{\text{max}}))$. We remark that in this case, $\gamma_2$ is not a geodesic for $V$. Since $t_{\text{max}}$ is a maximum point, the curvature $\kappa$ of $\gamma_c^+$ at $t_{\text{max}}$ needs to be smaller than the curvature $\kappa_2$ of $\gamma_2$.
But by the same computation,
$$\kappa^2\geq \lambda_1^2\geq \lambda_2^2\geq \lambda_2^2-||\nabla^V_{\dot\gamma_2}\dot\gamma_2||^2=\kappa_2^2\,.$$
This gives a contradiction and thus concludes the proof.
%Since at the point $x_0$ the eigenvalue $\lambda(x_0)$ of the shape operator of $S$ is strictly larger than the eigenvalue $\bar\lambda$ of $U_{\lambda_1}(x_0,N_0)$, and they have the same eigenspace by construction,  there exists $\epsilon>0$ such that $l_c^+(x_0,\epsilon)$ is contained in the future of $U_{\bar\lambda}(x_0,N_0)$.
%Consider the region $V$ between $U_{\bar\lambda}(x_0,N_0)$ and $P$, recalling $U_{\bar\lambda}(x_0,N_0)$ is the surface at distance $\bar d=\arctan\bar\lambda$ from $P$. The region $V$ is foliated by the surfaces at distance $d$ from $P$, for $d\in(0,\bar d)$. Therefore the function $d$ achieves a minimum on $l_c^+(x_0,\epsilon)$. Hence 
\end{proof}

\begin{remark}
Clearly a statement analogous to Lemma \ref{lemma calotta} holds for the line of curvature  $l_c^-(x_0,a)$ which, under the same assumptions, is entirely contained below the concave umbilical surface $U^-_{\lambda_1}(x_0,N_0)$, tangent to $S$ at $x_0$, obtained as the surface at distance $\tan\lambda_1$ in the future of a totally geodesic plane. See Figure \ref{fig:linescurvature}.
\end{remark}

\subsection{Estimating the width from below}
Observe that, in the extreme situation $\lambda_1=1$, the umbilical surfaces $U^+_{1}(x_0,N_0)$ and $U^-_{1}(x_0,N_0)$ have the following good property. Take two timelike planes $T_1$ and $T_2$ which intersect orthogonally in the timelike geodesic through $x_0$, directed by the vector $N_0$. Then $U^+_{1}(x_0,N_0)\cap T_1$ is a geodesic of $U^+_{1}(x_0,N_0)$, and $U^-_{1}(x_0,N_0)\cap T_2$ is a geodesic of $U^-_{1}(x_0,N_0)$. The endpoints at infinity of $U^+_{1}(x_0,N_0)\cap T_1$ determine a spacelike entire line of $\AdS^3$, which is dual to the line determined by the endpoints at infinity of $U^-_{1}(x_0,N_0)\cap T_2$. Hence the width of the convex hull of the four points at infinity is $\pi/2$. Indeed, the curves $U^+_{1}(x_0,N_0)\cap T_1$ and $U^-_{1}(x_0,N_0)\cap T_2$ are lines of curvature for a horospherical surface.
See Figures \ref{fig:strategy} and \ref{fig:horospherical}. Roughly speaking, in this subsection we want to quantify ``how close'' we get to this situation when $\lambda(x_0)$ (and thus also $\lambda_1$) approaches $1$. 

\begin{figure}[htbp]
\centering
\begin{minipage}[c]{.46\textwidth}
\centering
\includegraphics[height=7cm]{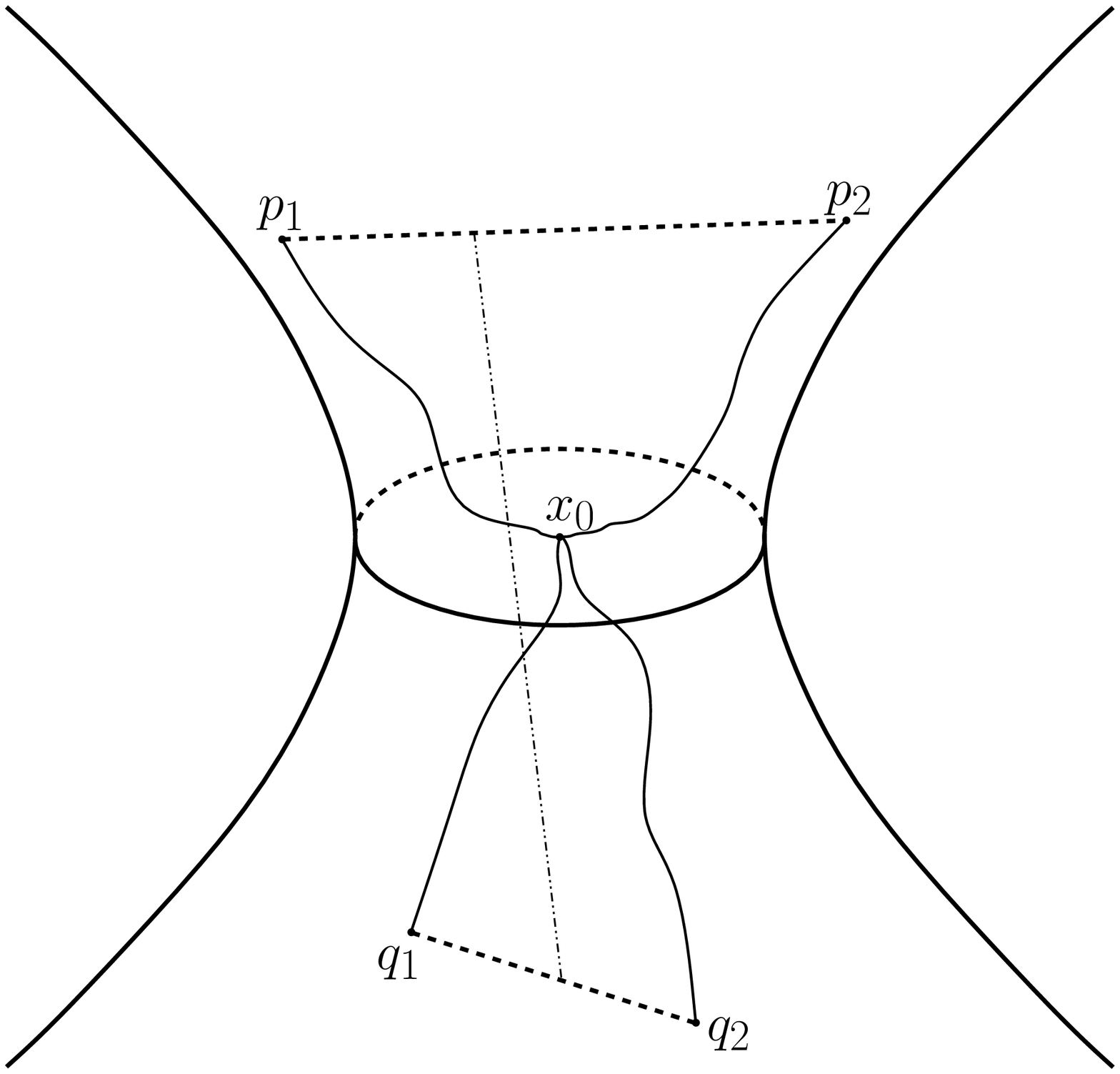} 
%\captionsetup{labelformat=empty}
\caption{We want to estimate the timelike distance between the lines $\overline{p_1p_2}$ and $\overline{q_1q_2}$, where $p_1,p_2,q_1,q_2$ are endpoints of segments on the lines of curvature.} \label{fig:strategy}
\end{minipage}%
\hspace{6mm}
\begin{minipage}[c]{.46\textwidth}
\centering
\includegraphics[height=7cm]{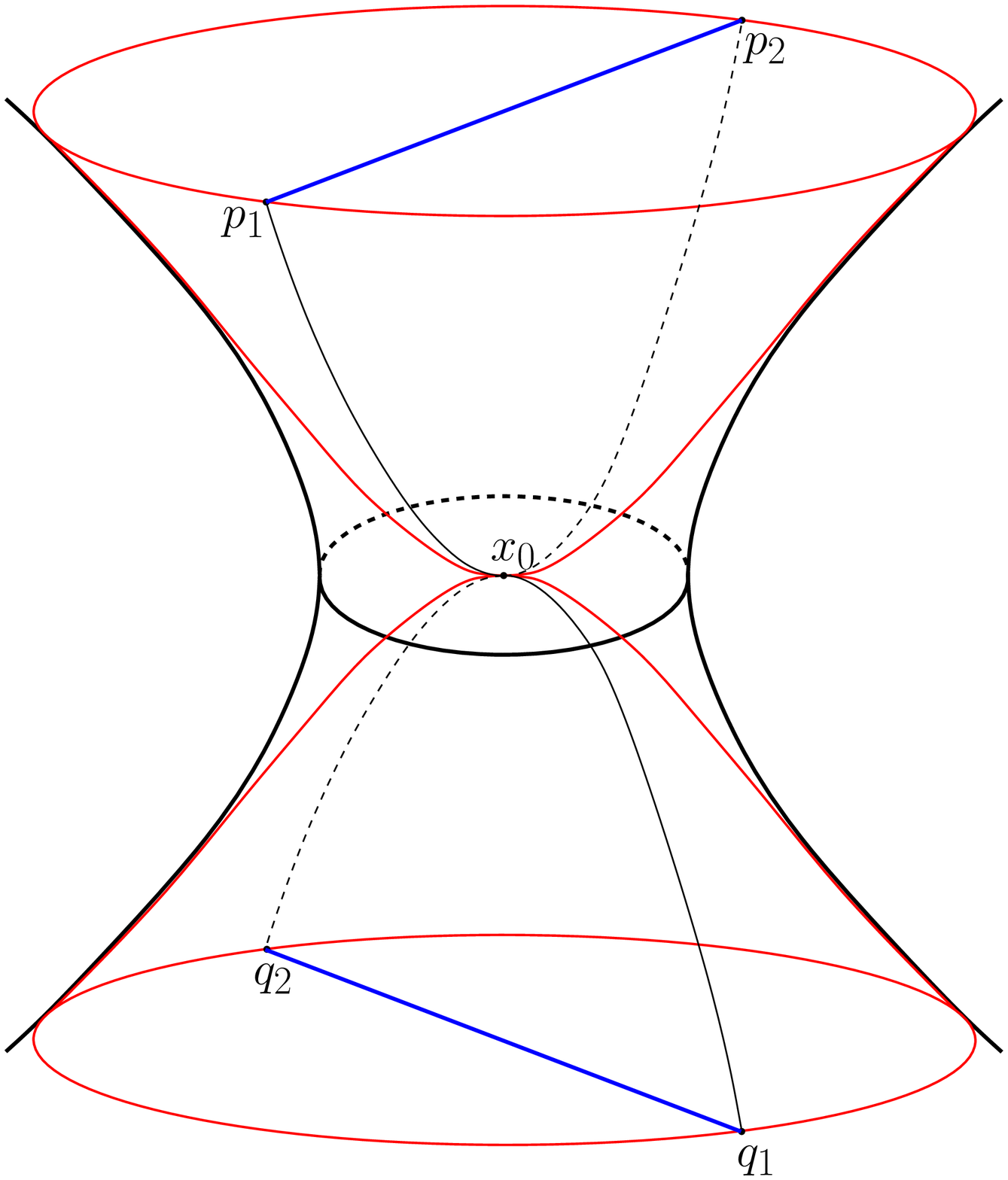}
%\captionsetup{labelformat=empty}
\caption{The configuration is optimal on a horospherical surface, when the lines of curvature are planar curves (and geodesics), and $\overline{p_1p_2}$ and $\overline{q_1q_2}$ are dual lines.} \label{fig:horospherical}
\end{minipage}
\end{figure}

Given a timelike totally geodesic plane $T=w^\perp$ (where $w\in\R^{2,2}$ is a spacelike vector of unit norm) and a point $p\in\AdS^3$, it is easy to see that $|\langle p,w\rangle|$ represents the hyperbolic sine of the length of the spacelike segment $\overline{pp_0}$, such that $p_0\in T$ and the line containing  $\overline{pp_0}$ is orthogonal to $T$ at $p_0$. If $T_r$ denotes the (timelike) surface composed of points for which this (signed) length is $r$, the following lemma gives an estimate on how the lines of curvature of a maximal surface escape from the surfaces $T_r$. Here $T=T_0$ is chosen as the timelike plane orthogonal to the line of curvature at the base point $x_0$.

\begin{lemma} \label{lemma estimate product}
There exist constants $\delta\in(0,1)$ and $t_0\geq 0$ as follows. Let $S$ be any maximal surface with $\lambda(x_0)=1-e^{-v(x_0)}\geq \delta$, and let $\gamma_c:[0,a]\to\AdS^3$ be a unit-speed parameterization of a line of curvature of $S$ with $\gamma_c(0)=x_0$, where $a=v(x_0)/2M$. If $\ph(t)=\langle \gamma_c(t),\dot\gamma_c(0)\rangle$, then
$$\ph(t)\geq e^t$$
for all $t\in[t_0,a]$.
\end{lemma}
\begin{proof}
Let us compute
$$\dot\ph(t)=\langle\dot\gamma(t),\dot\gamma(0)\rangle\,.$$
We now think $\gamma(t)$ as a point of $\wAdS\subset \R^{2,2}$. Denote $\psi(t)=\langle N(t),\dot\gamma(0)\rangle$, where we write $N(t)=N(\gamma(t))$ by a slight abuse of notation. Hence 
$$\ddot\gamma(t)=\gamma(t)+\lambda(\gamma(t)) N(\gamma(t))+\nabla^S_{\dot\gamma(t)}\dot\gamma(t)\,,$$
and therefore
$$\ddot\ph(t)=\ph(t)+\lambda(t)\psi(t)+\langle\nabla^S_{\dot\gamma(t)}\dot\gamma(t),\dot\gamma(0)\rangle\,.$$
If we denote $\rho(t)=\dot\ph(t)=\langle\dot\gamma(t),\dot\gamma(0)\rangle$ and $\alpha(t)=\langle\nabla^S_{\dot\gamma(t)}\dot\gamma(t),\dot\gamma(0)\rangle$, the triple $(\ph,\psi,\rho)$ solves the (non-linear, non-autonomous) system of ODEs
\begin{equation} \label{system}
\begin{cases}
\dot\ph(t)=\rho(t) \\
\dot\psi(t)=\lambda(t)\rho(t) \\
\dot\rho(t)=\ph(t)+\lambda(t)\psi(t)+\alpha(t)
\end{cases}\,,
\end{equation}
with the initial conditions
\begin{equation} \label{incond}
\begin{cases}
\ph(0)=0 \\
\psi(0)=0 \\
\rho(0)=1
\end{cases}\,.
\end{equation}

\begin{sublemma} \label{sublemma acc}
There exists $\delta\in(0,1)$ such that, if $\lambda(x_0)\geq\delta$, then
$$|\alpha(t)|< e^{-v(x_0)/4}\sqrt{1+\ph(t)^2+\psi(t)^2-\rho(t)^2}$$
for every $t\in[0,v(x_0)/2M]$.
In particular, setting $\epsilon=e^{-v(x_0)/4}$,
$$|\alpha(t)|< \epsilon\left(1+|\ph(t)|+|\psi(t)|\right)\,.$$
%provided $\ph(t),\psi(t)>0$.
\end{sublemma}
\begin{proof}
Since $\gamma$ is a unit-speed parameterization, $(\gamma(t),N(t),\dot\gamma(t),\nabla^S_{\dot\gamma(t)}\dot\gamma(t)/||\nabla^S_{\dot\gamma(t)}\dot\gamma(t)||)$ is an orthonormal frame for every $t$, provided $\nabla^S_{\dot\gamma(t)}\dot\gamma(t)\neq 0$. Hence
\begin{align*}
\dot\gamma(0)=&-\langle\dot\gamma(0),\gamma(t)\rangle\gamma(t)-\langle\dot\gamma(0),N(t)\rangle N(t)
+\langle\dot\gamma(0),\dot\gamma(t)\rangle\dot\gamma(t)+\langle\dot\gamma(0),\nabla^S_{\dot\gamma(t)}\dot\gamma(t)\rangle \frac{\nabla^S_{\dot\gamma(t)}\dot\gamma(t)}{||\nabla^S_{\dot\gamma(t)}\dot\gamma(t)||^2} \\
=& -\ph(t)\gamma(t)-\psi(t)N(t)+\rho(t)\dot\gamma(t)+\alpha(t)\frac{\nabla^S_{\dot\gamma(t)}\dot\gamma(t)}{||\nabla^S_{\dot\gamma(t)}\dot\gamma(t)||^2}
\,. \end{align*}
Therefore one gets
$$1=-\ph(t)^2-\psi(t)^2+\rho(t)^2+\frac{\alpha(t)^2}{||\nabla^S_{\dot\gamma(t)}\dot\gamma(t)||^2}\,.$$
Recalling that, from Corollary \ref{cor acceleration}, $||\nabla^S_{\dot\gamma}\dot\gamma||\leq M e^{-v(x_0)/2}< e^{-v(x_0)/4}$ for $\delta$ sufficiently large, one concludes the claim.
\end{proof}
By Corollary \ref{cor autovalore grande}, $\lambda(t)\geq \eta:=1-e^{-v(x_0)/2}$ for $t\in[0,v(x_0)/2M]$. We will compare the solution of the system \eqref{system} with the solution $(\ph_1(t),\psi_1(t),\rho_1(t))$ of the following system:
\begin{equation} \label{system2}
\begin{cases}
\dot\ph_1(t)=\rho_1(t) \\
\dot\psi_1(t)=\eta\rho_1(t) \\
\dot\rho_1(t)=(1-\epsilon)\ph_1(t)+(\eta-\epsilon)\psi_1(t)-\epsilon
\end{cases}\,,
\end{equation}
with the same initial conditions
\begin{equation} \label{incond2}
\begin{cases}
\ph_1(0)=0 \\
\psi_1(0)=0 \\
\rho_1(0)=1
\end{cases}\,.
\end{equation}

\begin{sublemma} \label{sublemma confronto}
Let $(\ph(t),\psi(t),\rho(t))$ be the solution of \eqref{system},\eqref{incond} and $(\ph_1(t),\psi_1(t),\rho_1(t))$ be the solution of \eqref{system2},\eqref{incond2}. For every interval $t\in[0,t_0)$ where $\ph(t),\psi(t)>0$, one has $\ph(t)\geq \ph_1(t)$.
\end{sublemma}
\begin{proof}
The system \eqref{system} can be written in the form of integro-differential equation:
$$\dot\rho(t)=\int_0^t \rho(s)ds+\lambda(t)\int_0^t \lambda(s)\rho(s)ds+\alpha(t)\,,$$
while system \eqref{system2} takes the form
$$\dot\rho_1(t)=(1-\epsilon)\int_0^t \rho_1(s)ds+(\eta-\epsilon)\int_0^t \eta\rho_1(s)ds-\epsilon\,.$$
By Sublemma \ref{sublemma acc}, as soon as $\ph(t),\psi(t)>0$ and $t\in[0,v(x_0)/2M]$,
$$\dot\rho(t)> (1-\epsilon) \int_0^t \rho(s)ds+(\eta-\epsilon)\int_0^t \eta\rho(s)ds-\epsilon\,.$$
Hence one gets
\begin{align*}
\dot\rho(t)-\dot\rho_1(t)>(1-\epsilon)\int_0^t (\rho(s)-\rho_1(s))ds+(\eta-\epsilon)\eta\int_0^t (\rho(s)-\rho_1(s))ds\,.
\end{align*}
This is enough to conclude that $\rho(t)>\rho_1(t)$ for all $t\in[0,t_0)$. Indeed, if $t=t_{\text{max}}$ were a maximal point for which $\rho(t)>\rho_1(t)$, then $\dot\rho(t_{\text{max}})-\dot\rho_1(t_{\text{max}})$ would still be strictly positive, thus giving a contradiction. As a direct consequence, $\ph(t)>\ph_1(t)$ for all $t\in[0,t_0)$.
\end{proof}

To conclude the proof of Lemma \ref{lemma estimate product}, it suffices to check by a direct computation that the solution of \eqref{system2} with initial conditions \eqref{incond2} is:
$$
\begin{pmatrix} \ph_1(t) \\ \psi_1(t) \\ \rho_1(t) \end{pmatrix}=\begin{pmatrix} \frac{\epsilon}{1-\epsilon+\eta(\eta-\epsilon)} \\ \frac{\eta\epsilon}{1-\epsilon+\eta(\eta-\epsilon)} \\ 0 \end{pmatrix}
+c_1 \begin{pmatrix} \frac{1}{\sqrt{1-\epsilon+\eta(\eta-\epsilon)}} \\ \frac{\eta}{\sqrt{1-\epsilon+\eta(\eta-\epsilon)}} \\ -1 \end{pmatrix}e^{-t\sqrt{1-\epsilon+\eta(\eta-\epsilon)}}
+c_2 \begin{pmatrix} \frac{1}{\sqrt{1-\epsilon+\eta(\eta-\epsilon)}} \\ \frac{\eta}{\sqrt{1-\epsilon+\eta(\eta-\epsilon)}} \\ 1 \end{pmatrix}e^{t\sqrt{1-\epsilon+\eta(\eta-\epsilon)}}\,,
$$
where the constants are $$c_1=\frac{1}{2}\left(-1-\frac{\epsilon}{\sqrt{1-\epsilon+\eta(\eta-\epsilon)}}\right)$$
and 
$$c_2=\frac{1}{2}\left(1-\frac{\epsilon}{\sqrt{1-\epsilon+\eta(\eta-\epsilon)}}\right)\,.$$
This shows in particular that $$\rho_1(t)=\cosh(t\sqrt{1-\epsilon+\eta(\eta-\epsilon)})-\frac{\epsilon}{\sqrt{1-\epsilon+\eta(\eta-\epsilon)}}\sinh(t\sqrt{1-\epsilon+\eta(\eta-\epsilon)})\,,$$
and is therefore positive for all $t\geq 0$ (since $\epsilon$ is small). Hence also $\ph_1(t)\geq 0$ and therefore $\ph(t)$ remains positive as well. Therefore the assumptions of Sublemma \ref{sublemma confronto} are satisfied for all $t\in[0,v(x_0)/2M]$. Hence there exists $t_0>0$ such that
$$\ph(t)\geq \ph_1(t)\geq e^t$$
for $t\in[t_0,a]$ as claimed.
\end{proof}

\begin{figure}[htbp]
\centering
\begin{minipage}[c]{.46\textwidth}
\centering
\includegraphics[height=6.8cm]{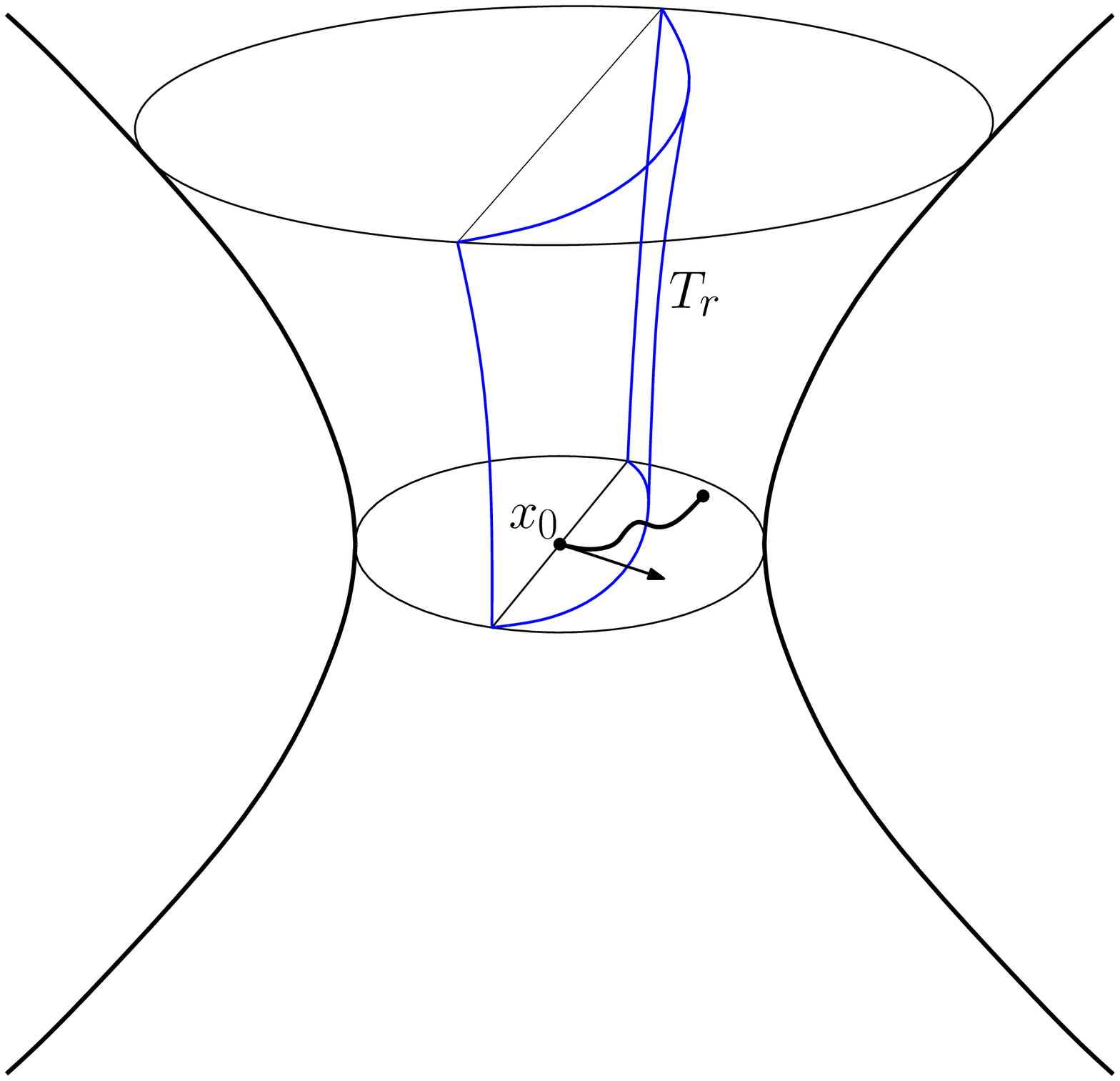} 
%\captionsetup{labelformat=empty}
\caption{Lemma \ref{lemma estimate product} asserts that the lines of curvature escape from the region bounded by the surface $T_r$.} \label{fig:surfacetr}
\end{minipage}%
\hspace{8mm}
\begin{minipage}[c]{.46\textwidth}
\centering
\includegraphics[height=6.8cm]{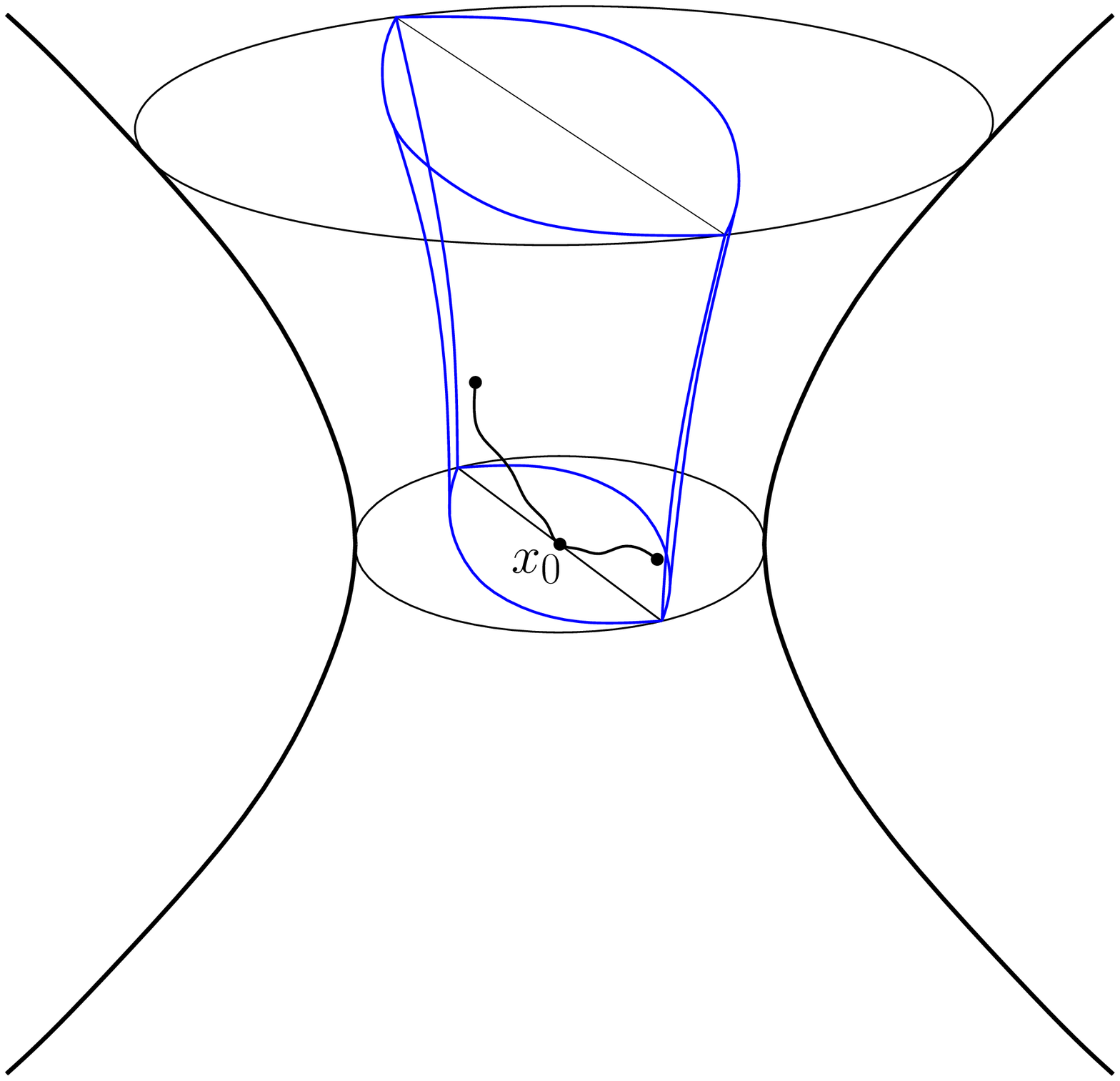}
%\captionsetup{labelformat=empty}
\caption{Lemma \ref{lemma estimate product 2} instead quantifies how the lines of curvature remain inside a thin slice bounded by two timelike surfaces.} \label{fig:fettina}
\end{minipage}
\end{figure}

\begin{remark} \label{rmk other estimate}
Using the same techniques as in Lemma \ref{lemma estimate product}, one can consider the function $\ph(t)=\langle \gamma_c(t),w\rangle$, where $w$ is the unit spacelike vector orthogonal to both $\dot\gamma_c(0)$ and to $N(x_0)$. One then similarly defines $\rho(t)=\dot\ph(t)=\langle\dot\gamma(t),w\rangle$ and 
 $\psi(t)=\langle \gamma(t),N(t)\rangle$. Hence the triple $(\ph,\psi,\rho)$ solves the system \eqref {system}, now with the initial conditions 
\begin{equation} \label{incondbis}
\begin{cases}
\ph(0)=0 \\
\psi(0)=0 \\
\rho(0)=0
\end{cases}\,.
\end{equation}
Since $\lambda\leq 1$, we consider the supersolution $(\ph_2(t),\psi_2(t),\rho_2(t))$ which solves the system:
\begin{equation} \label{system2bis}
\begin{cases}
\dot\ph_2(t)=\rho_2(t) \\
\dot\psi_2(t)=\rho_2(t) \\
\dot\rho_2(t)=(1+\epsilon)\ph_2(t)+(1+\epsilon)\psi_2(t)+\epsilon
\end{cases}\,,
\end{equation}
with the same initial conditions
\begin{equation} \label{incond2bis}
\begin{cases}
\ph_2(0)=0 \\
\psi_2(0)=0 \\
\rho_2(0)=0
\end{cases}\,.
\end{equation}
The latter system is easily solved, since it is equivalent to the equation $\ddot\rho(t)=2(1+\epsilon)\rho(t)$, and therefore one gets
$$\rho_2(t)=\frac{\epsilon}{\sqrt{2(1+\epsilon)}}\sinh(t\sqrt{2(1+\epsilon)})$$
and
$$\ph_2(t)=\psi_2(t)=\frac{\epsilon}{{2(1+\epsilon)}}(\cosh(t\sqrt{2(1+\epsilon)})-1)$$
as a solution to \eqref{system2bis}, \eqref{incond2bis}. Recalling that, under the usual assumptions, we have $\epsilon=e^{-v(x_0)/4}$. Thus for $t=a=v(x_0)/2M$,  $\ph_2(t)$ is dominated by an exponential of exponent $(v(x_0)/4)(-1+(2\sqrt{2}(1+\epsilon))/M)$, which is estimated from above by $e^{-v(x_0)/8}$ provided $M$ is sufficiently large (we can always replace $M$ by a larger constant, as we did several times before). In conclusion, we get $\ph(a)\leq\ph_2(a)\leq e^{-t/8}$. Of course the situation is symmetric, and one can prove the same upper bound for $-\ph(t)$.
\end{remark}

We report here the statement of a lemma, which was discussed in Remark \ref{rmk other estimate}.

\begin{lemma} \label{lemma estimate product 2}
There exists constans $\delta\in(0,1)$ and $t_0\geq 0$ as follows. Let $S$ be any maximal surface with future unit normal vector $N_0$ at $x_0$ and with $\lambda(x_0)=1-e^{-v(x_0)}\geq \delta$, and let $\gamma_c:[0,a]\to\AdS^3$ be a unit-speed parameterization of a line of curvature of $S$ with $\gamma_c(0)=x_0$, with $a=v(x_0)/2M$. If $\ph(t)=\langle \gamma_c(t),w\rangle$, where $w$ is the unit spacelike vector orthogonal to both $\dot\gamma_c(0)$ and to $N(x_0)$, then
$$|\ph(a)|\leq e^{-v(x_0)/8}\,.$$
\end{lemma}

To give an estimate from below for the width, we will consider a maximal surface with large principal curvatures at the point $x_0$, which we will assume to be the point $x_0=[0,0,1,0]$.
We are going to use again the coordinate system \eqref{cylindric coordinates}, which we write here again:
$$(r,\theta,\zeta)\mapsto [\cos\theta\sinh r,\sin\theta\sinh r,\cos\zeta\cosh r,\sin\zeta\cosh r]\,.$$
We are assuming the maximal surface $S$ is tangent at $x_0$ to the plane $\zeta=0$.
Hence the level sets $\zeta=c$ are totally geodesic planes orthogonal to the timelike like
which starts $x_0$ with initial tangent vector $N(x_0)$.

\begin{lemma} \label{lemma zeta}
Let $a=v(x_0)/2M$ and $\tan d_1=\lambda_1=1-e^{-v(x_0)/4}$. Let $p_1=\gamma_c^+(-a)$ and $p_2=\gamma_c^+(a)$ be the endpoints of the segment $l_c^+(x_0,a)$ of a line of curvature through $x_0$. Then $$|\sin(d_1-\zeta(p_i))|\leq\frac{\sin d_1}{\cosh(r(p_i))}\,.$$
In other words, $\zeta(p_i)\geq \bar\zeta$ where $\bar\zeta\leq d_1$ satisfies
$$\sin(d_1-\bar\zeta)=\frac{\sin d_1}{\cosh(r(p_i))}\,.$$
\end{lemma}
\begin{proof}
We know from Lemma \ref{lemma calotta} that $l_c^+(x_0,a)$ is entirely contained in the future-directed side of the surface $U_{\lambda_1}(x_0,N_0)$, for $a=v(x_0)/2M$ and $\lambda_1=1-e^{-v(x_0)/4}$. Recall that the surface $U_{\lambda_1}(x_0,N_0)$ is obtained as the surface at distance $d_1$ (past-directed) from the plane $\zeta=d_1$, where $\tan d_1=\lambda_1$. This plane is also defined by
$$\{\zeta=d_1\}=[0,0,\sin d_1,-\cos d_1]^T\,.$$ 
Observe that the product of $p_i$ and $[0,0,\sin d_1,-\cos d_1]$, in absolute value, is the sine of the timelike distance of $p_i$ from the plane $[0,0,\sin d_1,-\cos d_1]^T$. Hence if $(r(p_i),\theta(p_i),\zeta(p_i))$ are the coordinates of $p_i$, then 
\begin{align*}
\sin d_1&\geq |\langle p_i,[0,0,\sin d_1,-\cos d_1] \rangle|
\\ &=|\sin d_1 \cos(\zeta(p_i))\cosh(r(p_i))-\cos d_1 \sin(\zeta(p_i))\cosh(r(p_i))| \\
& =|\sin(d_1-\zeta(p_i))|\cosh(r(p_i))\,.
\end{align*}
from which the statement follows straightforwardly.
\end{proof}

\begin{figure}[htbp]
\centering
\includegraphics[height=6.5cm]{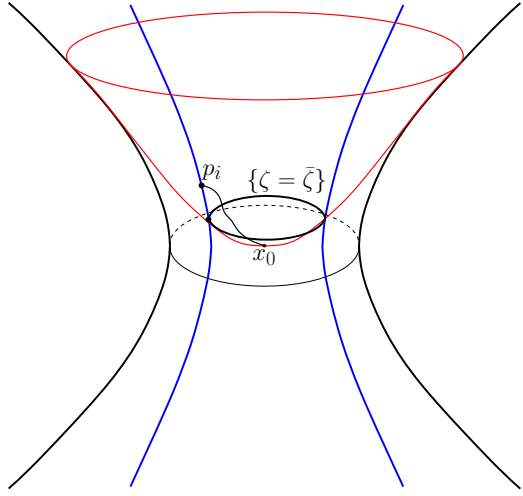}
\caption{In Lemma \ref{lemma zeta}, the intersection of the surface $U_{\lambda_1}(x_0,N_0)$ with the cylinder which contains $p_i$ lies in the plane $\{\zeta=\bar\zeta\}$. \label{fig:horizontal}}
\end{figure}

%\begin{lemma}
%Let $P_1$, $P_2$ be two totally geodesic plane having a common orthogonal geodesic segment $\overline{s_1s_2}$ (where $s_i\in P_i$) of timelike length $d_{\AdS^3(s_1,s_2)}=\bar\zeta$.  
%\end{lemma}

\begin{proof}[Proof of Theorem \ref{theorem width lambda big}]
As observed earlier, we can assume that there exists a point $x_0$ (and we set $x_0=[0,0,1,0]$) with $\lambda(x_0)\geq \delta$. Composing with an isometry, we also assume that (in the double cover $\wAdS$) $N_0=(0,0,0,1)$ and that the tangent vectors to the lines of curvature at $x_0$ are $\dot\gamma_c^+(0)=(1,0,0,0)$ and $\dot\gamma_c^+(0)=(0,1,0,0)$. Let as usual $p_1=\gamma_c^+(-a)$ and $p_2=\gamma_c^+(a)$ be the endpoints of the segment $l_c^+(x_0,a)$, and $q_1=\gamma_c^-(-a)$ and $q_2=\gamma_c^-(a)$ be the endpoints of the segment $l_c^-(x_0,a)$, for $a=v(x_0)/2M$. Recall also Figure \ref{fig:strategy}.

The width $w$ of the convex hull of $S$ is at least the supremum of the length of geodesic timelike segments which connect the spacelike segments $\overline{p_1p_2}$ and $\overline{q_1q_2}$, which we denote by $d_{\AdS^3}(\overline{p_1p_2},\overline{q_1q_2})$. Let $(r(p_i),\theta(p_i),\zeta(p_i))$ the coordinates of $p_i$ and $(r(q_i),\theta(q_i),\zeta(q_i))$ the coordinates of $q_i$. By Lemma \ref{lemma zeta}, $\zeta(p_i)\geq \bar\zeta$ and $\zeta(q_i)\leq -\bar\zeta$, for $i=1,2$.

Hence $d_{\AdS^3}(\overline{p_1p_2},\overline{q_1q_2})$ is certainly larger than 
$d_{\AdS^3}(\overline{p_1'p_2'},\overline{q_1'q_2'})$, where $p_i'$ has coordinates $(r(p_i),\theta(p_i),\bar\zeta)$ and $q_i'$ has coordinates $(r(q_i),\theta(q_i),-\bar\zeta)$. Compare also Figure \ref{fig:decrease}. Indeed, every timelike segment connecting $\overline{p_1p_2}$ and $\overline{q_1q_2}$ can be continued to a timelike segment connecting $\overline{p_1'p_2'}$ and $\overline{q_1'q_2'}$ of larger timelike length.

Now, the segment $\overline{p_1'p_2'}$ is clearly contained in the plane $\zeta=\bar\zeta$, and it contains a point $i_+$ with coordinates $(r(i_+),\theta(i_+)=\pi/2,\bar\zeta)$. See Figure \ref{fig:intersections}.
Therefore
$$i_+=[0,\sinh r(i_+),\cos\bar\zeta\cosh r(i_+),\sin\bar\zeta\cosh r(i_+)]\,.$$
Analogously, the segment $\overline{q_1'q_2'}$ contains the point
$$i_-=[\sinh r(i_-),0,\cos\bar\zeta\cosh r(i_-),-\sin\bar\zeta\cosh r(i_-)]\,.$$ 
Hence one can give the following bound for the width $w$:
$$\cos w\leq |\langle i_+,i_- \rangle|=2\cosh r(i_-)\cosh r(i_+)\cos\bar\zeta\sin\bar\zeta=\cosh r(i_-)\cosh r(i_+)\cos(2\bar\zeta)\,,$$
which is equivalent to
$$(\tan w)^2=\frac{1}{(\cos w)^2}-1\geq \frac{1}{(\cos(2\bar\zeta))^2}\frac{1}{(\cosh r(i_+))^2}\frac{1}{(\cosh r(i_-))^2}-1\,.$$

We now want to estimate the factors $\cos(2\bar\zeta)$ and $\cosh r(i_+)$, $\cosh r(i_-)$. For the former, let us write $\bar\zeta=d_1-(d_1-\bar\zeta)$ and compute
\begin{equation} \label{camomilla}
\cos(2\bar\zeta)=\cos(2d_1)\cos(2(d_1-\bar\zeta))+\sin(2d_1)\sin(2(d_1-\bar\zeta))\,.
\end{equation}
Recall that $\tan d_1=\lambda_1=1-e^{-v(x_0)/4}$, hence $(\cos d_1)^2=1/(1+\lambda_1^2)$ and
\begin{equation} \label{manzanilla}
\cos(2d_1)=(\cos d_1)^2-(\sin d_1)^2=\frac{1}{1+\lambda_1^2}-\frac{\lambda_1^2}{1+\lambda_1^2}\leq 2(1-\lambda_1)=2 e^{-v(x_0)/4}\,.
\end{equation}
For the second term of the RHS of Equation \eqref{camomilla}, using Lemma \ref{lemma zeta}, we have 
\begin{equation} \label{manzanilla2}
\sin(2(d_1-\bar\zeta))=2\sin(d_1-\bar\zeta)\cos(d_1-\bar\zeta)\leq \frac{2\sin d_1}{\cosh r(p_i)}\,.
\end{equation}
Observe that 
$$\cosh r(p_i)\geq |\sinh r(p_i)|\geq |\langle p_i,(1,0,0,0)\rangle|=|\langle p_i,\dot\gamma_c^+(0)\rangle|\geq e^{v(x_0)/2M}\,,$$
where in the last step we have used Lemma \ref{lemma estimate product} (up to changing the orientation of the parameterization $\gamma_c^+$ for one of the two points $p_i$). Therefore using the inequalities of Equations \eqref{manzanilla} and \eqref{manzanilla2} in \eqref{camomilla}, and relabeling $M$ by a yet larger constant, we get
$$\frac{1}{\cos(2\bar\zeta)}\geq e^{v(x_0)/M}\,,$$
provided $\lambda(x_0)$ is larger than the constant $\delta$. 

On the other hand, observe that $r(i_+)$ is the distance in the hyperbolic plane $\zeta=\bar\zeta$ of the point $i_+$ from the geodesic defined by $\theta=0$. By the convexity of the distance function, $r(i_+)$ is less than the maximum between the distance of $p_1'$ and $p_2'$ from the line $\theta=0$, which remains bounded by Lemma \ref{lemma estimate product 2} (see Remark \ref{rmk other estimate}). Actually, $r(i_+)$ tends to zero as $v(x_0)\to\infty$, and the same holds for $r(i_-)$. Hence the factors $\cosh r(i_+)$ and $\cosh r(i_-)$ remain bounded, and this concludes the proof that
$$\tan w\geq e^{v(x_0)/M}=\left(\frac{1}{1-\lambda(x_0)}\right)^{1/M}\,.$$
Hence, by a continuity argument, also the inequality
$$\tan w\geq \left(\frac{1}{1-||\lambda||_\infty}\right)^{1/M}$$
holds.
\end{proof}

\begin{figure}[htbp]
\centering
\begin{minipage}[c]{.46\textwidth}
\centering
\includegraphics[height=6.5cm]{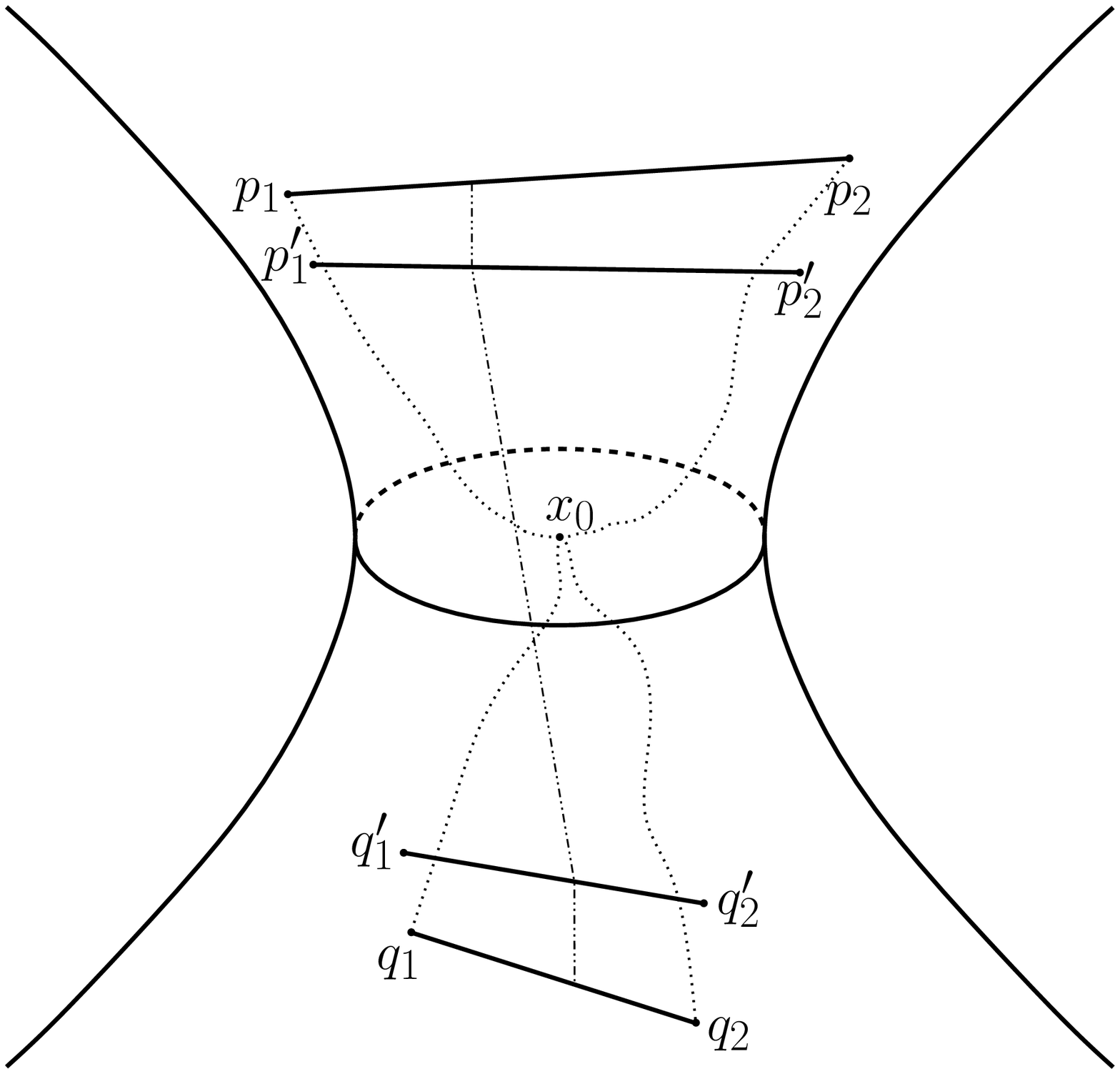} 
%\captionsetup{labelformat=empty}
\caption{The width is not increased when replacing the points $p_1,p_2,q_1,q_2$ by $p_1',p_2',q_1',q_2'$.} \label{fig:decrease}
\end{minipage}%
\hspace{6mm}
\begin{minipage}[c]{.46\textwidth}
\centering
\includegraphics[height=6.5cm]{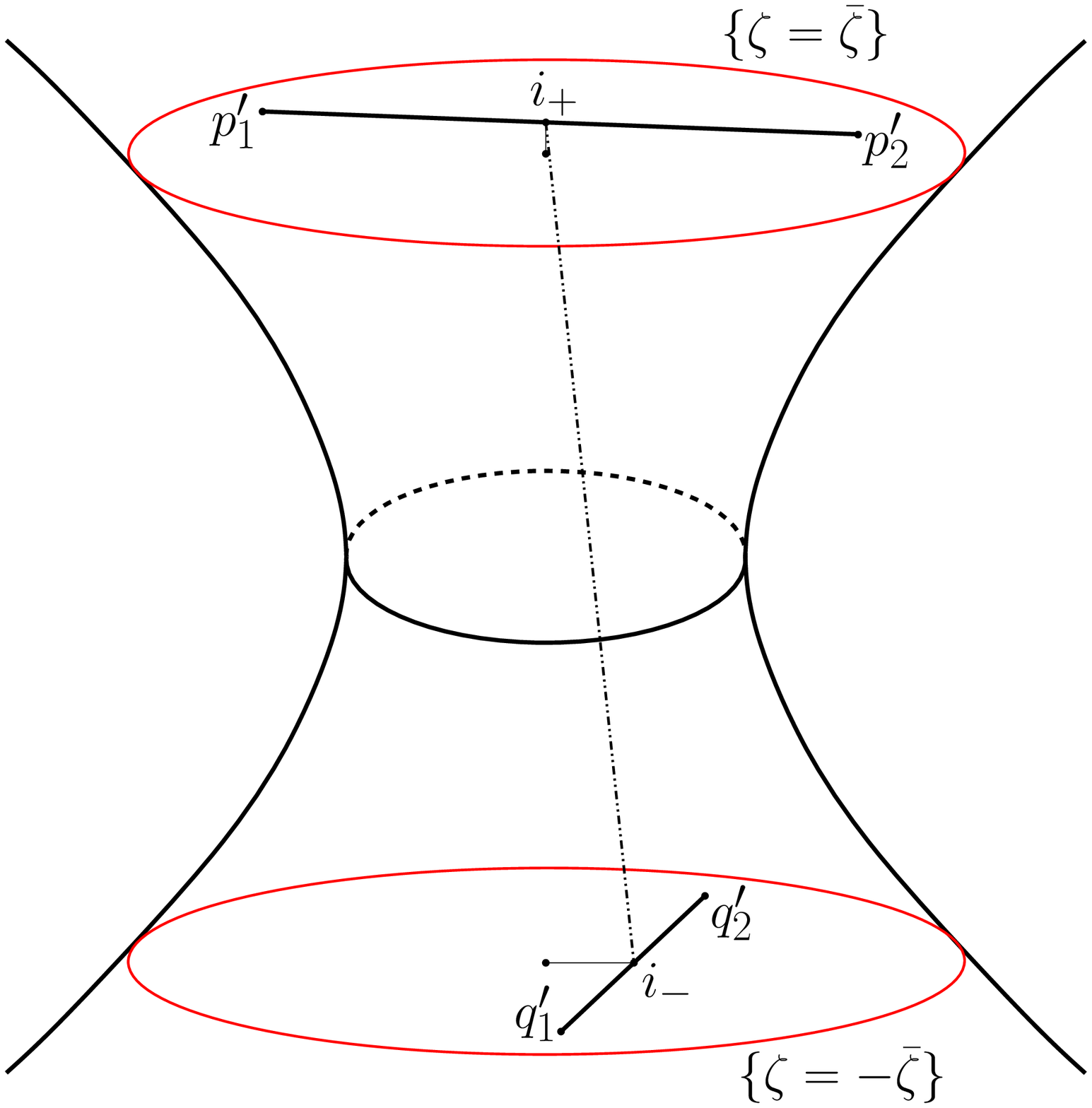}
%\captionsetup{labelformat=empty}
\caption{The position of the points $i_+$ and $i_-$, which are used to derive the lower bound on the width.} \label{fig:intersections}
\end{minipage}
\end{figure}

%------------------------------------------------------

\section{Application: minimal Lagrangian quasiconformal extensions} \label{section universal}

%\begin{remark} \label{remark completeness}
%A consequence of the results proved in \cite{bon_schl} is that the maximal surface $S$ with bounded principal curvatures, spanning the graph of a quasisymmetric homeomorphism, is complete. In fact, there is a bi-Lipschitz diffeomorphism from $S$ to $\Hyp^2$, and $\Hyp^2$ is complete. Such diffeomorphism is described also in Subsection \ref{subs Minimal Lagrangian extension}.
%\end{remark}

We begin this section by briefly introducing the theory of quasiconformal mappings and universal Teichm\"uller space. %We will give a brief account of the very rich and developed theory. 
Useful references are \cite{gardiner, gardiner2, ahlforsbook, fletchermarkoviclibro}.
Next, we discuss the applications of Theorem \ref{estimate principal curvatures and width ads}, Theorem \ref{theorem width lambda big} and Proposition \ref{estimate principal curvatures and width ads reverse}, proving Theorem \ref{minimal lag extension}, Theorem \ref{minimal lag extension large}, Theorem \ref{minimal lag extension below} and Corollary \ref{corollary extension}.

\subsection{Quasiconformal mappings} \label{subsection universal}

We recall the definition of quasiconformal map.

\begin{defi} \label{definitions quasiconformal map}
Given a domain $\Omega\subset\C$, an orientation-preserving homeomorphism $$f:\Omega\to f(\Omega)\subset\C$$ is \emph{quasiconformal} if $f$ is absolutely continuous on lines and there exists a constant $k<1$ such that $$|\partial_{\overline z}f|\leq k |\partial_{z}f|\,.$$
\end{defi}

Let us denote $\mu_f=\partial_{\overline z}f/\partial_{z}f$, which is called \emph{complex dilatation} of $f$. This is well-defined almost everywhere, hence it makes sense to take the $L_\infty$ norm. Thus a homeomorphism $f:\Omega\to f(\Omega)\subset\C$ is quasiconformal if $||\mu_f||_\infty<1$. Moreover, a quasiconformal map as in Definition \ref{definitions quasiconformal map} is called $K$-\emph{quasiconformal}, where
$$K=\frac{1+k}{1-k}\,.$$
It turns out that the best such constant $K\in[1,+\infty)$ represents the \emph{maximal dilatation} of $f$, i.e. the supremum over all $z\in\Omega$ of the ratio between the major axis and the minor axis of the ellipse which is the image of a unit circle under the differential $d_z f$. 

It is known that a $1$-quasiconformal map is conformal, and that the composition of a $K_1$-quasiconformal map and a $K_2$-quasiconformal map is $K_1 K_2$-quasiconformal. Hence composing a quasiconformal map with conformal maps does not change the maximal dilatation. 

Actually, there is an explicit formula for the complex dilatation of the composition of two quasiconformal maps $f,g$ on $\Omega$:
\begin{equation} \label{composition belt differentials}
\mu_{g\circ f^{-1}}=\frac{\partial_z f}{\overline{\partial_{ z} f}}\frac{\mu_g-\mu_f}{1-\overline{\mu_f}\mu_g}\,.
\end{equation}
Using Equation \eqref{composition belt differentials}, one can see that $f$ and $g$ differ by post-composition with a conformal map if and only if $\mu_f=\mu_g$ almost everywhere.

%\noindent {\bf{Measurable Riemann mapping Theorem}}. Given any measurable function $\mu$ on $\C$ there exists a unique quasiconformal map $f:\C\to\C$ such that $f(0)=0$, $f(1)=1$ and $\mu_f=\mu$ almost everywhere in $\C$.

%The uniqueness part of Measurable Riemann mapping Theorem means that every two solutions (which can be thought as maps on the sphere $\widehat\C$) of the equation $$ (\partial_{z}f) \mu=\partial_{\overline z}f$$
%differ by post-composition with a M\"obius transformation of $\widehat \C$. 

%\subsection{Quasiconformal deformations of the disc} \label{subsec: quasiconformal model disc}

%We are now ready to introduce the first model of universal Teichm\"uller space. 

The connection between quasiconformal homeomorphisms of $\Hyp^2$ and quasisymmetric homeomorphisms of the boundary of $\Hyp^2$ is made evident by the following classical theorem (see \cite{ahl_beur}).
\\

\noindent {\bf{Ahlfors-Beuring Theorem}}. Every quasiconformal map $\Phi:\Hyp^2\to\Hyp^2$ extends to a quasisymmetric homeomorphism of $\RP^1=\partial\Hyp^2$. Conversely, any quasisymmetric homeomorphism $\phi:\RP^1\to \RP^1$ admits a quasiconformal extension to $\Hyp^2$.

\subsection{Minimal Lagrangian extension} \label{subs Minimal Lagrangian extension}

Our purpose is to give a quantitative description of \emph{minimal Lagrangian} extensions of a quasisymmetric homeomorphism $\phi$.

\begin{defi}
A diffeomorphism $\Phi:\Hyp^2\to\Hyp^2$ is minimal Lagrangian if $\Phi$ is area-preserving and the graph of $\Phi$ is a minimal surface in $\Hyp^2\times \Hyp^2$. 
\end{defi}

The following characterization of minimal Lagrangian diffeomorphisms is well-known. A proof can be found in \cite[Proposition 1.2.6]{jeremythesis}.

\begin{prop}
A diffeomorphism $\Phi:\Hyp^2\to\Hyp^2$ is minimal Lagrangian if and only if $\Phi^*(g_{\Hyp^2})=g_{\Hyp^2}(b\cdot,b\cdot)$, where $b\in\Gamma(\mathrm{End}(TM))$ is a bundle morphism such that:
\begin{itemize}
\item $b$ is self-adjoint for $g_{\Hyp^2}\,;$
\item $\det b=1\,;$
\item $d^{\nabla}b=0\,$.
\end{itemize}
\end{prop}

\noindent Here $d^{\nabla}$ is the exterior derivative, hence $d^{\nabla}b$ is the two-form defined by:
$$d^{\nabla}b(v,w)=\nabla_{\tilde v} (b(\tilde w))-\nabla_{\tilde w} (b(\tilde v))-b[\tilde v,\tilde w]\,,$$
where $\tilde v,\tilde w$ are vector fields which extend the vectors $v,w$ in a neighborhood of the base point.
The vanishing of $d^{\nabla}b$ is the so-called \emph{Codazzi condition}.

In \cite{bon_schl}, entire maximal surfaces of uniformly negative curvature were used to prove the following theorem:

\begin{theorem}[{\cite[Theorem 1.4]{bon_schl}}]
For every quasisymmetric homeomorphism $\phi:\RP^1\to \RP^1$, there exists a unique quasiconformal minimal Lagrangian extension $\Pml:\Hyp^2\to\Hyp^2$.
\end{theorem}

The key observation is that the maximal surface with $\partial_\infty S=gr(\phi)$ corresponds to the minimal Lagrangian extension $\Pml$ of $\phi$. The extension is given geometrically in the following way. Fix a totally geodesic plane $P$ in $\AdS^3$, which is a copy of hyperbolic plane. Given a point $x\in S$, we define two isometries $\Phi^x_l,\Phi^x_r\in\isom(\AdS^3)$ which map the tangent plane $T_x S$ to $P$. The first isometry $\Phi^x_l$ is obtained by following the left ruling of $\partial_\infty \AdS^3$. Analogously $\Phi^x_r$ for the right ruling. This gives two diffeomorphisms $\Phi_l$ and $\Phi_r$ from $S$ to $P$, by 
$$\Phi_l(x)=\Phi^x_l(x)\,,\qquad\qquad\Phi_r(x)=\Phi^x_r(x)\,.$$ 
The diffeomorphism $\Phi$ is then defined as 
$$\Pml=(\Phi_l)^{-1}\circ \Phi_r\,.$$ 
In \cite[Lemma 3.16]{Schlenker-Krasnov} it is shown that the pull-back of the hyperbolic metric $h$ of $P$ on $S$ by means of $\Phi_r$ and $\Phi_l$ is given by 
\begin{equation} \label{pullback left}
\Phi_l^*h=I((E+JB)\cdot,(E+JB)\cdot)\,,
\end{equation}
 and 
\begin{equation} \label{pullback right}
 \Phi_r^*h=I((E-JB)\cdot,(E-JB)\cdot)\,,
 \end{equation}
  where $I$ is the first fundamental form of $S$, $J$ is the almost-complex structure of $S$, $B$ the shape operator and $E$ the identity. We are now ready to give a relation between the principal curvatures of $S$ and the quasiconformal distortion of $\Phi$:

\begin{prop} \label{estimate principal curvatures and qc coefficient}
Given a maximal surface $S$ in $\AdS^3$, the quasiconformal distortion of the minimal Lagrangian map $\Phi:\Hyp^2\rar\Hyp^2$ at a point $y\in\Hyp^2$ is given by $$K(y)=\left(\frac{1+|\lambda(x)|}{1-|\lambda(x)|}\right)^2\,,$$
where $y=(\Phi_l)(x)$.
Therefore, by taking $K(\Pml)=\sup_y K(y)$, namely $K(\Pml)$ is the maximal dilatation of $\Phi$, the following holds:
$$K(\Pml)=\left(\frac{1+||\lambda||_\infty}{1- ||\lambda||_\infty}\right)^2\,.$$
\end{prop}
\begin{proof}
%We want to compute the Beltrami differential $\mu$ of $\Phi=(A_l)^{-1}\circ A_r$. We take isothermal coordinates for the induced metric on $S$, so that the complex structure is given by $$J=\begin{pmatrix} 0 & -1 \\ 1 & 0 \end{pmatrix}$$
%and we can also assume $B$ is diagonal at a fixed point $x\in S$, i.e.
%$$B=\begin{pmatrix} \lambda & 0 \\ 0 & -\lambda \end{pmatrix}$$
%hence 
%$$dA_l=E+JB=\begin{pmatrix} 1 & \lambda \\ \lambda & 1 \end{pmatrix}.$$
%We can easily decompose $dA_l$ in its holomorphic and antiholomorphic part: in complex notation,
%$dA_l(z)=z+\lambda i\bar z$ and so the Beltrami coefficient of $A_l$ is $\mu_l=\lambda i$. Analogously, the Beltrami coefficient at $x$ of $A_r$ is $\mu_r=-\lambda i$. By the formula for the composition of Beltrami coefficients (see \cite{gardiner2}), 
%$$|\mu|=\left|\frac{\mu_r-\mu_l}{1-\bar\mu_l\mu_r}\right|=\frac{2\lambda}{1+\lambda^2}.$$

Let $h$ be the hyperbolic metric of $P$; it follows from Equations \eqref{pullback left} and \eqref{pullback right} that $$\Phi^*h=h((E+JB)^{-1}(E-JB)\cdot,(E+JB)^{-1}(E-JB)\cdot).$$
The quasiconformal distortion of $\Phi$ at a fixed point $x$ can be computed as the ratio between ${\sup||\Phi_*(v)||}$ and ${\inf ||\Phi_*(v)||}$
where the supremum and the infimum are taken over all tangent vectors $v\in T_x P$ with $||v||=1$.
Since $B$ is diagonalizable with eigenvalues $\pm\lambda$, $(E+JB)^{-1}(E-JB)$ can be diagonalized to the form
$$\begin{pmatrix} \frac{1-\lambda}{1+\lambda} & 0 \\ 0 & \frac{1+\lambda}{1-\lambda} \end{pmatrix}\,.$$
Hence, assuming $0<\lambda(x)<1$, the quasiconformal distortion is given by
$$K(\Phi_l(x))=\left(\frac{1+\lambda(x)}{1-\lambda(x)}\right)^2\,,$$
as claimed.
\end{proof}
%Note that this formula can be easily inverted, to obtain the supremum $\Lambda$ of principal curvatures in terms of $K$: we have
%$$\Lambda=\frac{\sqrt{K}-1}{\sqrt{K}+1}.$$

%\begin{remark} \label{estimate principal curvatures and qc coefficient hyp}
%The same relation holds in $\Hyp^3$ for $S$ a minimal surface and $\Phi$ is obtained by composing the hyperbolic Gauss maps from the surface to the two connected components of $\partial_\infty \Hyp^3\setminus \partial_\infty S$. Indeed, we have analogue formulae for the pull-back by $\Phi$, where $E\pm JB$ is replaced by $E\pm B$, recall the definition of first fundamental form at infinity in Subsection \ref{infinity}. This gives a quantitative proof of the fact that a minimal surface $S$ with principal curvatures in $[-1+\epsilon,1-\epsilon]$ has boundary at infinity a quasicircle.
%\end{remark}

We are now ready to prove Theorem \ref{minimal lag extension}.

\begin{reptheorem}{minimal lag extension}
There exist universal constants $\delta$ and $C_1$ such that, for any quasisymmetric homeomorphism $\phi$ of $\RP^1$ with cross ratio norm $||\phi||_{cr}<\delta$, the minimal Lagrangian extension $\Phi_{M\!L}:\Hyp^2\rar\Hyp^2$ has maximal dilatation bounded by: $$\ln K(\Phi_{M\!L})\leq C_1||\phi||_{cr}\,.$$
\end{reptheorem}
\begin{proof}
Putting together the inequalities in Proposition \ref{estimate width cross ratio norm}, Theorem \ref{estimate principal curvatures and width ads} and Proposition \ref{estimate principal curvatures and qc coefficient}, we obtain the following inequality:
\begin{equation} \label{equation k phi paperoga}
K(\Phi)\leq\left(\frac{1+C_1\tan w}{1-C_1\tan w}\right)^2
\leq\left(\frac{1+C_1\sinh(\frac{||\phi||_{cr}}{2})}{1-C_1\sinh(\frac{||\phi||_{cr}}{2})}\right)^2\,.
\end{equation}
Clearly, the inequality \eqref{equation k phi paperoga} holds provided $||\phi||_{cr}$ is sufficiently small so that $1-C\sinh(\frac{||\phi||_{cr}}{2})>0$. Since
$$\left.\frac{d}{dx}\right|_{x=0}\ln\left(\frac{1+C_1\sinh(\frac{x}{2})}{1-C_1\sinh(\frac{x}{2})}\right)^2=2C_1\,,$$
the claim is proved, by appropiatly choosing the constant (which depends on the choice of $\delta$, and is still called $C_1$) .
\end{proof}

We now move to the proof of Theorem \ref{minimal lag extension large}.

 \begin{reptheorem}{minimal lag extension large}
There exist universal constants $\Delta$ and $C_2$ such that, for any quasisymmetric homeomorphism $\phi$ of $\RP^1$ with cross ratio norm $||\phi||_{cr}>\Delta$, the minimal Lagrangian extension $\Phi_{M\!L}:\Hyp^2\rar\Hyp^2$ has maximal dilatation bounded by: $$\ln K(\Phi_{M\!L})\leq C_2||\phi||_{cr}\,.$$
\end{reptheorem}
\begin{proof}
By Theorem \ref{theorem width lambda big}, we have
$$\frac{1}{1-||\lambda||_\infty}\leq (\tan w)^{M}$$
and therefore
$$||\lambda||_\infty\leq 1-\frac{1}{(\tan w)^{M}}\,.$$
By Proposition \ref{estimate principal curvatures and qc coefficient}, if $K$ is the maximal dilatation of $\Phi_{M\!L}$, then
\begin{equation} \label{equation k phi gastone}
K^{1/2}=\frac{1+||\lambda||_\infty}{1- ||\lambda||_\infty}\leq \frac{2-\frac{1}{(\tan w)^{M}}}{\frac{1}{(\tan w)^{M}}}=2(\tan w)^{M}-1\,.
\end{equation}
Finally, using Proposition \ref{estimate width cross ratio norm},
$$K^{1/2}\leq 2\left(\sinh{\left(\frac{||\phi||_{cr}}{2}\right)}\right)^{M}-1\,.$$
Therefore one gets
$$\ln K\leq 2\left( (M/2)||\phi||_{cr}+\ln 2\right)\leq C_2||\phi||_{cr}\,,$$
choosing the constant $C_2$ sufficiently large, under the (repeatedly used) assumption that $||\lambda||_\infty$ is larger than a constant $\delta$ (and thus $K$ is larger than some universal constant $K_0>1$).
\end{proof}

Finally, by using the inequalities in Proposition \ref{estimate width cross ratio norm}, Proposition \ref{estimate principal curvatures and width ads reverse} and Proposition \ref{estimate principal curvatures and qc coefficient}, we obtain the following estimate:
$$||\phi||_{cr}\leq 2\ln\left(\frac{(\sqrt{K}+1-\sqrt{2})(\sqrt{K}+1+\sqrt{2})}{(\sqrt{K}-1+\sqrt{2})(1+\sqrt{2}-\sqrt{K})}\right)\,,$$
which clearly holds if the quasiconformal coefficient $K=K(\Pml)$ of the minimal Lagrangian extension $\Pml:\D\rar\D$ is in $[1,(1+\sqrt{2})^2)$.
Let us observe that 
the function 
$$K\mapsto 2 \ln\left( \frac{(\sqrt{K}+1-\sqrt{2})(\sqrt{K}+1+\sqrt{2})}{(\sqrt{K}-1+\sqrt{2})(1+\sqrt{2}-\sqrt{K})}\right)$$
is differentiable with derivative at 0 equal to 2. Hence the following holds:

\begin{reptheorem}{minimal lag extension below}
There exist universal constants $\delta$ and $C_0$ such that, for any quasisymmetric homeomorphism $\phi$ of $\RP^1$ with cross ratio norm $||\phi||_{cr}<\delta$, the minimal Lagrangian extension $\Phi:\Hyp^2\rar\Hyp^2$ has maximal dilatation  bounded by: $$C_0||\phi||_{cr}\leq \ln K(\Phi_{M\!L})\,.$$
The constant $C_0$ can be taken arbitrarily close to $1/2$.
\end{reptheorem}

In particular, any constant $C$ satisfying the statement of Theorem \ref{minimal lag extension} cannot be smaller than $1/2$.

\begin{repcor}{corollary extension}
There exists a universal constant $C$ such that, for any quasisymmetric homeomorphism $\phi$ of $\RP^1$, the minimal Lagrangian extension $\Phi_{M\!L}:\Hyp^2\rar\Hyp^2$ has maximal dilatation $K(\Phi_{M\!L})$ bounded by: $$\ln K(\Phi_{M\!L})\leq C||\phi||_{cr}\,.$$
\end{repcor}
\begin{proof}
In light of Theorem \ref{minimal lag extension} and Theorem \ref{minimal lag extension large}, it will be sufficient to prove that there exists a constant $K_0$ such that $\ln K(\Phi_{M\!L})\leq K_0$ for all quasisymmetric homeomorphisms $\phi$ with $\delta\leq||\phi||_{cr}\leq \Delta$. To prove this, suppose by contradiction that there exists a sequence $\phi_n$ with $\delta\leq||\phi_n||_{cr}\leq \Delta$ such that the corresponding minimal Lagrangian extensions have maximal dilatation $K_n\to\infty$. Therefore one can pick a sequence of points $x_n$ on the maximal surface $S_n$ (where of course $\partial_\infty S_n=gr(\phi_n)$) such that the (positive) principal curvature $\lambda_n(x_n)$ tends to $1$ as $n\to\infty$. 

By composing with isometries of $\AdS^3$, we can assume $x_n$ is a fixed point $x_0$ and all surfaces $S_n$ are tangent to the same plane through $x_0$. Indeed, composing with elements of $\PSL(2,\R)$ does not change the cross-ratio norm of $\phi_n$. By Lemma \ref{lemma bon schl} there exists a subsequence ${n_k}$ such that the maximal surfaces $S_{n_{k}}$ converge $C^\infty$ on compact sets to an entire maximal surface $S_\infty$ of nonpositive curvature. Using Theorem \ref{Compactness property of quasisymm homeo} there exists a further subsequence $n_{k_j}$ such that $\phi_{n_{k_j}}$ converges to a quasisymmetric homeomorphism $\phi_\infty$. Moreover (see also Remark \ref{remark compactness}), $S_\infty$ has asymptotic boundary $\partial_\infty S_\infty=gr(\phi_\infty)$. But by the $C^\infty$ convergence, the principal curvatures of $S_\infty$ at $x_0$ are $1$ and $-1$. By Lemma \ref{lemma flat horo}, $S_\infty$ is a horospherical surfaces and this gives a contradiction.
\end{proof}

%\begin{remark}
%In a completely analogous way, using the inequalities \eqref{equation k phi paperoga} and \eqref{equation k phi gastone}, one can prove that 
%$$\ln K(\Phi_{M\!L})\leq B\tan w$$
%for some universal constant $B$.
%\end{remark}

\clearpage
\bibliographystyle{alpha}
\bibliographystyle{ieeetr}
\bibliography{../bs-bibliography}

\end{document}